\documentclass[12pt, a4paper, parskip=half, abstracton]{scrartcl}

\usepackage{array}
\usepackage{marginnote}
\usepackage{xcolor}
\usepackage{amscd,amssymb,amsfonts,amsmath,latexsym,amsthm}
\usepackage{hyperref}
\usepackage[all,cmtip]{xy}
\usepackage{mathrsfs}
\usepackage{graphicx}
\usepackage{bm} 
\usepackage[nameinlink, capitalise]{cleveref}

\usepackage{etoolbox}

\def\hB{\hspace*{\fill}$\qed$}

\usepackage[nottoc]{tocbibind} 

\usepackage{defs_pp1}
\usepackage{slashed}
\usepackage[utf8]{inputenc}
\usepackage{microtype}
\usepackage[english]{babel}
\usepackage{mathtools}
\usepackage{bm}
\usepackage{esvect}

\title{Branched coarse coverings and transfer maps}
\author{
Ulrich Bunke\thanks{Fakult{\"a}t f{\"u}r Mathematik,
Universit{\"a}t Regensburg,
93040 Regensburg,
ulrich.bunke@mathematik.uni-regensburg.de} 
}

\numberwithin{equation}{section}
\newtheorem{theorem}{Theorem}[section] 
\newtheorem{prop}[theorem]{Proposition}
\newtheorem{lem}[theorem]{Lemma}

\newtheorem{ddd}[theorem]{Definition}
\newtheorem{kor}[theorem]{Corollary}
\newtheorem{ass}[theorem]{Assumption}

\newtheorem{construction}[theorem]{Construction}

\theoremstyle{remark}
\theoremstyle{definition}

\newtheorem{ex}[theorem]{Example}
\newtheorem{rem}[theorem]{Remark}
\newtheorem{constr}[theorem]{Construction}

\newcommand{\trc}{\mathrm{trc}}
\newcommand{\leb}{\mathrm{leb}}

\newcommand{\Ring}{\mathbf{Ring}}
\newcommand{\nRing}{\mathbf{Ring}^{\mathrm{nu}}}

\newcommand{\fin}{\mathrm{fin}}
\newcommand{\fdsc}{\mathrm{fdsc}}
\newcommand{\EE}{\mathrm{E}}

\newcommand{\adim}{\mathrm{adim}}
\newcommand{\udim}{\mathrm{udim}}
\newcommand{\fadim}{\mathrm{fadim}}

\newcommand{\Ban}{\mathrm{Ban}}
\newcommand{\crs}{\mathrm{crs}}
\newcommand{\ctr}{\mathrm{ctr}}

\newcommand{\mult}{\mathrm{mult}}

\newcommand{\free}{\mathrm{free}}
\newcommand{\per}{\mathrm{per}}

\newcommand{\fg}{\mathrm{fg}}
\newcommand{\topp}{\mathrm{top}}
\newcommand{\ee}{\mathrm{e}}

\newcommand{\alg}{\mathrm{alg}}
\newcommand{\exa}{\mathrm{ex}}

\newcommand{\nCalg}{C^{*}\mathbf{Alg}^{\mathrm{nu}}}

\newcommand{\BCov}{\mathbf{BrCov}}

\newcommand{\UBC}{\mathbf{UBC}}

\newcommand{\Yo}{\mathrm{Yo}}

\newcommand{\Res}{\mathrm{Res}}

\newcommand{\Hilb}{\mathbf{Hilb}}
\newcommand{\cR}{\mathcal{R}}

\newcommand{\BC}{\mathbf{BC}}

\newcommand{\Cofib}{\mathrm{Cofib}}
\newcommand{\bB}{{\mathbf{B}}}
\newcommand{\Fib}{{\mathrm{Fib}}}

\newcommand{\incl}{\mathrm{incl}}

\newcommand{\cP}{\mathcal{P}}

\newcommand{\UK}{\mathrm{UK}}

\newcommand{\cZ}{{\mathcal{Z}}}

\newcommand{\cL}{{\mathcal{L}}}
\newcommand{\cW}{{\mathcal{W}}}

\newcommand{\Add}{{\mathtt{Add}}}

\newcommand{\bA}{{\mathbf{A}}}

\newcommand{\const}{{\mathtt{const}}}

\newcommand{\Alg}{{\mathbf{Alg}}}
\newcommand{\nAlg}{\mathbf{Alg}^{\mathrm{nu}}}

\newcommand{\cO}{{\mathcal{O}}}
\newcommand{\cU}{{\mathcal{U}}}
\newcommand{\cY}{{\mathcal{Y}}}

 \newcommand{\Cat}{{\mathbf{Cat}}}

\DeclareMathOperator{\proj}{proj}

\newcommand{\lf}{\mathrm{lf}}

\renewcommand{\Dirac}{\slashed{D}}

\newcommand{\Ccat}{{C^{\ast}\mathbf{Cat}}}
\newcommand{\Calg}{{\mathbf{C}^{\ast}\mathbf{Alg}}}

\renewcommand{\Add}{\mathbf{Add}}

\newcommand{\bd}{\mathrm{bd}}
\newcommand{\op}{\mathrm{op}}

\newcommand{\perf}{\mathrm{perf}}

\newcommand{\add}{\mathrm{add}}

\newcommand{\cone}{\mathrm{cone}}

\newcommand{\nCcat}{C^{*}\mathbf{Cat}^{\mathrm{nu}}}

\renewcommand{\ind}{\mathrm{ind}}
\newcommand{\UCov}{\mathbf{UCov}}
\renewcommand{\tr}{\mathrm{tr}}
\renewcommand{\id}{\mathrm{id}}

\begin{document}
 \setcounter{tocdepth}{1}
\maketitle 
 
\newcommand{\width}{\mathrm{width}}

\begin{abstract}
We introduce the concepts of branched coarse coverings and transfers between coarse homology theories along them.
We show that various versions of coarse $K$-homology theories 
admit the additional structure of transfers. We show versions of Atiyah's $L^{2}$-index theorem in coarse homotopy theory and apply them to give a new argument for the corresponding step in  Higson's counterexample to the coarse Baum-Connes  conjecture.
\end{abstract}
 \tableofcontents

\section{Introduction}

This paper is about transfers between coarse homology  theories for coarse branched coverings.

 In order to motivate our definitions we first recall the notion of a transfer for homology theories in ordinary topology and then explain its generalization to coarse geometry.
Assume that $X$ is a topological space with an action of a group $G$, that 
$Y$ is a topological space with trivial $G$-action, and  that $f:X\to Y$ is an equivariant map which is a branched $G$-covering with branching locus $Z\subseteq Y$.  By definition this means that  for any neighbourhood $V(Z)$ of $Z$ the restriction
$f_{|X\setminus f^{-1}(V(Z))}:X\setminus f^{-1}(V(Z))\to Y\setminus V(Z)$ is a $G$-covering.  
Let  $(E^{G}, E)$ be a pair of  an equivariant homology theory and   its non-equivariant component  related by  natural equivalences    \begin{equation}\label{dfvsdfvsdfvsdfvdsdfv}  E(W/G) \stackrel{\simeq}{\to}E^{G}(W)
\end{equation}for spaces $W$ with free and proper $G$-action.
Then using excision
get  an equivalence \begin{equation}\label{hrthrtggertge}\tr_{f}:E(Y,V(Z))\stackrel{\simeq}{\to} E^{G}(X,f^{-1}(V(Z)))\end{equation} called the topological transfer for branched coverings.

The main goal of 
 the present paper is to  generalize this observation to coarse geometry.  We assume that $X$ is a $G$-bornological coarse space and $Y$ is a bornological coarse space \cite{equicoarse}. Recall that coverings in topology can be characterized by the unique path lifting property. In coarse geometry for  a coarse scale  determined by an entourage $U$  we can consider discrete paths with step width bounded by $U$ which we will call $U$-paths.  
We say that $f$ is a branched coarse $G$-covering if for every coarse scale $U$ there exists a sufficiently large coarse neighbourhood $V_{U}(Z)$ such that $f_{|X\setminus f^{-1}(V_{U}(Z))}:X\setminus f^{-1}(V_{U}(Z))\to Y\setminus V_{U}(Z)$
admits a unique lifting of $U$-paths. In general,   if we increase the step width $U$  we also must also increase the neighbourhood $V_{U}(Z)$ to be taken out. We let $\cZ:=(V_{U}(Z))_{U\in \cC_{Y}}$ denote the big family of subsets $V_{U}(Z)$. We will call $(f:X\to Y,\cZ)$  a branched coarse $G$-covering relative to $\cZ$, see \cref{wroigjowegewrgerfweferf} and \cref{kogpwererwfwerfwerf}.\ref{otzkjprtzjr9} for a more precise formulation.

In topology we have many interesting $G$-coverings which are not branched at all, i.e., with $Z=\emptyset$.
In coarse geometry branched coarse $G$-coverings with empty branching locus $\cZ=(\emptyset)$
are trivial by \cref{ertiohertherthreh}, i.e. the map of underlying $G$-coarse spaces is  the projection $X\cong G_{min}\otimes Y\to Y$. They are examples of the bounded coarse coverings considered in \cite{trans}. Even if $f:X\to Y$ is a connected $G$-covering of proper path metric spaces, then  the
induced map of $G$-bornological coarse spaces is only a branched coarse $G$-covering, see \cref{wrtohue9rheheththe}.
Its  branching locus is the big family $\cZ:=(Z_{r})_{r\in (0,\infty)}$, where $Z_{r}$ is the set of points in $Y$ whose  fibre $f^{-1}(y)$ is not $r$-discrete in the sense that it contains pairs of points $x,x'$ with $x\not=x'$ and $d(x,x')\le r$.

In  \cref{wegjiowgwergrwewf} we will introduce the notion of a pair  $(E^{G},E)$ of
an equivariant coarse homology theory and a coarse homology theory 
related by a transfer  transformation (equivalence) for coarse branched $G$-coverings, a natural transformation (equivalence)  with components 
$$\tr_{f}:E(Y, \cZ)\stackrel{(\simeq)}{\to} E^{G}(X,f^{-1}(\cZ))$$
for all branched coarse $G$-coverings. In    \cref{kgopertgtergebdbfgbdfgb},  \cref{kogpegregwerf} and   \cref{lkhperthtgretg}
we show that various coarse homology theories of $K$-theoretic nature admit this additional structure of transfers.
 In all   these examples the value of the homology theory on a $G$-bornological coarse space  $X$ is  given by applying a $K$-theory functor  to a   category of $X$-controlled objects in some coefficient category. The transfers will be actually constructed on the level of these controlled objects categories. Examples are the coarse algebraic $K$-homology $K^{\alg}\cX^{G}_{\bA}$ with coefficients in an additive category  $\bA$ considered in \cref{kgopertgtergebdbfgbdfgb} and the coarse topological $K$-homology $K\cX_{\bC}^{G}$ with coefficients in a $C^{*}$-category $\bC$ discussed in \cref{lkhperthtgretg}.  As an intermediate coarse homology theory between algebraic and topological coarse $K$-homology we will introduce the
  uncompleted topological coarse $K$-homology $H\cX_{\bC}^{G,\ctr}$ in \cref{kogpegregwerf}.
  Depending on the example, the transfers are only defined  for 
  subclasses of branched coarse coverings satisfying certain finite asymptotic dimension conditions. 
 
 The formalism of transfers developed here extends the transfers for bounded coarse coverings 
 introduced in \cite{coarsetrans},  \cite{coarsek}. The main motivation to consider transfers in these papers were  applications to injectivity results for assembly maps \cite{desc}. 

In the context of  coarse homotopy theory  based on the category  $G\BC$ of $G$-bornological coarse spaces   
the natural category for ordinary homotopy theory is not  the category $G$-topological spaces but rather  the category of $G$-uniform bornological coarse spaces $G\UBC$, see \cite{equicoarse} for definitions. These categories are  related by the cone functor $\cO^{\infty}:G\UBC\to G\BC$ (described in \cref{hojprthertgertgerg}) and the forgetful functor $\iota:G\UBC\to G\BC$. Equivariant (strong) coarse homology theories
are functors $E^{G}:G\BC\to \cC$ to a cocomplete stable target category which are coarsely invariant, excisive,
$u$-continuous and annihilate (weakly) flasques \cite[Def. 3.10 \& 4.19]{equicoarse}.
Similarly, local homology theories are functors $F^{G}:G\UBC\to
\cC$ to a cocomplete stable target category which are  homotopy invariant, excisive and $u$-continuous and vanish on flasques (the obvious equivariant version of \cite[Def. 3.12]{ass}).

By restriction along  the cone  functor
  strong coarse homology theories $E^{G}:G\BC\to \cC$ induce local  homology theories $\Sigma^{-1}E^{G}\cO^{\infty}:G\UBC\to \cC$ \cite[Lem. 9.6]{ass}.
  In this way,  following \cite[Sec. 3.3]{werfwerfwrefw}, the coarse topological $K$-homology $K\cX^{G}:G\BC\to \Sp$ constructed in \cite{coarsek} provides 
 an interesting  model for topological locally finite $K$-homology $\Sigma^{-1}K\cX^{G}\cO^{\infty}:G\UBC\to \Sp$, called the  local $K$-homology. This is an alternative to constructions in terms  of elliptic operators \cite{Baum_1982}, \cite{kasparovinvent} or the Baum-Higson-Schick bordism model \cite{zbMATH05901984}, \cite{Baum_2007}.  A major  motivation for the present paper
is to provide a proof of a version of Atiyah's $L^{2}$-index theorem for coverings in the framework of  coarse topological
$K$-homology. In contrast to all other proofs known to the author \cite{zbMATH03505915}, 
 \cite{zbMATH02028608}, \cite{zbMATH02246450}, \cite[Thm. 6.4]{Willett_2012I}
our proof does not involve differential operators. 

In the classical proofs differential operators are crucially used as very local representatives
 of $K$-homology classes. Differential operators can be lifted  along $G$-coverings  to $G$-equivariant differential operators  in an obvious manner.   In this way one obtains lifts of $K$-homology classes in the classical literature.
In our approach we use coarse homotopy theory and the comparison
 between coarse algebraic and coarse topological $K$-homology in order squeeze local $K$-homology classes
 so that they  can be lifted along coverings.

Classically lifts of $K$-homology classes along coverings were constructed on the cycle level. An advantage of the approach   of the present paper  is that  we consider this lifting as a natural transfer transformation between spectrum-valued coarse homology theories.  As an important consequence we get an automatic compatibility of the transfers with Mayer-Vietoris boundary maps.

  
As said above  our $L^{2}$-index theorem does not involve differential operators. It is rather a statement in the realm of the  abstract coarse  index theory whose basic features we recall in the following. 
%
Let $E^{G}:G\BC\to \cC$ be any strong equivariant coarse homology theory and  $\Sigma^{-1}E^{G}\cO^{\infty}:G\UBC\to \cC$ be its 
 associated equivariant local homology theory. These two functors are related  by the cone boundary $\partial^{\cone}:\Sigma^{-1}E^{G}\cO^{\infty}\to E^{G}\iota:G\UBC\to \cC$ \cite[(9.1)]{ass}.
We interpret classes in $\Sigma^{-1}E^{G}\cO^{\infty}(X)$ as an abstraction of (symbol) classes of $G$-invariant differential operators on $X$, and $E^{G}(\iota X)$ as the home of their index classes.
The cone boundary then sends the symbol class of the  (abstract) differential operator to its (abstract) index.
 
This point of view is justified in  \cite{indexclass}. In the case  of equivariant coarse topological $K$-homology  $E^{G} =K\cX^{G}$  an invariant  Dirac operator $\Dirac$ of degree $n$ on a  complete Riemannian $G$-manifold $X$ naturally gives rise to a local $K$-homology class  $\sigma(\Dirac)$ in $K\cX_{n+1}^{G}\cO^{\infty}(X)$ such that $\partial_{X}^{\cone}(\sigma(\Dirac)) $ in $K\cX_{n}^{G}(\iota X)$ is indeed its  coarse index class $\ind(\Dirac)$.

By \cref{wtiogwtgerwferferfw} the cone functor $\cO^{\infty}$ sends uniform $G$-coverings $f:X\to Y$ to branched coarse coverings
 $(\cO^{\infty}(f):\cO^{\infty}(X)\to \cO^{\infty}(Y),\cO^{-}(Y))$.
  Assume $(E^{G},E)$ is related   by a natural transfer (equivalence) for branched coarse $G$-coverings in the sense of \cref{wegjiowgwergrwewf}.
Then by \cref{kophokhotrpherthetrhee} we have an induced cone transfer (equivalence)
$$\tr_{\cO^{\infty}(f)}:E(\cO^{\infty}(Y))\stackrel{(\simeq)}{\to} E^{G}(\cO^{\infty}(X))$$
for any uniform $G$-covering 
 $X\to Y$. This is the    analogue of  \eqref{dfvsdfvsdfvsdfvdsdfv} for the local homology theories induced from $E^{G}$ and $E$. 
 
 We now describe the content of our $L^{2}$-index theorems  in some detail. We consider the algebraic and the topological case at once.
 The coefficient categories $\bA$ are $k$-linear additive categories for some commutative ring $k$ or $C^{*}$-categories and then $k=\C$.
 Recall that the coarse homology theories $H\cX^{G}_{\bA}$ in the present paper are constructed by  composing a version of the $K$-theory functor $H$ (e.g. $K^{\alg}$ or $K^{\topp}$ depending on the context)  with the functor which associates to a $G$-bornological coarse space $X$  the category  $\bV_{\bA}^{G}(X)$ of $X$-controlled objects in the coefficient category $\bA$.
In   \cref{kopherhetreg} or \cref{okgpergrgwerf}   we  introduce the concept of a trace on  $k$-linear additive categories or $C^{*}$-categories. In this introduction we assume for simplicity that $X$ is $G$-bounded, coarsely connected, and that $G$ acts with finite stabilizers.  Then  in    \cref{hoertighprtgretgerg} or  \cref{jkopwergergwrefwre}   we explain how a trace $\tau$ on $\bA$ determines a   trace $\tau^{G}_{X}$ on $\bV_{\bA}^{G}(X)$.
In \cref{erthgerthertertgertt} we introduce the concept of a tracing of $H$ as a transformation
which associates to each pair $(\bB,\kappa)$ of a $k$-linear category or a $C^{*}$-category  with a trace a homomorphism $\kappa^{H}:\pi_{0}H(\bB)\to  k$.  If $H$ is traced, then we get
an induced homomorphism
$$\tau^{H,G}_{X}:\pi_{0}H \cX^{G}_{\bA}(X)\to k\ .$$

  Our 
 $L^{2}$-index theorems \cref{koprhererge} or \cref{igoewrifoperfrefw} are then   statements  about an equality  of  functions
 \begin{equation}\label{iuhugiregesrg}\tau^{H,G}_{X} \circ \partial_{X}^{\cone} \circ \tr_{\cO^{\infty}(f)}=\tau^{H}_{Y} \circ\partial_{Y}^{\cone} :\pi_{1}H\cX_{\bA}\cO^{\infty}(Y)\to k\ .
\end{equation}

If we specialize $H\cX_{\bA}^{G}$ to the equivariant coarse topological $K$-homology $K\cX^{G}$ with coefficients in the $C^{*}$-category $\Hilb_{c}(\C)$ with its canonical trace (then $H=K^{\topp}$), then $\tau^{K^{\topp}}_{Y}:\pi_{0}K\cX(Y)\to \C$ is the usual dimension homomorphism and 
$\tau^{K^{\topp},G}_{X}:\pi_{0}K\cX^{G}(X)\to \C$ is what is usually called the $G$-dimension or von-Neumann dimension.
Let $Y$ be a closed Riemannian manifold with a generalized Dirac operator $ \Dirac_{Y}$.
If $X\to Y$ is a $G$-covering and  $ \Dirac_{X}$ is the lift of $ \Dirac_{Y}$ to a $G$-invariant operator on $X$, then
we have the equality \begin{equation}\label{wrfrefwerfwerfwr}\tr_{\cO^{\infty}(f)}(\sigma( \Dirac_{Y}))=\sigma(\Dirac_{X})\ .
\end{equation}
Then from  \eqref{iuhugiregesrg} 
 we get the equality  $$\tau^{G,K^{\topp}}_{X} \ind(\Dirac_{X}) =\tau^{K^{\topp}}_{Y} \ind( \Dirac_{Y}) \ .$$  
 This statement 
precisely recovers the  classical formulation of Atiyah's $L^{2}$-index theorem.
In this way the paper provides a  completely independent proof of this classical result. 
But as we do not want to talk about analysis to much we will not expand the details of the verification of \eqref{wrfrefwerfwerfwr}, see also \cref{ojperherthetrg}.

In the remainder of this introduction we indicate  the contents of the sections.
 In \cref{jkgopwerferwfwerf} we introduce the concept of a branched coarse covering.
In \cref{klgfrewegrrefgwerf} we introduce uniform coverings and show how this notion interacts with 
forming cones, squeezing spaces and Rips complexes.
In \cref{kohgprthtrgergtreg} we consider various notions of finite asymptotic dimension and finite uniform topological dimension. We further analyse the compatibility of these notions with the cone construction.
In \cref{wgererwf9} we introduce the formalism of transfers for coarse homology theories in an axiomatic way.

In the next three   \cref{kgopertgtergebdbfgbdfgb},  \cref{kogpegregwerf} and   \cref{lkhperthtgretg}
we show that our examples of coarse $K$-homology theories admit the structure of transfers.

 In \cref{okvpwekvpewvdf} we show how a trace on a coefficient $C^{*}$-category gives rise to traces on the Roe categories of controlled objects. In \cref{gokpwegregwre9} we show that algebraic $K$-homology with coefficients
 in an additive category with a trace admits a version of the $L^{2}$-index theorem.
 In \cref{ohkjptwehgrtherge} we show the $L^{2}$-index theorem for topological $K$-homology with coefficients in a $C^{*}$-category with a trace.
 Finally in \cref{lkthperthergetrgtre} we present Higson's counterexample to the coarse Baum-Connes conjecture.

\cref{jiogwerfwrefrefw} is devoted to the construction a comparison map between versions of  the algebraic $K$-theory 
  of $C^{*}$-categories  considered as $\Z$-linear  categories and the topological $K$-theory of $C^{*}$-categories.
  The main result is \cref{oiw0rgwrefwerfwrefw} which states that there exists a  trace-preserving  approximation of  topological $K$-theory  by the $\cL^{1}$-twised  algebraic homotopy $K$-theory. 
The existence of such an approximation is a crucial ingredient in our proof of the $L^{2}$-index theorem  \cref{igoewrifoperfrefw}.  
  

This paper represents a part of a lecture course on the coarse Baum-Connes conjecture  taught in  fall 2024, Regensburg
 (see also \cite{Bunke:2024aa} for another part).

{\em Acknowledgement:  The author was supported by the SFB 1085 (Higher Invariants) funded by the Deutsche Forschungsgemeinschaft (DFG).  He  thanks M. Ludewig for allowing him to  freely use the material on branched coarse coverings which was initially  part  of a  joined project. 
He further thanks  A. Engel for motivating discussion and suggesting that
the uncompleted topological $K$-homology might be an interesting object to consider.}

\section{Branched coarse coverings}\label{jkgopwerferwfwerf}

 
 
 In this section we introduce the concept of a branched coarse covering.
We work in the category $\BC$ of  bornological coarse spaces  introduced in \cite{buen}. 
In order to state the basic definition we first introduce some notation. The complement 
of a subset $Z$ of a set $Y$ will be denoted by   $Z^{c}:=Y\setminus Z$.
 The right-restriction of an entourage $V$ on $Y$ to a subset $A$  will be denoted by $V_{A}:=V\cap (Y\times A)$. The preimage of the entourage under a map $f:X\to Y$ will be written  as $f^{-1}(V):=(f\times f)^{-1}(V)$.

We now consider two bornological coarse spaces $X$ and $Y$ and a map
$f:X\to Y$  between the underlying sets. Furthermore, we  consider a  big family
  $\cZ$    on $Y$.  
   
 \begin{ddd}\label{wroigjowegewrgerfweferf}
The pair  $(f:X\to Y,\cZ)$ is called a   branched coarse covering relative to $\cZ$  if it has the following properties:
 \begin{enumerate}
 \item \label{etrherthtrherthrth} $f$ is controlled.
 \item\label{etrherthtrherthrth1} $f$ is  bornological and has locally finite fibres.
 \item\label{rgegerwerefwref} There exists an entourage $P$ of $X$  (called a coarse connection) such that: 
\begin{enumerate}  \item\label{rgegerwerefwrefa}  For every coarse entourage $V$ of $Y$  there exists member $Z$ of $\cZ$
such that: \begin{enumerate} \item  \label{wetgrgffsgvdfgsfdgsdfg} for every
 $(y',y)$ in $V_{Z^{c}}$  and $x$ in $f^{-1}(y)$ there exists a unique 
$x'$ in $ f^{-1}(y')$ such that $(x',x)\in P$.  \item \label{wrgtuiwhigwergegr}The subset
$P\cap f^{-1}(V_{Z^{c}})$ is a coarse entourage of $X$.
\end{enumerate}
 \item \label{rgegerwerefwrefb}For every coarse entourage $U$ of $X$ there exists a member $Z$ in $\cZ$  such that
$U_{f^{-1}(Z^{c})}\subseteq P$.
 \end{enumerate}
  \end{enumerate}
 \end{ddd}

Note that $P$ is not necessarily a coarse entourage of $X$,
and that $f$ is in general not a morphism in $G\BC$
 since it is not proper.
 
 \begin{rem}\label{wrotgwgregfefwref} In the situation of 
   \cref{wroigjowegewrgerfweferf}.\ref{wetgrgffsgvdfgsfdgsdfg}   we say that the parallel transport at scale $V$ is defined on $ Z^{c}$. 
 This notion is motivated as follows. 
 A discrete $V$-controlled path  
 in $Z^{c}$ is a finite sequence $\gamma:=(y_{0},\dots,y_{n})$ of points in $Z^{c} $   such that
 $(y_{i},y_{i+1})\in V$ for all $i$ in $\{0,\dots,n-1\}$. 
 For every  
  $(y',y)$ in $V_{Z^{c}}$ we have a canonical bijection $\Phi_{y',y}:f^{-1}(y)\to f^{-1}(y')$ determined by the condition that $P\cap (f^{-1}(y')\times f^{-1}(y))$
 is the graph of $\Phi_{y',y}$. 
  The composition \begin{equation}\label{qewfqewfewfwefwedq}
\Phi_{\gamma}:\Phi_{y_{n},y_{n-1}}\circ\dots\circ \Phi_{y_{1},y_{0}}:f^{-1}(y_{0})\to f^{-1}(y_{n})
\end{equation}
 will be considered as a parallel transport along the path $\gamma$. \hB
  \end{rem}

 The following lemma shows that the  connection of a branched coarse covering is essentially unique.
 Let $f:X\to Y$ be a branched coarse covering relative to $\cZ$.
 \begin{lem}\label{wregegerfrwferf}
 If $P,P'$ are two connections for $f$ and $\cZ$, then for every coarse entourage $V$ of $Y$   there exists a member $Z$ of $\cZ$
 such that
 $ P\cap f^{-1}(V_{Z^{c}})=P'\cap f^{-1}(V_{Z^{c}})$.
 \end{lem}
\begin{proof}
It suffices to show that  for every coarse entourage $V$ of $Y$   there exists a sufficiently large member $Z$ of $\cZ$ such that $P\cap f^{-1}(V_{Z})\subseteq P'$.
 
 We first  observe that  the conditions in 
 \cref{wroigjowegewrgerfweferf}.\ref{rgegerwerefwrefa} and
    \cref{wroigjowegewrgerfweferf}.\ref{rgegerwerefwrefb}  are preserved  under enlarging the member $Z$. 
Let $V$ be a coarse entourage  of $Y$   and let $Z$ in $\cZ$ be such that    \cref{wroigjowegewrgerfweferf}.\ref{rgegerwerefwref} is satisfied for $P$.
Then by  
 \cref{wroigjowegewrgerfweferf}.\ref{wrgtuiwhigwergegr} the set $ P\cap f^{-1}(V_{Z})$ is a coarse entourage of $X$.  
After possibly enlarging $Z$, by  \cref{wroigjowegewrgerfweferf}.\ref{rgegerwerefwrefb} for $P'$ we have $P\cap f^{-1}(V_{Z})\subseteq P'$.
 \end{proof}

It follows from   \cref{wregegerfrwferf} that for every every coarse entourage $V$ of $Y$  there exists a member $Z$ in $\cZ$ such that the parallel transports for 
$V$-controlled paths in $Z^{c}$ defined by $P$ and $P'$ coincide.

For a coarse space we let $\pi^{\crs}_{0}(X)$ denote the set of coarse components of $X$. 
Let $f:X\to Y$ be a  morphism of bornological coarse spaces
and recall the notions of a coarse covering from   \cite[Def. 3.3.16]{unik}  and of a bounded coarse covering from \cite{trans}, \cite[Def. 3.4.16]{unik}.
 \begin{lem}\label{ertiohertherthreh}
 If  $f$ is a coarse branched covering  relative to the  big family $(\emptyset)$, then
 it is a coarse covering.
 In particular,  the underlying map  of coarse spaces   restricted to any coarse component $Y_{i}$ of $Y$   
 is  isomorphic to the projection  $I_{min}\otimes Y_{i}\to Y_{i}$ for some set $I$.
 
If in addition for every bounded set $B$ of $X$   the map  $\pi^{\crs}_{0}(B)\to \pi^{\crs}_{0}(f(B))$ has uniformly finite fibres, then $f$ is a bounded coarse covering  
\end{lem}
\begin{proof}
Let $f_{*}:\pi_{0}^{\crs}(X)\to \pi_{0}^{\crs}(Y)$ denote the induced map between the sets of coarse components.
We first verify the condition \cite[Def. 3.4.16(1)]{unik} requiring  that the restriction of $f$ to any coarse component $X_{i}$ of $X$ is an isomorphism from $X_{i}$ to the coarse component $Y_{f_{*}i}$. 
Let $P$ be a connection for $f$.
 It is clear that $f(X_{i})$ is contained in $Y_{f_{*}i}$. Let $x$ be in $X_{i}$ and set $y:=f(x)$ in $Y_{f_{*}i}$. If $y'$ is any other point in $Y_{i}$, then there exists an entourage $V$ of $Y$ such that
$(y',y)\in V$.  Then  $f^{-1}(V)\cap P$ is by \cref{wroigjowegewrgerfweferf}.\ref{rgegerwerefwrefa} a coarse entourage of $X$ and we have the parallel transport
 $x'$ in $f^{-1}(y')$ of $x$ at scale $V$. Since $X_{i}$ is coarsely connected we have  $x'\in X_{i}$ and hence $y'\in f(X_{i})$.  If $x''$ is any point in $f^{-1}(y')\cap X_{i}$, then there exists an entourage $U$ of $X$ such that $(x'',x)\in U$.
 Then $f(U)$ is a coarse entourage of $Y$ and $x''$ is the parallel transport of $x$ at scale $f(U)$. Both, $x'$ and $x''$ are then  parallel transports of $x$ at scale $f(U)\cup V$ and hence equal. 
 We conclude that 
  the restriction $f_{|X_{i}}:X_{i}\to Y_{f_{*}i}$  bijective.

  For every coarse entourage $V$ of $Y$ we have $V=f(f^{-1}(V)\cap P)$, and for every coarse entourage $U$ of $X$
  we have  $U=f^{-1}(f(U))\cap P$ by \cref{wroigjowegewrgerfweferf}.\ref{rgegerwerefwrefb}. This implies that $f_{|X_{i}}:X_{i}\to Y_{f_{*}i}$ is an isomorphism of coarse spaces. Furthermore,  $P=U(\pi_{0}^{\crs}(X)):=\bigcup_{U\in \cC_{X}}U$  (here $U(\pi_{0}^{\crs}(X))$ is the notation from \cite{unik}) and the coarse structure of $X$ is generated by 
  the entourages of the form $f^{-1}(V)\cap P$ for all coarse entourages $V$ of $Y$.
 
 This finishes the proof of the first assertion, see also \cite[Ex. 3.3.18]{unik} for the second assertion.
 
 For the last assertion, the additional condition  on the bornology
is precisely the restriction on the bornologies of the domains of bounded coarse coverings.
\end{proof}

\begin{ex}
Let $I$ be a set and $Y$ be a bornological coarse space. Then the projection
$f:I_{min,min}\otimes Y\to Y$ is a bounded coarse covering and therefore a branched coarse covering
 relative to the big family $(\emptyset)$. \hB
\end{ex}

 \begin{ex}\label{wrtohue9rheheththe}
 We consider a  proper  path metric space $Y$ and  let
 $X\to Y$ be a topological covering. We  equip $X$    with the induced path metric
 such the projection map is a local isometry. Then $X$ is again a proper metric space.

%

  We consider  $X$
 and $Y$ with the   bornological coarse structures  induced by the metrics.
 For $r$ in $[0,\infty)$ we let $V_{r}$ denote the metric entourage of $Y$ of width $r$.
Similarly  the metric entourages of $X$ will be denoted by $U_{r}$.

For $r$ in $(0,\infty)$ a   metric space is called $r$-separated of any two distinct points   have distance $>r$.
We let $Z_{r}$ be the subset of points  $y$ of $Y$ such that the fibre $f^{-1}(y)$ is not $r$-separated.  \begin{lem} The family $\cZ:=(Z_{r})_{r\in [0,\infty)}$ is a big family.\end{lem} \begin{proof}
We have $V_{s}[Z_{r}]\subseteq Z_{r+2s}$.  \end{proof}
  We call $\cZ$ the canonical big family associated to the covering.
   

 \begin{lem} The map $f:X\to Y$ is a coarse branched covering  relative to the canonical big family $\cZ$.
 \end{lem}
 \begin{proof}
 The map $f$ is $1$-Lipschitz and therefore controlled and bornological.

  The fibres of $f$ are locally finite since they are discrete and  $X$ is proper.

We choose the connection  
 $$P:=\bigcup_{r\in [0,\infty)} (U_{r})_{f^{-1}(Z^{c}_{2r})} \ .$$
   The condition from  \cref{wroigjowegewrgerfweferf}.\ref{rgegerwerefwref}. is satisfied since for the   metric entourage 
 $V_{s}$ of $Y$  we can take the member
 $Z_{2s}$ in  $\cZ$.
 \end{proof} 
 \end{ex}

\begin{ex}
Let $(Y_{n})_{n\in \nat}$ be a sequence of simplicial complexes. We consider $Y_{n}$ as a path-metric space
 with the spherical path metric.  Let $X_{n}\to Y_{n}$ be the universal covering with the induced path metric.
 We form the bounded unions $Y:=\bigsqcup_{n\in \nat}^{\bd}Y_{n}$ and
 $X:=\bigsqcup_{n\in \nat}^{\bd}X_{n}$ in $\BC$  and let $f:X\to Y$ be the canonical map.
 We furthermore consider the big family $\cZ:=(Y_{\le n})_{n\in \nat}$ with $Y_{\le n}:=\bigcup_{m\le n}Y_{m}$.
 Let $\ell(n)$ denote the minimum of the length of non-contractible loops in $Y_{n}$.
 \begin{lem}
 If $\lim_{n\to \infty}\ell(n)=\infty$, then $f:X\to Y$ is a branched coarse covering  relative to the big family $\cZ$.\hB
 \end{lem}
 \end{ex}

 \begin{ex}
 The asymptotically faithful covering sequences of spaces of graphs in \cite[Def. 2.2]{Willett_2012I}
give rise to branched coarse coverings.     \hB
\end{ex}

In the following we introduce the equivariant generalization of the concept of a branched coarse covering.
   We let   $G$ be a group and refer to \cite{equicoarse} for the category $G\BC$ of $G$-bornological coarse spaces.  
  We consider two $G$-bornological coarse spaces $X$ and $Y$ and an equivariant map $f:X\to Y$ between the the underlying sets and  an
 invariant big family $\cZ$.

\begin{ddd}\label{kogpwererwfwerfwerf} \mbox{} \begin{enumerate} \item  $f$ is a $G$-equivariant coarse branched covering relative to $\cZ$ if it is 
a  coarse branched covering relative to $\cZ$  for a $G$-invariant connection. \item \label{otzkjprtzjr9}
If there exists $Z$ in $\cZ$ such that  $G$ acts freely and transitively on the fibres of $f_{|f^{-1}(Z^{c})}$, then
we call $f$ a branched coarse $G$-covering. \end{enumerate}
\end{ddd}

 This condition implies that the parallel transport is also compatible with the $G$-action.
 
 \begin{ex}\label{jiiwogergwerfref}
Let $Y$ be a bornological coarse space and
$I$ be a $G$-set. We consider $I_{min,min}\otimes Y$ as a $G$-bornological coarse space, where $Y$ has the trivial $G$-action.
Then the projection $I_{min,min}\otimes Y\to Y$ is a $G$-equivariant branched coarse  covering 
 relative to the big family $(\emptyset)$.
If $I=G$, then it is a branched coarse $G$-covering.
 \hB
\end{ex}

\begin{ex}\label{jgwpergrfwff}
We consider the subspace $Y:=\{y\in\R^{2}\mid \|y\|\ge 1 \}$ of $\R^{2}$ and let $X\to Y$
be the universal covering with its canonical $\Z$-action. We equip $Y$ with the  path metric structure induced by the euclidean metric on $\R^{2}$.
Since $Y$ is a proper path-metric space, by   \cref{wrtohue9rheheththe}
the map $f:X\to Y$ is a   branched coarse $\Z$-covering  relative   to the big family
of $\cZ:=(B(0,r)\cap Y)_{r\ge 0}$. \hB
\end{ex}

%
%
%
%
%
%
%


\begin{ex}\label{kopgergwerg9}
Let $G$ be a group and $(G_{n})_{n\in \nat}$ be a decreasing family of normal cofinite subgroups such that
$\bigcap_{n\in \nat} G_{n}=\{e\}$.  We form the $G$-set
$$S:=\bigsqcup_{n\in \nat} G/G_{n}$$ and 
  the $G$-bornological coarse space
$$X:= G_{can,min}\otimes S_{min,min}\ .$$
We let $X_{n}$ denote the subset $G\times G/G_{n} $ of $X$.
 
We consider the set $$Y:=G\backslash X$$ with the minimal coarse structure
such that $f:X\to Y$ is controlled, and the minimal compatible  bornology such that the projection $f:X\to Y$ is bornological.
We let $Y_{n}:=f(X_{n})\cong G/G_{n}$ and  define the big family $\cZ:=(Y_{\le n})_{n\in \nat}$ with $Y_{\le n}:=\bigcup_{m\le n}  Y_{n}$.
\begin{lem} If $G$ is countably generated, then the map $f:X\to Y$ is a  branched coarse  $G$-covering relative to the big family $\cZ$.
\end{lem}
\begin{proof}
The map $f$ is controlled and has locally finite fibres. 
The group $G$ acts freely and transitively on the fibres of $f$.

 Since $G$ is countably generated  there exists a cofinal sequence $(U_{n})_{n\in \nat}$  of  symmetric  invariant coarse entourages  of $G_{can}$ such that $G_{n}$ is $U^{2}_{n}$-discrete  for every $n$ in $\nat$. 
Then we define the coarse connection by $$P:=\bigcup_{n\in \nat} U_{n}\times \diag(G/G_{n})\ .$$
 
    The condition in  \cref{wroigjowegewrgerfweferf}.\ref{rgegerwerefwrefb} follows from the cofinality of $(U_{n})_{n\in \nat}$.
   Note that the 
   coarse entourages  $V=f(U)$  for  invariant coarse entourages $U$ of $X$ 
   are cofinal in the coarse structure of $Y$. To this end we  observe that $f(U)\circ f(U')=f(U\circ U')$ for  invariant entourages $U,U'$.

    We fix an invariant coarse entourage  $U$ of $X$ and consider the coarse entourage $V:=f(U)$ of $Y$.
We take $m$ such that $U\cap (X_{m}\times X_{m})\subseteq U_{m}$ and check that
   the condition from  \cref{wroigjowegewrgerfweferf}.\ref{wetgrgffsgvdfgsfdgsdfg} holds  true for the member $Y_{\le m}$ of $\cZ$. Let $n$ be in $\nat$ with $n> m$.
   Let $(g,g'G_{n})$ be in $X_{n}$ and assume that  $(\tilde g,\tilde g'G_{n})$ is such that
   $(g,\tilde g)\in U_{n} $ and $g'G_{n}=\tilde g' G_{n}$ so that
   $([g,g'G_{n}],[\tilde g,\tilde g'G_{n}])\in f(U)\cap (Y_{n}\times Y_{n})$.  We can assume that $ g'=\tilde g'$.
   Let $h$ be in $G$ such that $((g,g'G_{n}),(h\tilde g,h  g'G_{n}))\in P\cap (X_{n}\times X_{n})$.  We must show that this implies $h=e$.
    First of all  $h g'G_{n}=g' G_{n}$ implies that $h\in G_{n}$  by normality
   of $G_{n}$. Furthermore, by the symmetry and $G$-invariance of    $U_{n}$ the relations $(g,\tilde g)\in U_{n}$ and $(g,h\tilde g)\in U_{n}$  imply $(e,\tilde g^{-1}h\tilde g)\in U^{2}_{n}$. Since $\tilde g^{-1}h\tilde g\in G_{n}$
   and $G_{n}$ is $U_{n}^{2}$-discrete we can conclude that $\tilde g^{-1}h\tilde g=e$, hence $h=e$.

   In order to check the condition in \cref{wroigjowegewrgerfweferf}.\ref{wrgtuiwhigwergegr} we observe that $U_{f^{-1}(Z_{n}^{c})}\cap P=U_{f^{-1}(Z_{n}^{c})}$.
  \end{proof}
\end{ex}

In the following we discuss the base-change of branched coarse coverings.
For a bornological coarse space $X$ we let $\cC_{X}$ and $\cB_{X}$ denote the coarse structure and the bornology of $X$.
Let $f:X\to Y$ be a branched coarse covering relative to $\cZ$.
 
%
%

Let $h:Y'\to Y$ be a morphism of bornological coarse spaces.
We form the pull-back \begin{equation}\label{qewfwefwfqfwfq}
\xymatrix{X'\ar[r]^{g}\ar[d]^{f'}&X\ar[d]^{f}\\Y'\ar[r]^{h}&Y}
\end{equation}
 
in coarse spaces. We furthermore equip $X'$ with the bornology 
generated by $g^{-1}(\cB_{X})$.
We consider the big family $\cZ':=h^{-1}(\cZ)$ on $Y'$.

\begin{lem}\label{wejigowegfrefrefwe}
$f':X'\to Y'$ is a branched coarse covering relative to $\cZ'$.
\end{lem}
\begin{proof}
First of all $f'$ is controlled by definition. Furthermore, the bornology on $X'$ is compatible with the coarse structure.

We next show that  $f'$ has locally finite fibres.
Let $y'$ be in $Y'$  and assume that  $B$ in $X$ is bounded. Then  
$$g^{-1}(B)\cap f^{\prime,-1}(y')\subseteq   \{y'\}  \times (f^{-1}(h(y))\cap B)$$   is indeed  finite.
 
%

Since $f'(g^{-1}(B))\subseteq h^{-1}(f(B))$ we also see that $f'$ is bornological. 
 
 Let $P$ be  the connection for $f$. Then we set $P':=g^{-1}(P)$.
 
 We now check that  $P'$ is a  connection for $f'$.
Let $V'$ be a coarse entourage   of $Y'$. 
Then $V=h(V)$ is a coarse entourage of $Y$ and we can choose the member
$Z$ of $\cZ$ for $V$ such that the  \cref{wroigjowegewrgerfweferf}.\ref{rgegerwerefwref}
is satisfied with $P$. We then claim that this condition is satisfied for
$V'$ and $Z':=f^{-1}(Z)$ with  $P'$.

We start with    \cref{wroigjowegewrgerfweferf}.\ref{wetgrgffsgvdfgsfdgsdfg}.
Let $(y_{1}',y_{0}')$ be in $V'_{Z^{\prime,c}}$ and $x'_{0}$ be in $f^{\prime,-1}(y_{0}')$.
 We set $y_{i}:=h(y_{i}') $  for $i=0,1$ and $x_{0}:= g(x_{0}') $. Then there exists a unique
 $x_{1}$ in $f^{-1}(y_{1})$ such that $(x_{1},x_{0})\in P$. 
Since the underlying square of sets in \eqref{qewfwefwfqfwfq} is cartesian
  there exists a  unique 
 $x_{1}'$ in $X'$ such that $f'(x_{1}')=y_{1}'$ and $g(x_{1}')=x_{1}$.
 By definition of $P'$ we have $(x_{1}',x_{0}')\in P'$.
 For uniqeness, assume that 
  $x_{1}''$ is in $f^{\prime,-1}(y_{1}')$ and $(x_{1}'',x_{0}')\in P'$. Then
 $(g(x_{1}''),g(x_{0}'))\in P$, $g(x_{1}'')\in f^{-1}(y_{1})$ and hence
 $g(x_{1}'')=g(x_{1}')$. Since also $f'(x_{1}'')=f'(x_{1}')$ we conclude that  $x_{1}'=x_{1}''$.
 
We now verify   \cref{wroigjowegewrgerfweferf}.\ref{wrgtuiwhigwergegr}.
Note that $f'(f^{\prime,-1}(V'_{Z^{\prime,c}})\cap P')\subseteq V'$ and
$g(f^{\prime,-1}(V'_{Z^{\prime,c}})\cap P')\subseteq f^{-1}(V_{Z})\cap P$.
Since $V'$ in $\cC_{Y'}$ and $ f^{-1}(V_{Z})\cap P\in \cC_{X}$
by the choice of $Z$ we conclude that 
that $f^{\prime,-1}(V'_{Z^{\prime,c}})\cap P'\in \cC_{X'}$ since  the  square in
 \eqref{qewfwefwfqfwfq} is cartesian in coarse spaces.

  We finally verify  \cref{wroigjowegewrgerfweferf}.\ref{rgegerwerefwrefb}. 
  Let $U'$ be in $\cC_{X'}$. Then there exists a member $Z$ of $\cZ$ such that
  $g(U')_{f^{-1}(Z)}\in P$. Applying $g^{-1}$ we conclude that 
  $U'_{f^{\prime,-1}(Z')}\subseteq P'$ if we set $Z':=h^{-1}(Z)$.
   \end{proof}

\begin{rem}
For a pair $(X,\cY)$ of a bornological coarse space and a big family 
  the coarse corona $\partial_{\cY} X$ was  introduced in
  \cite{werfwerfwrefw}  as the Gelfand dual of the quotient of the algebra of bounded functions on $X$ whose variation
 vanishes away from $\cY$ by the ideal generated by functions supported on $\cY$.
 
We now consider a coarse branched covering $f:X\to Y$ relative to the family $\cZ$. We then get an induced map
$$\partial f:\partial_{f^{-1}(\cZ)}X\to \partial_{\cZ}Y$$
of coronas. \hB
%
\end{rem}

\section{Uniform coverings, cones and Rips complexes}\label{klgfrewegrrefgwerf}

We work in the context of uniform bornological coarse spaces $\UBC$ introduced in 
\cite{buen}.

For a uniform space $X$ we let $\cU_{X}$ denote the  poset of uniform entourages of $X$ with the opposite of the inclusion relation.
Let $f:X\to Y$ be a map between the underlying sets of 
uniform bornological coarse spaces.
 
\begin{ddd}\label{thetgtrggegtr}
The map $f:X\to Y$ is called a uniform covering if:
\begin{enumerate}
\item \label{jkkpgwegrewfre} $f$ is controlled and uniform.
\item\label{jkkpgwegrewfre1} $f$ is  bornological and has locally finite fibres.
\item \label{okgerpetgefe}There exists an entourage $P$ of $X$ (called the connection) such that: \begin{enumerate} \item \label{wergoijergergwerf} For every   $U$ in $\cU_{Y}$ the entourage $ f^{-1}(U)\cap P$ is uniform and coarse.
\item \label{rwegjweriogewrgwregwreg} There exists  $U$ in $\cU_{Y}$ such that for every $y$ in $Y$ 
the map $f$ restricts to a 
uniform homeomorphism
$$\tilde U[f^{-1}(y)] \cong f^{-1}(U[y])  \stackrel{\cong}{\to} f^{-1}(y)_{disc}\otimes U[y]\ ,$$ where $\tilde U:=f^{-1}(U)\cap P$.
\item \label{wergjwerogewrgreg} The set $\{f^{-1}(U)\cap P\mid U\in \cU_{Y}\}$ is cofinal in $\cU_{X}$.
\end{enumerate}
\end{enumerate}
\end{ddd}

 \begin{lem}
If $P$ and $P'$ are two connections for $f$, then there exist uniform entourages $U$ and $U'$ of $ Y$ such that
$f^{-1}(U)\cap P\subseteq P'$ and $f^{-1}(U')\cap P'\subseteq P$.
\end{lem}
\begin{proof}
Let $U_{0}$  be in $\cU_{Y}$. Then $f^{-1}( U_{0} )\cap P'\in \cU_{X}$ by \cref{thetgtrggegtr}.\ref{wergoijergergwerf}. Then by \cref{thetgtrggegtr}.\ref{wergjwerogewrgreg} there exists $U$ in $\cU_{Y}$ such that
$f^{-1}(U)\cap P\subseteq  f^{-1}( U_{0})\cap P'\subseteq P'$. 
  \end{proof}

Assume that $X$ and $Y$ are $G$-uniform bornological coarse spaces and that $f$ is equivariant.
\begin{ddd}
\mbox{}
\begin{enumerate}
\item $f$ is a $G$-equivariant uniform covering if $f$ is a uniform covering with a $G$-invariant connection $P$. 
\item $f$ is a uniform $G$-covering if it is a $G$-equivariant uniform covering 
and $G$ acts freely and transitively on the fibres of $f$. \end{enumerate}

\end{ddd}
 
 \begin{ex}
 In the situation of \cref{wrtohue9rheheththe} assume that
 there exists an $r$ in $(0,\infty)$ such that  all fibres of $f$ are $r$-discrete.  
This happens e.g. if $Y$ is compact. 
 Then
 $f:X\to Y$ is a uniform covering. For the connection $P$ we can take the metric entourage $U_{r/2}$ of $X$. \hB
 \end{ex}

\begin{ex}
The covering $X\to Y=\{y\in \R^{2}\mid \|y\|\ge 1 \}$ from \cref{jgwpergrfwff} is a   uniform $\Z$-covering.
The analogous construction with $X'\to  Y':=\{y\in \R^{2}\mid \|y\|>0 \}$ does not yield
 a uniform covering since condition  \cref{thetgtrggegtr}.\ref{rwegjweriogewrgwregwreg} is violated.
\hB
\end{ex}

Let $X$ be a $G$-uniform bornological coarse space. We first introduce the construction of the squeezing space $\hat X_{h}$.
We start with the bounded union \begin{equation}\label{weirug9egrferw}\hat X:= \bigsqcup_{ \nat}^{\bd}X
\end{equation} in $G\UBC$.
We let $X_{n}$ denote the $n$-th component of $\hat X$.
The uniform and the coarse structure of $\hat X$ are generated by entourages of the form
$\bigsqcup_{\nat} U$ for uniform or coarse entourages $U$  of $X$, respectively. The bornology of $\hat X$ is generated by bounded subsets of $X$ placed in one of the components. 
On $\hat X$ we consider the invariant  big family $\hat \cY:=(\hat X_{\le n})_{n\in \nat}
$ given by $\hat X_{\le n}:= \bigsqcup_{m\le n }X_{m}$.

Recall the construction of hybrid coarse structures from \cite[Sec. 5.1]{buen}.

\begin{ddd}\label{kohperthtregtg}
The squeezing space $\hat X_{h}$ is the $G$-bornological coarse space  obtained from $\hat X$ by equipping it with the 
hybrid coarse structure with respect to the big family $\hat \cY$.
\end{ddd}

 By definition a coarse entourage of $\hat X$ belongs to the hybrid structure if for every uniform entourage $V$ of $X$ there exists $m$ in $\nat$ such that we have
$ U\cap (X_{n}\times X_{n})\subseteq V$ for  all $n$ in $\nat$ with $n\ge m$.
The identity of underlying sets induces a morphism $\hat X_{h}\to \iota \hat X$,
where $\iota:G\UBC\to G\BC$ is the forgetful functor.

The construction of the squeezing space extends to a 
  functor $$X\mapsto \hat{X}_{h}:G\UBC \to G\BC$$
  in the obvious way. We furthermore have a natural transformation  
  $$(\hat{-})_{h}\to (\hat{-}): G\UBC\to G\BC\ .$$ 

 Let $f:X\to Y$ be a controlled, bornological and uniform (but not necessarily proper) map between $G$-uniform bornological coarse spaces.  
  We then get  a controlled and bornological map 
 $\hat f:\hat X_{h}\to \hat Y_{h}$
between the associated  squeezing spaces.
 We consider the invariant big family \begin{equation}\label{regwerwerfwefwef}\hat \cZ:=(\hat Y_{\le n})_{n\in \nat}\ , \quad \hat Y_{\le n}:=\bigcup_{m\le n}Y_{m}
\end{equation}on $\hat Y_{h}$

\begin{lem} If  $f$ is a  uniform covering, then the map $ \hat f :\hat X_{h}\to \hat Y_{h}$ is a
 branched coarse  covering  with respect to the big  family $\hat \cZ$.
\end{lem}
\begin{proof}  
\cref{wroigjowegewrgerfweferf}.\ref{etrherthtrherthrth1} is clearly satisfied. 
The map $\hat f:\hat X\to \hat Y$ is  controlled and uniform. Since
$\hat X_{h}$ is defined using the induced big family $f^{-1}(\hat \cZ)$ the map 
$\hat X_{h}\to \hat Y_{h}$ is also controlled 
by the functoriality of the construction of hybrid structures. Hence \cref{wroigjowegewrgerfweferf}.\ref{etrherthtrherthrth}
is satisfied.

Using   the  connection $P$ for $f$ we define the coarse  connection   for
$\hat f$ by $Q:=\bigsqcup_{n\in \nat} P$.
  
 Let $V$ be a coarse entourage of $\hat Y_{h}$. 
 Let $U$ be a uniform entourage of $Y$ as in \cref{thetgtrggegtr}.\ref{rwegjweriogewrgwregwreg}.
Then we choose $m$ in $\nat$ such that $V\cap (Y_{n}\times Y_{n})\subseteq U$ for all $ n\ge m+1$.
Then \cref{wroigjowegewrgerfweferf}.\ref{wetgrgffsgvdfgsfdgsdfg} is satisfied for the member  $\hat Y_{\le m}$ of $\hat \cZ$.

By \cref{thetgtrggegtr}.\ref{wergoijergergwerf} it is clear that $\hat f^{-1}(V_{\hat Y_{\le m}^{c}})\cap Q$
is a coarse entourage of $\hat X$.
We must show that it also belongs to the hybrid structure.
Let $\tilde U$ be a uniform entourage of $X$. Then by \cref{thetgtrggegtr}.\ref{wergjwerogewrgreg} there exists a uniform entourage
$U'\subseteq U$ of $Y$ such that $f^{-1}(U')\cap P\subseteq \tilde U$.
Since $V$ is a hybrid entourage there exists $m'\ge m$ such that $V\cap 
 (Y_{n}\times Y_{n})\subseteq U'$ for all $n\ge m'+1$.
 Then
 $\hat f^{-1}(V_{\hat Z^{c}_{\le m}})\cap Q \cap (X_{n}\times 
 X_{n})\subseteq \tilde U$ for all $n\ge m'$.
This shows that $\hat f^{-1}(V_{\hat Z_{\le m}^{c}})\cap Q$ belongs to the hybrid structure. We thus have verified

 Finally let $\hat V$ be a coarse entourage of $\hat X_{h}$. 
 We choose a uniform entourage $U$ of $Y$. 
   Then there exists  $k$ in $\nat$ such that $
 \hat V\cap (X_{n}\times X_{n})\subseteq f^{-1}(U)\cap P$ for all $n$ in $\nat$ with $n\ge k+1$.
 This implies  that  \cref{wroigjowegewrgerfweferf}.\ref{rgegerwerefwrefb}.
 $\hat V_{X_{\le k}^{c}}\subseteq Q$. We have verified  
  \cref{wroigjowegewrgerfweferf}.\ref{rgegerwerefwrefb}.
 \end{proof}

%
 
%
%
%

 \begin{kor}\label{kjthgprthhdh}\mbox{}\begin{enumerate}
 \item  If $f$ is a $G$-equivariant uniform covering, then  $\hat f:\hat X_{h}\to \hat Y_{h}$ is a $G$-equivariant branced coarse $G$-covering with respect to $\hat \cZ$.
 \item

 If $f:X\to Y$ is a uniform $G$-covering, then $\hat f:\hat X_{h}\to \hat Y_{h}$ is a branched coarse $G$-covering with respect to $\hat \cZ$.
 \end{enumerate}
 \end{kor}

\begin{constr}\label{hojprthertgertgerg}
 We  recall the cone construction $\cO^{\infty}:\UBC\to \BC$. 
For $Y$ in $\UBC$ the underlying bornological space  of $\cO^{\infty}(Y)$ is  $\R\otimes Y$ on which
we consider the big family $\cO^{-}(Y):=(\cO_{\le n}(Y))_{n\in \nat}$   given by the subsets $\cO_{\le n}(Y):=(-\infty,n)\times Y $.
The  coarse structure of $\cO^{\infty}(Y)$  is    the hybrid   associated to
the coarse and uniform structure on $\R\otimes Y$ and the big family $\cO^{-}(Y)$. Thus a coarse entourage $V$ of $\R\otimes Y$ is hybrid, if for every uniform entourage $U$ of $\R\otimes Y$ there exists
$n$ in $\nat$ such that
$V_{|\cO(Y)_{\le n}^{c}}\subseteq U$. \hB
\end{constr}

We define a map of sets $\hat i:\hat X\to \cO^{\infty}(\hat X)$ 
which restricts to $X_{n}\ni  x\mapsto (n,x)\in \cO^{\infty}(X)$ for all $n$ in $\nat$.
 
  \begin{lem}\label{kwogkprefwerf}
  The map $\hat i_{X}:\hat X_{h}\to \cO^{\infty}( \hat X)$ is a coarse embedding.
  \end{lem}

We now assume that $f:X\to Y$ is a uniform covering. It induces 
 map of sets
$\cO^{\infty}(f):\cO^{\infty}(X)\to \cO^{\infty}(Y)$ given   by $(t,x)\mapsto (t,f(x))$.  

\begin{prop}\label{wtiogwtgerwferferfw}
The map $\cO^{\infty}(f):\cO^{\infty}(X)\to \cO^{\infty}(Y)$ is a branched coarse covering  with respect  to $\cO^{-}(Y)$.
\end{prop}
\begin{proof}
We verify the conditions listed in  \cref{wroigjowegewrgerfweferf}.  

 \cref{wroigjowegewrgerfweferf}.\ref{etrherthtrherthrth}
follows from \cref{thetgtrggegtr}.\ref{jkkpgwegrewfre}  and the usual functoriality of the cone \cite{buen}.

  \cref{wroigjowegewrgerfweferf}.\ref{etrherthtrherthrth1} follows  from \cref{thetgtrggegtr}.\ref{jkkpgwegrewfre1}.  
%

We now consider \cref{wroigjowegewrgerfweferf}.\ref{rgegerwerefwref}. 
As   connection  for $\cO^{\infty}(f)$  we take $\hat P:=(\R\times \R)\times P$.
For an entourage $U$ of $Y$ we set $\tilde U:=P\cap f^{-1}(U)$.
 
 We start with checking \cref{wroigjowegewrgerfweferf}.\ref{wetgrgffsgvdfgsfdgsdfg}. 
Let $V$ be a coarse entourage of $\cO^{\infty}(Y)$. Let $U$ in $\cU_{Y}$  be an entourage as in \cref{thetgtrggegtr}.\ref{rwegjweriogewrgwregwreg}.
We can find $n$ in $\nat $ such that if $((t',y'),(t,y))$ is in $V$ and $t\ge n$, then $t'\ge n-1$ and $(y',y)\in U$.  

Let now $(t,y)$ be in $\cO^{\infty}(Y)$, $t\ge n$,  and $(t,x)$ be in $\cO^{\infty}(f)^{-1}(\{(t,y)\})$. Consider  furthermore a point $((t',y'),(t,y))$ in $V$.
Then there is a unique $(t',x')$ in $\cO^{\infty}(f)^{-1}(t',y')$
such that $(t',x')\in   \hat P$. Indeed, we must take for $x'$ the uniquely determined element in $\tilde U[x]$ over $y'$ in $U[y]$.
  
 We now check  \cref{wroigjowegewrgerfweferf}.\ref{wrgtuiwhigwergegr}. 
 We must check that  
 $\hat P\cap \cO^{\infty}(f)^{-1} (\cO^{\infty}(Y)^{c}_{\le n})$ is a coarse entourage of $\cO^{\infty}(X)$. First of all   \cref{thetgtrggegtr}.\ref{wergoijergergwerf} implies that
 $\hat P\cap \cO^{\infty}(f)^{-1} (\cO^{\infty}(Y)^{c}_{\le n})$   is a coarse entourage of $\R\otimes X$. We must further show that
  if $\hat U$ is in $\cU_{X}$ and $\epsilon $ is in $(0,\infty)$, then there  exists  $n'$ in $\nat$ such that for  $((t',x'),(t,x)) \in \hat P\cap \cO^{\infty}(f)^{-1} (\cO^{\infty}(Y)^{c}_{\le n})$ with $t\ge n'$ implies  $|t'-t|<\epsilon$ and $(x',x)\in \hat U$.
 
 First of all by \cref{thetgtrggegtr}.\ref{wergjwerogewrgreg}
there exists   $U'$ in $\cU_{Y}$ such that $U'\subseteq U$
and $\tilde U'\subseteq \hat U$. There exists $n'$ in $\nat$
such that $n\le n'$ and $((t',y'),(t,y))\in V$ with $t\ge n'$ implies that
$(y',y)\in U'$ and $|t'-t|<\epsilon$. Hence if $((t',x'),(t,x)) \in \hat P\cap \cO^{\infty}(f)^{-1} (\cO^{\infty}(Y)^{c}_{\le n})$ such that $t\ge n'$, then also $(x',x)\in \tilde U'\subseteq \hat U$.

 We finally check  \cref{wroigjowegewrgerfweferf}.\ref{rgegerwerefwrefb}.
Let $\hat V$ be a  coarse entourage of $\cO^{\infty}(X)$. Then we must show that
there exists $n$ in $\nat$ such that
$\hat V_{\cO^{\infty}(X)^{c}_{\le n}}\subseteq \hat P$. By 
   \cref{thetgtrggegtr}.\ref{wergjwerogewrgreg}
there exists a uniform entourage $\hat U$ of $\R\otimes X$ contained in $\hat P$. In view of the definition of the hybrid structure there exists $n$ in $\nat $ such that $\hat V_{\cO^{\infty}(X)^{c}_{\le n}}\subseteq \hat U$.
Hence $\hat V_{\cO^{\infty}(X)^{c}_{\le n}}\subseteq \hat P$ as required.
   \end{proof}

 \begin{kor}\label{okgpthertetg}\mbox{}\begin{enumerate}
 \item  If $f$ is a $G$-equivariant uniform covering, then  $\cO^{\infty}( f): \cO^{\infty}(X)\to \cO^{\infty} ( Y)$ is a $G$-equivariant branced coarse $G$-covering with respect to $\cO^{-}(Y)$.
 \item\label{pklhtrthrteht}

 If $f:X\to Y$ is a uniform $G$-covering, then $\cO^{\infty}( f): \cO^{\infty}(X)\to \cO^{\infty} ( Y)$  is a branced coarse $G$-covering with respect to $\cO^{-}(Y)$.
 \end{enumerate}
 \end{kor}

To a bornological coarse space $X$ with entourage $U$ we associate the Rips complex $P_{U}(X)$
which we consider as  a uniform bornological coarse space. If $X'$ and $U'$ is a second bornological coarse space $X$ with entourage $U'$ and  $f:X\to X'$ is a   map  $f(U)\subseteq U'$, then we get an induced map
$P(f):P_{U}(X)\to P_{U'}(X')$.
If $f$ is proper or bornological, respectively, then $P(f)$ has the same property.

We now assume that $f:X\to Y$ is a branched coarse covering with respect to the big family $\cZ$ in $Y$.
\begin{prop}\label{sfvfdvsfd9v} For every $U$ in $\cC_{X}$ 
there exists a member $Z$ in $\cZ$ such that $P(f_{|f^{-1}(Z^{c})}):P_{U}(f^{-1}(Z^{c}))\to P_{f(U)}(Z^{c})$
is a uniform covering.
\end{prop}
\begin{proof}
The map $P(f):P_{U}(X)\to P_{f(U)}(Y)$ is $1$-Lipschitz and therefore  controlled and uniform. Hence
 \cref{thetgtrggegtr}.\ref{jkkpgwegrewfre}  is satisfied.
 
 Since $f$ is bornological, also $P_{U}(f)$ is bornological.

We  choose $Z$ in $\cZ$  such that the parallel transport for $f$ at scale $f(U)^{2}$ is defined on $Z^{c}$.
For a  simplex in $P_{f(U)}(Z^{c})$ the choice of a    preimage of a vertex uniquely determines a preimage of 
the simplex.  

The preimage of a point in a simplex of $P_{f(U)}(Z^{c})$  is in bijection  to   the fibre of $f$
on a vertex of this simplex by a map of propagation $1$. This implies that the fibres of $f$ are locally finite. 
Hence
 \cref{thetgtrggegtr}.\ref{jkkpgwegrewfre1}  is satisfied.

We now argue that \cref{thetgtrggegtr}.\ref{okgerpetgefe} is satisfied.
For the connection of $P(f_{|Z^{c}})$ we can take the metric entourage of width $1$.
For the entourage $U$ in \cref{thetgtrggegtr}.\ref{rwegjweriogewrgwregwreg} we can also take the metric entourage of width $1$. 
\end{proof}
%
%
Assume now that $f:X\to Y$ is a $G$-equivariant branched coarse covering with respect to the invariant big family $\cZ$ of $Y$.
If $U$ is a $G$-invariant entourage of $X$, then $P(f):P_{U}(X)\to P_{f(U)}(Y)$ is clearly $G$-invariant.
The first assertion of the following corollary follows from the assertion of \cref{sfvfdvsfd9v}
while the second uses details of the proof.
\begin{kor}\label{okoprtkgpobgfbdfgbdfgbdfgbdgfb}\mbox{}
\begin{enumerate}
\item 
 For every $U$ in $\cC_{X}^{G}$ 
there exists a member $Z$ in $\cZ$ such that $P(f_{|f^{-1}(Z^{c})}):P_{U}(f^{-1}(Z^{c}))\to P_{f(U)}(Z^{c})$
is a $G$-equivariant uniform covering.
\item\label{elrtjhperthtrgrtgetr} $f$ is a   branched coarse $G$-covering with respect to $\cZ$, then with $U$ and $Z$ as above 
  $P(f_{|f^{-1}(Z^{c})}):P_{U}(f^{-1}(Z^{c}))\to P_{f(U)}(Z^{c})$
is a  uniform $G$-covering.
\end{enumerate}
\end{kor}

\section{Finite asymptotic dimension}\label{kohgprthtrgergtreg} 
   
We start with some notions associated with a family 
$\cW$ of subsets of a set $X$.
 If $U$ is an an entourage of $X$ and $B$ is a subset, then we say that $B$ is 
  $U$-bounded if $B\times B\subseteq U$. 
%
%
%
  The entourage $U$   is called
  a Lebesgue entourage of $\cW$   if every $U$-bounded subset of $X$ is contained in a member of $\cW$. 
  We let $$\leb(\cW):=\{U\subseteq X\times X\mid \mbox{$U$ is Lebesgue entourage of $\cW$}\}$$
  denote the set of all Lebesgue entourages of $\cW$.
  The bound of $\cW$ is the entourage
$$\bd(\cW) :=\bigcup_{W\in \cW} W \times W $$ of $X$.
Finally, the 
 multiplicity of the family $ \cW$ is defined as
$$\mult(\cW):=\sup_{x\in X}   | \{W\in \cW \mid x\in W\}|\ .$$

Note that $\mult(\cW)$ is an element of $\nat\cup \infty$.

We now assume that $X$ is a coarse space with coarse structure $\cC_{X}$.

  
   \begin{ddd}\label{uihefigvweeffd1} \mbox{}
   \begin{enumerate}
 \item $X $ has  dimension $n$ at scale  $U$  in $\cC_{X}$ if there exists a  family $\cW$  with  $U\in \leb(\cW)$,  $\bd(\cW)\in \cC_{X}$,  and $\mult(\cW)\le n+1$.
  \item \label{kgopwergrwef}
 $X $ has finite dimension at coarse scales if
 it     has   finite dimension at every   scale   $U$  in $\cC_{X}$    (the dimension  may depend on the scale).
  \item \label{okgopwererfwerf} $X $
  has  asymptotic dimension $\le n$ if it has dimension $\le n$ at every   scale $U$  in $\cC_{X}$ . 
\end{enumerate}
  \end{ddd} 
    If $\cY$ is a big family on $X$, then we say that $(X,\cY )$ has a property listed in \cref{uihefigvweeffd1} eventually
  if there exists a member $Y$ in $\cY$ such that $Y^{c} $  has this property. Thereby in \cref{kgopwergrwef} the set $Y$ may depend on $U$.

\begin{construction}\label{weiogowegrefwfwerfwf}{\em 
Let $f:X\to Y $ be a coarse branched covering relative to the big family $\cZ$ on $Y$ and assume that $(X,f^{-1}(\cZ)) $ has 
eventually finite dimension at coarse scales. Let $P$ denote a connection for $f$.
Let $V$ be an entourage of $Y$. Then we can choose a member $Z$ of $\cZ$ as follows:
\begin{enumerate}
\item We first choose $Z$ such that  $U:=f^{-1}(V_{Z^{c}})\cap P$ is a coarse entourage of $X$. Here we apply  \cref{wroigjowegewrgerfweferf}.\ref{wrgtuiwhigwergegr}.
\item After increasing $Z$  if necessary there exists a   covering $\cW=(W_{i})_{i\in I}$  of $f^{-1}(Z^{c})$  with Lebesgue entourage $U^{2}$, $\tilde U:=\bd(\cW)\in \cC_{X}$, and  with multiplicity $n:=\mult(\cW)<\infty$. This uses the assumption.
\item  We increase $Z$ further in order to ensure that  restriction of $f$ to any $\tilde U$-bounded subset of $f^{-1}(Z^{c})$
is injective.  To this end by \cref{wroigjowegewrgerfweferf}.\ref{wetgrgffsgvdfgsfdgsdfg} we must choose $Z$  so large that the parallel transport at scale $f(\tilde U)$ is defined on $Z^{c}$. We then restrict the covering $\cW$ to the preimage of the new $Z^{c}$.
\end{enumerate}
Recall the definition 
$U(B):=\{x\in X\mid U[x]\subseteq B\}$ of the $U$-thinning of a subset $B$ of $X$.
Since $U^{2}$ is a Lebesgue entourage of $\cW$ the
  family   of  $U$-thinnings $(U(W_{i}))_{i\in I}$   is still a covering.
By choosing appropriate subsets $\hat W_{i}\subseteq U(W_{i})$ for all $i$ in $I$ we can construct a partition 
 $\hat \cW:=(\hat W_{i})_{i\in I}$ of $f^{-1}(Z^{c})$.   
 We have $U[\hat W_{i}]\subseteq W_{i}$ for every $i$ in $I$.
 
 If $(X,f^{-1}(\cZ))$ eventually has finite asymptotic dimension $\le n'$, then we can take $n=n'+1$ independently of $V$.
%
%
 \hB
}
 \end{construction}
 
 \begin{construction}\label{weiogowegrefwfwerfwf1}{\em 
Let $f:X\to Y $ be a coarse branched covering relative to the big family $\cZ$ on $Y$ and assume that $(Y,\cZ) $ has 
eventually finite dimension at coarse scales. We choose a connection  $P$  for $f$. Le  $V$ be an entourage of $Y$. Then we can choose a member $Z$ of $\cZ$ as follows:
\begin{enumerate}
\item  Using the assumption we first choose $Z$ in $\cZ$ such that there exists a covering $\cW=(W_{i})_{i\in I}$
of $Z^{c}$ with Lebesgue entourage $V^{2}$, $\tilde V:=\bd(\cW)\in \cC_{Y}$, and $n:=\mult(\cW)<\infty$.
\item Using \cref{wroigjowegewrgerfweferf}.\ref{rgegerwerefwrefa} we then enlarge $Z$ such that the parallel transport at scale $\tilde V$ is defined on $Z^{c}$ and
$f^{-1}(\tilde V_{Z^{c}})\cap P$ is a coarse entourage of $X$. 
We restrict $\cW$ to the new $Z^{c}$.
%
\end{enumerate}
Then  $(V(W_{i}))_{i\in I}$ is still a covering of $Z^{c}$.
By choosing appropriate subsets $\hat W_{i}\subseteq V(W_{i})$ for all $i$ in $I$ we can construct a partition 
 $\hat \cW:=(\hat W_{i})_{i\in I}$ of $Z^{c}$.    We have $V[\hat W_{i}]\subseteq W_{i}$ for every $i$ in $I$.
%
%

If $(Y,\cZ)$ eventually has finite asymptotic dimension $\le n'$, then we  can take $n=n'+1$ independently of $V$.
 \hB
}
 \end{construction}

%
%
%
%
%
%

Let $f:X\to Y$ be a coarse branched  covering relative to the family $\cZ$ in $\cY$.  
\begin{lem}\label{woigwgwegferfrefw} If  $(Y, \cZ)$ has eventually finite asymptotic dimension $\le n$ (finite  dimension at coarse scales), then
$(X,f^{-1}(\cZ))$ also has eventually finite asymptotic dimension $\le n$ (finite dimension at coarse scales).
\end{lem}
\begin{proof}  
It suffices to show for every $U$ in $\cC_{X}$, that if  $Y$  eventually  has dimension $n$ at scale $V:=f(U)$, then
$X$ has eventually dimension $n$ at scale $U$.
We can assume that $U$ and hence $V$ are symmetric.

We choose a connection $P$.
 By assumption  we can choose $Z$ in $\cZ$ and a covering  $\cW$ of $ Z^{c}$ with Lebesgue entourage $   V $, bound $\bd(\cW)\in \cC_{X}$ and $\mult(\cW)\le n+1$  such that the parallel transport is defined at scales $\bd(\cW)$, $V\circ \bd(\cW)$  and $V$ on $  Z^{c}$, and that $ U_{f^{-1}(Z^{c})}\subseteq P$.

Without loss of generality we can assume that all members of $\cW$ are non-empty. For every $W$ in $\cW$ we choose a base point $y_{W}$. Then we define the family of subsets
$\tilde \cW:=((\tilde W_{x})_{x\in f^{-1}(y_{W})})_{W\in \cW}$  of $X$, where $\tilde W_{x}:=\{\Phi_{y',y_{W}}(x)\mid y'\in W\}$, the set of all   parallel transports  
of points $x$ in $f^{-1}(y_{W})$ to points in fibres of points in $W$. 

We claim that $\tilde W$ does the job.
We first   check that $  U_{ f^{-1}(Z^{c})}$ is a Lebesgue entourage of $\tilde \cW$. Thus let $ A$ be a $  U_{ f^{-1}(Z^{c})}$-bounded subset of $  f^{-1}(Z^{c})$. Then $f(A)$ is $ V $-bounded. Hence there exists some $W$ in $\cW$ such that $f(A)\subseteq W$.
Let $a_{0}$ be in $A$.   Then $f(a_{0})\in W$ and since the parallel transport is defined at scale $\bd(\cW)$ on $ Z^{c}$ there exists a unique $x$ in $f^{-1}(y_{W})$ such that $(a,x)\in P$. 
We claim that $A\subseteq \tilde W_{x}$.
Let $a$ be any other point in $A$. Then $(a,a_{0})\in  U_{ f^{-1}(Z^{c})}\subseteq P$. Since $f(A)$ is
  $V$-bounded we know that   $a$ is the parallel transport of $a_{0}$. Since the parallel transport on $ Z^{c}$ is even defined at scale $V\circ \bd(W)$ we know that $a$ is the parallel transport of $y_{W}$. Hence $a\in \tilde W_{x}$, too.

 Let $x_{0}$ be in $X$. If $W$ in $\cW$ is such that $f(x_{0})\in W$, then
 there exists a unique $x$ in $f^{-1}(y_{W})$ with $x_{0}\in \tilde W_{x}$. Indeed, $x$ must be the parallel transport at scale $V$ (at this point we use the assumption that $V$ is symmetric) of
 $x_{0}$ to the fibre $f^{-1}(y_{W})$.
This observation implies that $\mult(\tilde \cW)\le \mult(W)\le n+1$. 
  \end{proof}

In the following we propose the analogue of the notion of topological dimension $\le n$ in the context of uniform
spaces. Recall that a topological space $X$ has topological dimension $\le n$ if every covering $\cW$ of $X$ has a refinement
$\cV$ with $\mult(\cV)\le n+1$. In the context of uniform spaces it is natural to add a condition on  the lebesgue 
entourages of $\cV$ in order to ensure that $\cV$ does not become too small.
The following definition is designed such the cone functor and the squeezing space functor 
translate finite uniform topological dimension into finite asymptotic dimension, see \cref{kohpehtgeg}.

Let $X$ be a uniform space with uniform structure $\cU_{X}$.
\begin{ddd}\label{okhnpehetre} $X$ has uniform topological dimension $\le n$   (finite uniform topological  dimension at uniform scales) if
there exists a cofinal function $\kappa:\cU_{X}\to \cU_{X}$ such that
for every $U$ in $\cU_{X}$ there exists a covering $\cW$ of $X$ with
  $U\in \leb(\cW)$,  $\bd(\cW)\subseteq \kappa(U)$ and $\mult(\cW)\le n+1$ ($\mult(\cW)<\infty$).
 \end{ddd}
  
We write $\udim(X)\le n$ if $X$ has uniform topological dimension $\le n$.
 
 \begin{rem}
 If the uniform structure on $X$ is induced by a metric and
 then function $\kappa$ is required to be linear in the metric scale, then
 \cref{okhnpehetre} reduces the definition of finite Assouad–Nagata dimension. \hB
 \end{rem}

 \begin{lem}\label{okgpertherhtr}
 If $X$ is an $n$-dimensional simplicial complex with the spherical path metric, then $X$
 has  uniform topological dimension $\le n$.
  \end{lem}
\begin{proof}  For $n$ in $\nat$
we let $\cW_{n}$ be the covering of $X$ by the stars of the vertices in the $n$th barycentric subdivision.
We set furthermore  $\cW_{-1}:=(X)$.
Then $\mult(\cW_{n})\le \dim(X)+1$ for all $n\in \nat\cup \{-1\}$.
We define $s:\cU_{X} \to \nat\cup \{-1\}$ 
by $$s(U):=\max\{n\in \nat \cup \{-1\}\mid U\in \leb(\cW_{n})   \}\ .$$
For every $U$ in $\cU_{X}$ we can then take the covering  $\cW_{s(U)}$.
We  observe that
$\kappa:\cU_{X}\to \cU_{X}$ given by $U\mapsto \bd(W_{s(U)})$ is cofinal.
  \end{proof}

   Note that smooth manifolds admit triangulations.
   \begin{kor}\label{knknonkpfgfghnfnfhnf9}
  A  closed Riemannian  manifold $M$ with the metric uniform structure
  has uniform topological dimension $\le \dim(M)$.
  \end{kor}
 
  For complete manifolds must add bounded geometry assumptions which imply that
  the uniform  structure given by the Riemannian metric is  compatible with the spherical path metric   coming from a triangulation.  

Let $X$ be a uniform bornological coarse space with a big family $\cY=(Y_{i})_{i\in I}$. 
We let $X_{h}$ denote the bornological coarse space obtained from $X$ by replacing the  coarse structure $\cC_{X}$
by the hybrid coarse structure  $\cC_{X_{h}}$ associated to $\cY$.

\begin{prop} \label{ggroewpgkorep}\mbox{}
\begin{enumerate}
\item \label{kophrtherthe}We assume that $\cY=(Y_{n})_{n\in \nat}$, i.e., that the big family is indexed by natural numbers.

\item \label{kophrtherthe1}We assume that  the uniform structure of $X$  admits a countable cofinal family of entourages.
\item We assume that there exists a uniform entourage $U_{0}$ such that $\cY$ is $U_{0}$-increasing in the sense that there exists $n_{0}$ in $\nat$ such that for all $n\ge n_{0}$ we have $U_{0}[Y_{n}^{c}]\subseteq Y_{n-1}^{c}$. 
\item We assume that $X$ has finite (coarse) asymptotic dimension.

\item We assume that $X$ has finite    uniform topological dimension.
\end{enumerate}
Then 
$X_{h}$ has finite asymptotic dimension.
\end{prop}
\begin{proof}
Assume that $\adim(X)\le n'-1$ and $\udim(X)\le n''-1$. The assumptions \ref{ggroewpgkorep}.\ref{kophrtherthe} and   \ref{ggroewpgkorep}.\ref{kophrtherthe1} on $X$ imply that every hybrid entourage in $\cC_{X_{h}}$ is contained in an entourage of the form  $U\cap U_{\delta}$, where $U$ is in $\cC_{X}$ and
 $\delta:\nat \to \cU_{X}$ is a cofinal function determining  $$U_{\delta}:=\{(x,y)\in X\times X\mid \forall n\in \nat  : x\in Y_{n}^{c} \lor y\in  Y_{n}^{c} \implies (x,y)\in \delta(n)\}\ .$$  
 We will construct a covering $\cW$ with $\mult(\cW)\le 3n'+2n''$,
$\bd(\cW)\in \cC_{X_{h}}$ and $U\cap U_{\delta}\in \leb(\cW)$.
Since $\adim(X)\le n'-1$ we can find a covering $\cW'$ with $\bd(\cW')\in \cC_{X}$,  $U \in \leb(\cW')$
and $\mult(\cW')\le n'$.

Let $\kappa: \cU_{X}\to \cU_{X}$ be as in \cref{okhnpehetre}.
For every $n$ in $\nat $ we can  choose a covering $\cW_{n}$ of $X$ with 
$\delta(n)\in \leb(\cW)$, $\bd(\cW)\subseteq \kappa(\delta(n))$ and $\mult(\cW_{n})\le n''$.


By cofinality of $\delta$ there exists $n_{1}\ge n_{0}$   in $\nat$ such that $\delta(n_{1})\subseteq U_{0}$.
We then define the covering 
\begin{equation}\label{wefqwedqqewdq}\cW:=(\cW' \cap Y_{n_{1}+1})\cup \bigcup_{n\ge n_{1}+1} ((Y_{n+1}\setminus Y_{n-1})\cap \cW'\cap  \cW_{\delta(n)})
 \ .
\end{equation} 
We show that  $U\cap U_{\delta}\in \leb(\cW)$.
 Let $B$ be $U\cap U_{\delta}$-bounded.  
  If $B\subseteq Y_{n_{1}+1}$, then since $B$ is $U$-bounded   there exists
 $W$ in $\cW' $ with $B\subseteq W$. 
 Hence
 $B\subseteq  Y_{n_{1}+1} \cap W$ and therefore $B$ is contained in a member of $\cW$.
Since $B$ is coarsely bounded,  if $B\not\subseteq Y_{n_{1}+1}$, then   there exists  a minimal $n> n_{1}+1$ such that
 $B\subseteq Y_{n+1}$.  Minimality of $n$ implies that    $B\cap Y_{n}^{c}\not=\emptyset$. 
 Since $B$ is $U_{\delta}$-bounded we get   $B\subseteq \delta(n)[Y_{n}^{c}]\subseteq U_{0}[Y_{n}^{c}]\subseteq Y_{n-1}^{c}$. Hence $B\subseteq Y_{n+1}\setminus Y_{n-1}$.
Since $B$ is $ U_{\delta}$-bounded it is $\delta(n)$-bounded and there exists $W$ in $\cW_{\delta(n)}$
a 
such that $B\subseteq W$. Since $B$ is $U$-bounded there exists furthermore   $W'$ in $\cW'$ with $B\subseteq W'$. Consequently, $B\subseteq  (Y_{n+1}\setminus Y_{n-1})\cap W\cap W'$
and is therefore contained in a member of $\cW$.
This finishes the verification of  $U\cap U_{\delta}\in \leb(\cW)$.

 By an inspection of \eqref{wefqwedqqewdq} we get 
 the crude estimate
  $\mult(\cW)\le 3n'+2n''$.
  We define the cofinal function $\delta':\nat\to \cU_{X}$ by 
 $$\delta'(n):=\left\{\begin{array}{cc} X\times X&n\le n_{1}+1 \\   \kappa(\delta(n-1))& n>n_{1}+1  \end{array} \right.
 \ .$$ 
 
 Then again by an inspection of \eqref{wefqwedqqewdq} we get   $\bd(\cW)\subseteq \bd(\cW')\cap U_{\delta'}\in \cC_{X_{h}}$.
  \end{proof}

Let $X$ be a uniform bornological coarse space. Recall the \cref{kohperthtregtg} of the squeezing space $\hat X_{h}$ and of the   \cref{hojprthertgertgerg} of the cone $\cO^{\infty}(X)$.
\begin{kor}\label{kohpehtgeg}
Assume that the uniform structure of $X$ has countably generated.
If  $X$ has  finite (coarse) asymptotic dimension and finite uniform topological dimension, then $\hat X_{h}$
and $\cO^{\infty}(X)$ have
  finite asymptotic dimension.
\end{kor}

\section{Coarse homology theories and  transfers}\label{wgererwf9}

In this section we axiomatically  introduce the concept of transfers along coarse branched coverings for equivariant coarse homology theories. We build on the axioms for a equivariant coarse homology theory proposed in  \cite{equicoarse}. 
Transfers along bounded coarse coverings have been  introduced in \cite{trans}, see \cref{ertiohertherthreh}. 
  In this section we   generalize this concept  to branched coarse coverings.
  In order to capture the complete functoriality of transfers for bounded coarse coverings in  \cite{trans}
  we introduced a $(2,1)$-category $G\BC_{\tr}$ of spans of bornological coarse spaces whose wrong-way directions consisted of 
  bounded coarse coverings. In the present paper we will not consider the full possible functoriality of transfers for branched coarse coverings. In particular we will not discuss their composition. 
  We can therefore work with the one-categories ${}^{G}\BCov$ and $G\BCov$ which we will introduce below.

We let $G\BC^{2}$ be the category of pairs in $G\BC$ defined as follows:
\begin{enumerate}
\item objects: The objects of $G\BC^{2}$ are pairs $(X,\cY)$ of $X$ in $G\BC$ and an invariant big family $\cY$ on $X$.
\item morphisms: A morphism $g:(X',\cY')\to (X,\cY)$ in $G\BC^{2}$ is a morphism $g:X'\to X$ in $G\BC$ such that $\cY'\subseteq g^{-1}(\cY)$.
\item composition:   The composition is inherited from the composition in $G\BC$. 
\end{enumerate}

We have a natural inclusion $G\BC\to G\BC^{2}$ sending $X$ to $(X,(\emptyset))$.

Let $\cC$ be a cocomplete stable $\infty$-category.
Any functor $E^{G}:G\BC\to \cC$ extends to a functor
$E^{G}:G\BC^{2}\to \cC$ given by
\begin{equation}\label{qfewdqewfefr}E^{G}(X,\cY):=\Cofib(E^{G}(\cY)\to E^{G}(X))\ .
\end{equation}
By definition we have a natural  fibre sequence
$$E^{G}(\cY)\to  E^{G}(X)\to E^{G}(X,\cY)\to \Sigma E^{G}(\cY)\ .$$

 
We define the category ${}^{G} \BCov$ of $G$-equivariant branched coarse coverings as follows:
\begin{enumerate}
\item The objects of ${}^{G} \BCov$ are pairs  $(f:X\to Y,\cZ)$  of a $G$-equivariant branched coarse covering
  $f:X\to Y$   relative to the invariant big family $\cZ$ on $Y$.
\item A morphism $(g,h):(f':X'\to Y',\cZ')\to (f:X\to Y,\cZ)$ is a pair of morphisms $g,h$ in $G\BC$  such that 
 \begin{equation}\label{qwefqwfohjfioqwefwefqerwf}
\xymatrix{X'\ar[r]^{g}\ar[d]^{f'}&X\ar[d]^{f}\\Y'\ar[r]^{h}&Y}
\end{equation}   is a cartesian diagram of coarse spaces, the bornology of $X'$ is induced via $g$ from the bornology of $X$,  and $\cZ'\subseteq f^{-1}(\cZ)$(see \cref{wejigowegfrefrefwe}).
\item composition: The composition is inherited from the composition  in $G\BC$.
\end{enumerate}

Sometimes will will consider full subcategories of ${}^{G}\BCov$ defined by putting additional conditions on the objects, e.g. the condition of having eventually finite asymptotic dimension, which will be indicated by a superscript like in ${}^{G}\BCov^{\fadim}$.

We let
$G\BCov$ be the full subcategory of   ${}^{G}\BCov$ of branched coarse $G$-coverings.

 We have forgetful functors $$s:{}^{G}\BCov\to G\BC^{2} \ , \quad t:{}^{G}\BCov\to G\BC^{2}$$ (source and target) given by $s(X\stackrel{f}{\to} Y,\cZ):=(X,f^{-1}(\cZ))$
 and  $t(X\stackrel{f}{\to} Y,\cZ):=(Y ,\cZ)$. The description of these functors on morphisms is the obvious one.
 Furthermore, the restriction of $t$ of $G\BCov$ factorizes over a functor
 $G\BCov\to \BC^{2}$ which we will also denote by $t$.

 We now consider a functor $$E^{G}:G\BC\to \cC \ .$$  
 Then we can define the functors
 $$t^{*}E^{G},s^{*}E^{G}:{}^{G}\BCov\to \cC\ .$$

 
 Let $C$ be a condition on the objects of ${}^{G}\BCov$ and
 ${}^{G}\BCov^{C}$ be a full subcategory of $G$-equivariant branched coarse coverings
 satisfying  $C$.

 \begin{ddd}\label{wegjiowgwergrwewf1}
We say that $E^{G}$ has transfers for $G$-equivariant branched coarse coverings satisfying condition $C$
 if $E^{G}$ is equipped with the additional structure of a natural transformation
$$\tr:t^{*}E^{G}\to s^{*}E^{G}:{}^{G}\BCov^{C}\to \cC\ .$$
\end{ddd}

We write $\tr_{f}$ for the evaluation of $\tr$ at the $G$-equivariant branched coarse covering $f$.

\begin{ex}
If $E^{G}$ has transfers for bounded coarse coverings 
in the sense of  \cite{trans}, then it has transfers
for  $G$-equivariant branched coarse coverings  $(f:X\to Y, \cZ)$ satisfying the condition
$$C=\mbox{$f$ is bounded coarse covering}\ ,$$
see 
\cref{ertiohertherthreh}. \hB
\end{ex}

We now consider a pair of an equivariant coarse homology theory $E^{G}:G\BC\to \cC$ 
and a coarse homology theory $E:\BC\to \cC$.

 \begin{ddd}\label{wegjiowgwergrwewf}
We say that $E^{G}$ and $E$ are  related by a transfer (equivalence) for   branched coarse $G$-coverings satisfying condition $C$
 if  we have the  additional structure of a natural transformation (equivalence)
$$\tr:t^{*}E\to s^{*}E^{G}:G\BCov^{C}\to \cC\ .$$
\end{ddd}

We again write $\tr_{f}$ for the evaluation of $\tr$ at the  branched coarse $G$-covering $f$.

\begin{ex}
We can define $$E(-):=E^{G}(G_{min,min}\otimes \Res_{G}(-)):\BC\to \cC\ .$$
To give a transfer equivalence relating
$E^{G}$ and $E$ for  branched coarse $G$-coverings which are bounded coarse coverings and therefore    isomorphic to
$G_{min,min}\otimes X\to X$ is equivalent to give a trivialization of the $G$-action on
the functor $E^{G}(G_{min,min}\otimes -)$ induced by the right $G$-action on $G_{min,min}$. \hB
\end{ex}

We define the category ${}^{G}\UCov$ of  $G$-equivariant uniform coverings  as follows:
\begin{enumerate}
\item The objects of ${}^{G}\UCov$ are $G$-equivariant uniform coverings $f:X\to Y$.   
\item A morphism $(g,h):(f:X\to Y)\to (f':X'\to Y')$ is a pair of morphisms
$g,h$ in $\UBC$ such that $$ 
\xymatrix{X'\ar[r]^{g}\ar[d]^{f'}&X\ar[d]^{f}\\Y'\ar[r]^{h}&Y}
$$   is a cartesian square of uniform and coarse spaces and the bornology of $X'$ is induced via $g$ from the bornology of $X$.
\end{enumerate}
We let $G\UCov$ denote the full subcategory of ${}^{G}\UCov$  of uniform $G$-coverings.
%
%

The cone functor $ \cO^{\infty}$ induces a functor 
 $$\cO^{\infty}:{}^{G}\UCov\to {}^{G}\BCov$$ which sends
 $(f:X\to Y)$ to $(\cO^{\infty} (f):\cO^{\infty}(X)\to \cO^{\infty}(Y),\cO^{-}(Y))$.
These constructions  preserve $G$-coverings.
%
%
%
%
%
%

Note that the transfer for a branched coarse covering is defined one the level of relative homology.
Taking advantage of the fact that the members of the big family $\cO^{-}(Y)$ are flasque
for cones of uniform coverings we can define  an absolute transfer.

Let $C$ be a condition ensuring that $E^{G}$ has transfers for  $G$-equivariant coarse coverings
satisfying $C$. We let ${}^{G}\UCov^{C}$ be the full subcategory of $G$-equivariant uniform coverings $f:X\to Y$
such that $(\cO^{\infty}(f):\cO^{\infty}(X)\to \cO^{\infty}(Y),\cO^{-}(Y))\in {}^{G}\BCov^{C}$.
See e.g. \cref{kohpehtgeg} for conditions involving dimensions.

The horizontal maps in the square below are equivalences since $\cO^{-}(X)$ and $\cO^{-}(Y)$ consist of flasque members.
\begin{ddd}\label{khphertgetrg}
We define the cone transfer
$$\tr_{\cO^{\infty}}:t^{*}E^{G}\cO^{\infty}\to s^{*}E^{G}\cO^{\infty}:{}^{G}\UCov^{C}\to \cC$$ such that
 $$\xymatrix{E^{G}(\cO^{\infty}(Y)) \ar[r]^-{\simeq}\ar[d]^{\tr_{\cO^{\infty}(f)}} &E^{G}(\cO^{\infty}(Y),\cO^{-}(Y)) \ar[d]^{\tr_{\cO^{\infty}(f)}} \\E^{G}(\cO^{\infty}(X) ) \ar[r]^-{\simeq} & E^{G}(\cO^{\infty}(X),\cO^{-}(X))}$$ 
naturally commutes.
\end{ddd}

We let $G\UCov^{C}$ be the full subcategory of ${}^{G}\UCov $ of  uniform $G$-coverings.
 We furthermore consider a pair  $(E^{G},E)$  of homology theories related by a transfer (equivalence) for branched coarse $G$-coverings
satisfying condition $C$ (see \cref{wegjiowgwergrwewf}).
 
 \begin{ddd}\label{kophokhotrpherthetrhee}
We define the cone transfer (equivalence)
$$\tr_{\cO^{\infty}}:t^{*}E\cO^{\infty}\to s^{*}E^{G}\cO^{\infty}:G\UCov^{C}\to \cC$$ such that
 $$\xymatrix{E(\cO^{\infty}(Y)) \ar[r]^-{\simeq}\ar[d]^{\tr_{\cO^{\infty}(f)}} &E(\cO^{\infty}(Y),\cO^{-}(Y)) \ar[d]^{\tr_{\cO^{\infty}(f)}} \\E^{G}(\cO^{\infty}(X) ) \ar[r]^-{\simeq} & E^{G}(\cO^{\infty}(X),\cO^{-}(X))}$$ 
naturally commutes.
\end{ddd}

\section{The transfer for coarse algebraic $K$-homology}\label{kgopertgtergebdbfgbdfgb}

In this section we show that the equivariant coarse algebraic $K$-theory functor $K\cX^{G}_{\bA}$ with coefficients in an additive category $\bA$ with strict $G$-action
admits the additional structure of transfers for all $G$-equivariant branched coarse coverings.
We will construct this structure on the level of additive categories of equivariant controlled objects in $\bA$.
We choose an approach which can easily be modified to cover the case of topological coarse $K$-homology in subsequent sections.

Let $k$ be a commutative base ring.
We consider an  idempotent complete and   sum complete compactly generated $k$-linear additive category $ \bA$ with $G$-action.
We write $M\to gM$ and $A\mapsto gA$ for the action of $g$ in  $G$ on objects and morphisms of $\bA$.

To  a $G$-bornological coarse space $X$ we associate  the $k$-linear additive category $\bV^{G}_{\bA}(X)$   of $X$-controlled objects
in $\bA$ which we describe in the following.

\begin{enumerate}
\item The objects of $\bV^{G}_{\bA}(X)$ are locally finite equivariant $X$-controlled objects $(M,\rho,\mu)$ where:
\begin{enumerate}
\item $M$ is an object of $ \bA$.
\item $\rho=(\rho_{g})_{g\in G}$ is a family of isomorphisms $\rho_{g}:M\to gM$ satisfying the cocycle condition
$g\rho_{g'}\circ \rho_{g}=\rho_{g'g}$ for all $g,g'$ in $G$.
\item\label{kogpwergewrf} $\mu$ is an equivariant  finitely additive locally finite projection-valued measure on $M$ defined on the  power set $\cP(X)$ of $X$.   Equivariance is the condition  $\mu(gY)=\rho_{g}^{-1}g\mu(Y) \rho_{g}$ for all $Y$ in $\cP(X)$ and $g$ in $G$. Local finiteness is the condition that   there exists a family
$(M(x),u(x))_{x\in X}$ of compact objects $M(x)$ of $\bA$ and  morphisms $u(x):M(x)\stackrel{}{\to}  M$
 presenting $M$ as the direct sum of the family $(M(x))_{x\in X}$
such that for every bounded set $B$ of $X$ the set $\{x\in B\mid M(x)\not=0\}$ is finite and for every $Y$ in $\cP(X)$
 we have  \begin{equation}\label{regewrfwefreewerfr} \mu(Y)u(x)=
\left\{\begin{array}{cc} u(x)&x\in Y\\0 & x\not\in Y  \end{array} \right.\ .
\end{equation} 
    \end{enumerate}
 \item Morphisms in $\bV^{G}_{\bA}(X)$  are equivariant and controlled morphisms in $\bA$.  More precisely, a  morphism  $A:(M,\rho,\mu)\to (M',\rho',\mu')$ in $\bV^{G}_{\bA}(X)$ is a morphism $A:M\to M'$ in $\bA$
 which satisfies $gA=\rho_{g}A\rho_{g}^{-1}$ for all $g$ in $G$, and which is $U$-controlled for some    $U$ in $\cC_{X}$ in the sense  that $\mu'(B')A\mu(B)=0$ whenever $B,B'$ are in $\cP(X)$ and $B'\cap U[B]=\emptyset$. 
 \item The $k$-linear structure and the composition is inherited from $ \bA$.
 \end{enumerate}
 
 In the notation of \cref{kogpwergewrf} for every $x$ in $X$  we define a morphism
 $u(x)^{*}:M\to M(x)$ (the projection onto the summand $M(x)$) by the condition that
\begin{equation}\label{refwrfwrfrwfwer}u(x)^{*}u(x')=\left\{\begin{array}{cc} \id_{M(x)} &x=x'\\0 & x\not=x'  \end{array} \right.\ .
\end{equation} 
A morphism $A:M\to M'$  is completely determined by its matrix coefficients  
 $$A_{x',x}:=u'(x')^{*}A u(x):M(x)\to M'(x')  \ .$$
The only condition is that for every $x$ in $X$ the set $\{x'\in X\mid A_{x',x}\not=0\}$ is finite. If $A$ is
 a morphism in $\bV^{G}_{\bA}(X)$, then  $$\supp(A):=\{(x',x)\in X\times X\mid A_{x',x}\not=0\}$$
 is an invariant coarse entourage of $X$ called the propagation of $A$.

For a morphism $f:X\to X'$ of $G$-bornological coarse space we define the $k$-linear  additive  functor
$f_{*}:\bV^{G}_{\bA}(X)\to \bV^{G}_{\bA}(X')$ as follows:
 
\begin{enumerate}
\item The functor $f_{*}$ sends the object $(M,\rho,\mu)$ to the object $f_{*}(M,\rho,\mu)=(M,\rho,f_{*}\mu)$
with $f_{*}\mu(Y):=\mu(f^{-1}(Y))$ for all $Y$ in $\cP(X')$.
\item It acts as identity on morphisms.
\end{enumerate}
One checks that $f_{*}(M,\rho,\mu)$ is again a locally finite equivariant $X'$-controlled object.
If $A:(M,\rho,\mu)\to (M',\rho',\mu')$ is a morphism in $\bV^{G}_{\bA}(X)$, then
$A$ considered as a morphism $f_{*}(M,\rho,\mu)\to f_{*}(M',\rho',\mu')$ is still equivariant and controlled, now over $X'$.
 We thus get a functor
$$\bV_{\bA}^{G}:G\BC\to \Add_{k}$$
from $G$-bornlogical coarse spaces to $k$-linear additive categories.

The construction is also functorial in the $k$-linear additive category $\bA$. 
We let $G\Add_{k}^{\sharp}$ be the subcategory of $ G\Add_{k}$ of idempotent   and sum complete
compactly generated additive categories  with $G$-action and   equivariant compact object preserving  additive functors.
\begin{lem}\label{tlrkhpetrhertgetre}
The construction $\bA\mapsto \bV^{G}(X)$ extends to a functor
$G\Add_{k}^{\sharp} \to \Add_{k}$.
\end{lem}
\begin{proof}
This is straightforward.
\end{proof}

Let $\cY$ be an invariant big family in $X$. 
  For $Y$ in $\cY$ we can regard $\bV_{\bA}^{G}(Y)$ as a full subcategory of
$\bV_{\bA}^{G}(X)$ of objects supported on $Y$. Then $$\bV_{\bA}^{G}(\cY):=\bigcup_{Y\in \cY}\bV_{\bA}^{G}(Y)$$
is a Karoubi filtration  \cite[Def. 8.3]{equicoarse}  and we can define the quotient category
$$\bV_{\bA}^{G}(X,\cY):=\frac{\bV_{\bA}^{G}(X)}{\bV_{\bA}^{G}(\cY)}\ .$$
In this way we get a functor
$$\bV^{G}_{\bA}:G\BC^{2}\to \Add_{k}\ , \qquad  (X,\cY)\mapsto \bV_{\bA}^{G}(X,\cY)\ .$$

We want to construct equivariant coarse homology theories by composing $\bV^{G}_{\bA}$ with  
 functors $H:\Add_{k}\to \cC$.
  
\begin{ddd}\label{kopwegrfw}\mbox{}
 We say that $H:\Add_{k}\to \cC$ is a homological functor for $k$-linear additive categories if it satisfies the following conditions:
 \begin{enumerate}
 \item $\cC$ is a cocomplete stable $\infty$-category.
 \item $H$ sends equivalences of $k$-linear additive categories to equivalences.
 \item $H$ preserves filtered colimits.
 \item $H$ sends Karoubi filtrations to  fibre sequences.
 \item $H$ annihilates flasque objects  (see \cite[Def. 8.1]{equicoarse}).
 \end{enumerate}
 \end{ddd}

\begin{ex}\label{koopherhegg}
Here is a list examples of homological functors.
\begin{enumerate}
\item  We let $\Add:=\Add_{\Z}$ denote the category of additive categories.
 As explained in \cite[Sec. 8.1]{equicoarse} (see also \cref{okgpwerewfwerfw}) the non-connective algebraic $K$-theory functor $$K^{\alg}:\Add\to \Sp$$
\eqref{werfwerfewfwefwref} is a homological functor for additive categories. 
\item  \cite{Bunke:2017aa}   proposes to consider the functor 
$$\UK:\Add\stackrel{\Ch^{b}_{\infty}}{\to} \Cat^{\exa}_{\infty}\stackrel{\cU_{\loc}}{\to} \cM_{\loc}$$
 where the second morphism is the universal localizing invariant of 
\cite{MR3070515}.  The properties of $\UK$ verified in the reference imply that it is a homological functor.
\item The
homotopy $K$-theory $K^{\alg}H:\Add\to \Sp$ for additive categories \eqref{wergoowerfwefewfwe} is a homological functor by \cref{kopgwerfewfwefw}.
\item Hochschild homology $HH:\Add_{\Q}\to D(\Q)$ or cyclic homology $CH:\Add_{\Q}\to D(\Q)$ 
for $\Q$-linear additive categories are homological functors, see \cite{Caputi_2020} for a discussion of these in connection with coarse homology theories.
\item \label{hetggretg} If $R$ is a unital $k$-algebra and $H:\Add_{k}\to \cC$ is a homological functor, then we can form a new homological functor
$H_{R}(-):=H(R\otimes_{k} -)$. If $R$ is non-unital, then we set
$H_{R}:=\Fib(H_{R^{+}}\to H_{k})$, where  the map is induced by the canonical projection $R^{+}\to k$ from the unitalization $R^{+}$. 
\item \label{gkweropfwerfwfrfwrefw}
 We   consider the forgetful functor  $\cZ:\Add_{k}\to \Add$.
 If $H:\Add\to \cC$ is a homological functor for additive categories, then its restriction along $\cZ$  is a homological functor 
 $$H\cZ :\Add_{k}\to \cC$$ for $k$-linear additive categories. 
\item  In the case $k=\C$ the most important example for the present paper is the homological
functor
$K^{\alg}H\cZ_{\cL^{1}}:\Add_{\C}\to \Sp$, where
$\cL^{1}$ is the algebra of trace class operators on the Hilbert space $\ell^{2}$.
Here we combine the constructions from \cref{hetggretg} and \cref{gkweropfwerfwfrfwrefw}.
 
 \end{enumerate}
  \hB
   \end{ex}
   
   Let $\bA$ be an idempotent complete and   sum complete compactly generated $k$-linear additive category $ \bA$ with $G$-action.
For a   functor $H:\Add_{k}\to \cC$ we consider the composition \begin{equation}\label{freiufiweofwerfwref}H\cX^{G}_{\bA}:=H\circ \bV^{G}_{\bA}:G\BC\to \cC
\end{equation}

 \begin{theorem}\label{jrtzjrtzhrt}
 If $H $ is  homological, then  
 $ H\cX^{G}_{\bA} $ is an equivariant coarse homology theory which is in addition
 strong and continuous.
 \end{theorem}
\begin{proof}
In  \cite[Sec. 8.3]{equicoarse} this statement was shown for $H=K$. We just observe that  the argument 
 only depends on the properties of the functor $ H$ listed in \cref{kopwegrfw}.
\end{proof}
  
  Using the notation of  \eqref{qfewdqewfefr} we have
  \begin{equation} H\cX^{G}_{\bA}(X,\cY)\simeq H(\bV^{G}_{\bA}(X,\cY))\ .
\end{equation}


 Applying \cref{jrtzjrtzhrt} to the homological functor
 $K^{\alg}\cZ:\Add_{k}\to \Sp$ (see \cref{gkweropfwerfwfrfwrefw}) we get:
 \begin{kor}\label{gwergweropifgpwerfwerfwerf}
 The functor $$K^{\alg}\cX^{G}_{\bA}:=K^{\alg}\cZ\circ \bV_{\bA}^{G}:G\BC\to \Sp$$
 is an equivariant   coarse homology theory called the equivariant coarse algebraic $K$-homology with coefficients in $\bA$. It    is   strong and continuous. 
 \end{kor}
 
 Note that $K^{\alg}\cX^{G}_{\bA}$ is also strongly additive by  \cite[Prop. 8.19]{equicoarse}
 in the sense that it sends free unions of $G$-bornological coarse spaces to products.


%


We now start with the construction of  transfers.
We let $$\ell:\Add_{k}\to \Add_{k,2,1}$$ denote the Dwyer-Kan localization at equivalences.
We realize $\Add_{k,2,1}$ by the strict two category of additive categories, additive functors and equivalences.
In the following theorem we consider the functor $\ell \bV^{G}_{\bA}:G\BC\to \Add_{k,2,1}$.
  Going over to this localization is necessary since
the construction of the transfer involves choices which would destroy the one-categorical naturality.

\begin{theorem}\label{kogpwreregwe9}
There exists a canonical  natural transformation
\begin{equation}\label{vdfpsojvklsdfvsfdv}\tr:t^{*}\ell \bV^{G}_{\bA}\to s^{*}\ell \bV^{G}_{\bA} :{}^{G}\BCov  \to \Add_{k,2,1}\ .
\end{equation}
\end{theorem}
\begin{proof} 
First of all, for every   $(X\to Y,\cZ)$  in ${}^{G}\BCov$ we will construct a $k$-linear functor $$\tr_{f}:\bV^{G}_{\bA}(Y,\cZ)\to \bV^{G}_{\bA}(X,f^{-1}(\cZ))$$
by describing its action on objects and morphisms.
Assume that 
 $(M,\rho,\mu)$ is an object in $\bV_{\bA}^{G}(Y)$. Then by \cref{kogpwergewrf} we can choose
a presentation $(M(y), u(y))_{y\in Y}$ of $M$ as a sum of   
the  components associated to the points of $Y$.

We then define the object
$$\tr_{f}(M,\rho,\mu):=(N,\sigma,\nu)$$ in $\bV^{G}_{\bA}(X)$ such that 
  $(N,(v(x))_{x\in X})$ is a choice of   a sum of the family
$(M(f(x)))_{x\in X}$. The families  $(gM(y),gu(y))_{y\in Y}$ and $(gM(f(x)),gv(x))_{x\in X}$ present $gM$ and $gN$ as  sums.
The cocycle
$\sigma=(\sigma_{g})_{g\in G}$ is defined such that
$$\sigma_{g}v(x)= gv(g^{-1}x)\ gu(g^{-1}x)^{*}\rho_{g}\ u(f(x)) :M(f(x))\to gN$$
for all $x$ in $X$.
%
Finally, the projection valued measure is determined in analogy to \eqref{regewrfwefreewerfr} by the condition 
 \begin{equation}\label{regewrfwefreewerfr1} \nu(W)v(x)=
\left\{\begin{array}{cc} v(x)&x\in W\\0 & x\not\in W  \end{array} \right.\ .
\end{equation} 
 One checks that $\tr_{f}(M,\rho,\mu)$ is an object of $\bV^{G}_{\bA}(X)$.
 In particular, in order to verify local finiteness one uses  \cref{wroigjowegewrgerfweferf}.\ref{etrherthtrherthrth1}.
 
 Let $[A]:(M,\rho,\mu)\to (M',\rho',\mu')$ be a morphism in $\bV_{\bA}^{G}(Y,\cZ)
$  represented by
a $V$-controlled morphism $A$ in $\bV_{\bA}^{G}(Y)$
  for some $V$ in $\cC_{Y}$. We let $P$ denote a connection for $f$  and choose $Z$ in $\cZ$ for $V$ 
  as in \cref{wroigjowegewrgerfweferf}.\ref{rgegerwerefwref}.  
  
  Since $[A]=[\mu'(Z^{c})A\mu(Z^{c})]$  
  we can assume that $A$ is supported on $Z^{c}$.
  Let $(N,\sigma,\nu)$ and $ (N',\sigma',\nu')$ be the transfers of $(M,\rho,\mu)$ and $(M',\rho',\mu')$. 
  We then define
  $\tr_{f}(A):=B:(N,\sigma,\nu)\to (N',\sigma',\nu')$ such that
its  matrix is given by 
\begin{equation}\label{werfwerweg}B_{x',x} = \left\{\begin{array}{cc} A_{f(x'),f(x)} & (f(x'),f(x))\in V_{Z^{c}} \:\& \:(x',x)\in P  \\0&else \end{array} \right. 
\end{equation} from 
$M(f(x))$ to $M(f(x'))$. So 
$B_{x',x}\not=0$ implies  that $x'$ is the  parallel transport of $x$ along the one-step $V$-path $(f(x),f(x'))$.

One checks that $B$ is $f^{-1}(V_{Z^{c}})\cap P$-controlled and equivariant. Furthermore, the class $[B]$ in
$\bV^{G}_{\bA}(X,f^{-1}(\cZ))$ only depends on the class $[A]$.  One finally checks that
this construction is compatible with the composition since path-lifting is compatible with composition of paths. This finishes the construction of the functor $\tr_{f}$.
Note that it depends on the choice of sums (which come together with the families of canonical inclusions).
Different choices lead to canonically isomorphic functors. This will be employed in the following construction.

For every morphism \begin{equation}\label{vsdfvdsfvsvfd}\Phi:\xymatrix{X'\ar[r]^{g}\ar[d]^{f'}&X\ar[d]^{f}\\Y'\ar[r]^{h}&Y}
\end{equation} in ${}^{G}\BCov$ we must construct
a natural transformation of functors
 \begin{equation}\label{sdbsdfbkopsgdb}t_{\Phi}:g_{*}\tr_{f'}\to \tr_{f}h_{*}:\bV^{G}_{\bA}(Y',\cZ')\to \bV^{G}_{\bA}(X,f^{-1}(\cZ))\ .
\end{equation} 
  filling the square \begin{equation}\label{rgwerfefefwef}\xymatrix{\bV^{G}_{\bA}(Y')\ar[r]^{h_{*}}\ar[d]^{\tr_{f'}}&\bV^{G}_{\bA}(Y)\ar[d]^{\tr_{f}}\\\bV^{G}_{\bA}(X')\ar[r]^{g_{*}}\ar@{=>}[ur]^{t_{\Phi}}&\bV^{G}_{\bA}(X)}\ .
\end{equation}
 Thereby we must make sure that
  if $$ 
\Psi:\xymatrix{X''\ar[r]^{g'}\ar[d]^{f''}&X'\ar[d]^{f'}\\Y''\ar[r]^{h'}&Y'}
$$ is a second morphism, then the composition 
$$\xymatrix{\bV^{G}_{\bA}(Y'')\ar[r]^{h'_{*}}\ar[d]^{\tr_{f''}}&\bV^{G}_{\bA}(Y')\ar[r]^{h_{*}}\ar[d]^{\tr_{f'}}&\bV^{G}_{\bA}(Y)\ar[d]^{\tr_{f}}\\\bV^{G}_{\bA}(X'')\ar[r]^{g'_{*}}\ar@{=>}[ur]^{t_{\Psi}}&\bV^{G}_{\bA}(X')\ar[r]^{g_{*}}\ar@{=>}[ur]^{t_{\Phi}}&\bV^{G}_{\bA}(X)}$$
 is the transformation associated to $\Phi\circ \Psi$. 
 
Saying that an object $R$  of $\bA$ is presented as the sum of a family $(R_{i})_{i\in I}$ of objects means that we are given 
the family of structure morphisms $(R_{i}\to R)_{i\in I}$.
In the following argument we will not write these structure morphisms explicitly
but one must keep track of them in the background. 
   If $(M',\rho',\mu')$ is  in $ \bV^{G}_{\bA}(Y')$ and $M'$ is represented as the sum of
 $(M'(y'))_{y'\in Y'}$, then writing  	$h_{*}(M',\mu',\rho')= (M,\mu,\rho)$ we see that $M=M'$ is canonically isomorphic to the  sum of the family $((M'(y'))_{ y'\in h^{-1}(y)})_{y\in Y}$. In a first step we
 form the partial sums $ \bigoplus_{y'\in h^{-1}(y)} M'(y')\cong M(y)$ and in a second step we get an isomorphism
  $  \bigoplus_{y\in Y} M(y)\cong M$.
 
  We get   
 $\tr_{f}h_{*}(M',\mu',\rho')=(N,\sigma,\nu)$, where $N$ is  represents the  sum of the family 
 $$   ((M'(y'))_{ y'\in h^{-1}(f(x))})_{x\in X}\ .$$
 The index family  can  be reordered as $((M'(f(x')))_{x'\in g^{-1}(x)})_{x\in X}$ whose sum is
 the object appearing in $g_{*}\tr_{f'}(M',\rho',\rho')=(\tilde N,\tilde \sigma,\tilde \nu)$.   We let $t_{\Phi}:\tilde N\to N$ denote the corresponding isomorphism.
 One then checks the compatibility with morphisms (naturality)  and compositions.
 \end{proof}

Since a homological functor $H$   for $k$-linear additive categories sends equivalences to equivalences
it has a factorization over $\ell$ as in 
$$ \xymatrix{\Add_{k}\ar[rr]^{H}\ar[dr]^{\ell}&&\cC\\&\Add_{k,2,1}\ar@{-->}[ur]^{H_{2,1}}&}\ .$$
We conclude   an equivalence of functors
\begin{equation}\label{gwrewfwerfwf}H\cX^{G}_{\bA}\simeq H_{2,1}\ell \bV_{\bA}^{G}:G\BC\to \cC\ .
\end{equation} 
 Therefore \cref{kogpwreregwe9} immediately implies:
\begin{theorem}\label{poiobipdfgbdfgbd}
The equivariant coarse  homology functor $H\cX^{G}_{\bA}$  admits the structure of transfers
\begin{equation}\label{vsdfvsdfvsdfvfd}\tr:t^{*}H\cX^{G}_{\bA}\to s^{*}H\cX^{G}_{\bA}:{}^{G}\BCov\to \cC\ .
\end{equation} for all $G$-equivariant 
coarse branched coverings. 
\end{theorem}

\begin{rem}
It follows from an inspection of the internal details of the construction that the transfer for $H\cX^{G}_{\bA}$ for $G$-equivariant branched
coarse coverings extends the transfer for bounded coarse coverings constructed in \cite{trans}. \hB
\end{rem}

We now assume that $\bA$ has the trivial $G$-action.
We then consider the functors
$$\bV_{\bA}^{G}:G\BC\to \Add_{k}\ , \quad \bV_{\bA}:\BC\to \Add_{k}\ .$$ 

\begin{theorem}\label{pkgtrgtegw} The natural transformation from \eqref{vdfpsojvklsdfvsfdv} restricts to a 
 natural equivalence $$\tr:t^{*}\ell \bV_{\bA}\stackrel{\simeq}{\to} s^{*}\ell \bV_{\bA}^{G}:G\BCov\to \Add_{k,2,1}\ .$$ 
\end{theorem}
\begin{proof} We  reuse the argument for 
 \cref{kogpwreregwe9}. Since $G$ acts trivially on $\bA$ the cocycles in the controlled objects become actions of $G$ on the corresponding objects.
 We can interpret $\bV_{\bA}(Y)$ as the full subcategory of $\bV_{\bA}^{G}(Y)$ of 
 objects $(M,\rho,\mu)$ where $\rho$ is the trivial  $G$-action.
 For varying $Y$ we interpret $\bV_{\bA}$ as a subfunctor of $\bV_{\bA}^{G}$.
 We then define  $$\tr :\bV_{\bA} \to \bV_{\bA}^{G} $$
 as the restriction of the transfer constructed in  \cref{kogpwreregwe9} to this subfunctor.
 For morphisms $\Phi$ in $G\BCov$ as in \eqref{vsdfvdsfvsvfd} we furthermore define the transformations
  $$t_{\Phi}:g_{*}\tr_{f'}\to \tr_{f}h_{*}:\bV_{\bA}(Y',\cZ')\to \bV^{G}_{\bA}(X,f^{-1}(\cZ))$$
  by restricting the transformation  
 $t_{\Phi}$ constructed in the proof of   in  \cref{kogpwreregwe9}.
 In this way we obtain the desired natural transformation.
 It remains to show that it is an equivalence.

  To this end, for every  branched coarse $G$-covering $(f:X\to Y,\cZ)$  relative to $\cZ$ we construct an inverse functor
  $$r:\bV_{\bA}^{G}(X,f^{-1}(\cZ))\to \bV_{\bA}(Y,\cZ)\ .$$
  Let $Z$ be in $\cZ$ such that $G$ acts freely and transitively on the fibres of $f$ over $Z^{c}$.
   Let
  $(N,\sigma,\nu)$ be in $\bV_{\bA}^{G}(X)$. Then
  we choose a presentation $(N(x),v(x))_{x\in X}$ of $N$ as a sum of contributions of the points of $X$.
  We choose a set-theoretic section $s:Z^{c}\to f^{-1}(Z^{c})$ of $f$.
Then we define $M$ as the sum of $(N(s(y)))_{y\in Z^{c}}$
and let $u(y):M(s(y))\to M$ denote the structure maps. 
The projection-valued measure $\mu$
on $M$ is uniquely characterized by 
 \begin{equation}\label{regewrfwefreewerfr11} \mu(W)u(y)=
\left\{\begin{array}{cc} u(y)&y\in W\\0 & y\not\in W  \end{array} \right.\ .
\end{equation} 
 In this way we get the object $(M,\mu):=r(N,\sigma,\nu)$ of $\bV_{\bA}(Y)$.
 Let now
 $[B]:(N,\sigma,\nu)\to (N',\sigma',\nu')$ be a morphism in $\bV_{\bA}^{G}(X,f^{-1}(\cZ))$.
 We can assume that $B$ is supported on $f^{-1}(Z^{\prime,c})$ for some $Z'$ in $\cZ$ such that $Z\subseteq Z'$, that there
 exists an entourage $V$ of $Y$ such that  $B$ is $f^{-1}(V_{Z^{\prime,c}})\cap P$-controlled, and that
 the parallel transport at scale $V$ is defined on $Z^{\prime,c}$.
 For every $y$ in $Z^{\prime,c}$ and $y'$ in $Y$ with $(y',y)\in V$ we define $g(y,y')$  in $G$  such that $g(y,y')s(y')$ in $f^{-1}(y')$  is  the parallel transport of $s(y)$ along the $V$-path $(y,y')$.
 We define $A:M\to M'$ uniquely such that its matrix coefficients are given by 
\begin{equation}\label{werfwergwg}A_{y',y}  :=\sigma'_{g(y',y);s'(y'),g(y',y)s'(y')} B_{ g(y',y)s'(y'), s(y)} :M(y)\to M'(y')\ .
\end{equation}  
  One checks that the class $[A]:(M,\mu)\to (M',\mu')$ does only depend on the class $[B]$, and that the
  construction is compatible with compositions.

  We have  morphisms
  $\alpha_{(M,\rho)}:r(\tr_{f}(M,\rho)) \to  (M,\rho)$ and
  $\beta_{(N,\sigma,\nu)}:\tr_{f}(r(N,\sigma,\nu)) \to (N,\sigma,\nu)$. 
 For $\alpha_{(M,\rho)}$  we note that
 $r(\tr(M,\rho))=(M',\rho')$ such that $M'$ is a sum of the family $(N(s(y))_{y\in Y}$ which can be identified with the family
 $(M(y))_{y\in Y}$ summing to $M$.  We let $\alpha_{(M,\mu)}$ be the resulting isomorphism  $M'\stackrel{\cong}{\to} M$.
 For $\beta_{(N,\sigma,\nu)}$ we note that 
 $\tr_{f}(r(N,\sigma,\nu))=(N',\sigma',\nu')$  and $N'$ is the sum of the family $(N(s(f(x)))_{x\in X}$.
 For $x$ in $ f^{-1}(Z^{c})$ we can define $g(x)$ in $G$ uniquely  such that $g(x)s(f(x))=x$.
 The family $(\rho_{g;x,g(x)s(f(x))}^{-1})_{x\in f^{-1}(Z^{c})}$ induces an isomorphism from the family $(N(s(f(x)))_{x\in f^{-1}(Z^{c})}$
to the family
 $(N(x))_{x\in f^{-1}(Z^{c})}$ which sums to the summand $ f^{-1}(Z^{c})N$ of $N$. We let $\beta_{(N,\sigma,\nu)}$ be the resulting  morphism
 $N'\to N$. Then $[\beta_{(N,\sigma,\nu)}]$ is an isomorphism in $\bV^{G}_{\bA}(X,f^{-1}(\cZ))$.
 
One checks that the isomorphisms $[\alpha_{(M,\rho)}]$ and $[\beta_{(N,\sigma,\nu)}]$
  are components of natural isomorphisms $\alpha:r\circ \tr_{f}\stackrel{\cong}{\to } \id$ and
  $\beta:\tr_{f}\circ r\stackrel{\cong}{\to} \id$
   witnessing that
  $\tr_{f}$ and $r$ are inverse to each other equivalences.
      \end{proof}

Recall that $\bA$ is   an  idempotent complete and   sum complete compactly generated additive category. Let $H:\Add_{k}\to \cC$ be a homological functor for $k$-linear additive categories and 
 consider the functors
$$H\cX_{\bA}^{G}:G\BC\to \cC\ , \quad H\cX_{\bA}:\BC\to \cC\ .$$ 
 In view of  \eqref{gwrewfwerfwf} the \cref{pkgtrgtegw} implies:

 \begin{theorem} The  functors $H\cX_{\bA}$ and $H\cX_{\bA}^{G}$ are related by a natural transfer equivalence 
\begin{equation}\label{vsdfvsdfvwrdfv}\tr:t^{*}H\cX_{\bA}\stackrel{\simeq}{\to} s^{*}H\cX^{G}_{\bA}:G\BCov\to \cC\ .
\end{equation} 
 \end{theorem}
  
%
%
%

%

\begin{kor}
The transfers or transfer equivalences are compatible with natural transformations between coarse homology theories induced by natural transformation between homological functors.
\end{kor}

 \begin{ex}
For example, there is a natural transformation of homological functors functors $$c^{GJ}:K^{\alg}\cZ  \to HC^{-} :\Add_{\Q}\to D(\Q)$$ given by the 
Goodwillie-Jones character.
The following square of functors on ${}^{G}\BCov$ commutes
$$\xymatrix{t^{*}K^{\alg}\cX^{G}_{\bA}\ar[r]^{c^{GJ}}\ar[d]^{\tr} &t^{*}HC^{-} \cX^{G}_{\bA} \ar[d]^{\tr} \\ s^{*}K^{\alg}\cX^{G}_{\bA} \ar[r]^{c^{GJ}} &s^{*}HC^{-} \cX^{G}_{\bA} } \ .$$ \hB
\end{ex}

A similar statement holds for the transfer equivalences for  branched coarse $G$-coverings.

 \section{Uncompleted topological coarse $K$-homology}\label{kogpegregwerf}

In this section we introduce the equivariant  uncompleted topological coarse $K$-homology
and show that it admits transfers for $G$-equivariant branched coarse coverings and transfer
equivalences for branched coarse $G$-coverings satisfying a finiteness condition \cref{jgoiertbgdbgf}. The uncompleted topological coarse $K$-homology
is designed to be an algebraic approximation of the equivariant topological coarse $K$-homology.
\footnote{A first version of this theory was constructed by A. Engel (unpublished yet).}

Analogously to  the equivariant topological $K$-homology $K\cX^{G}_{\bC}$ constructed in \cite{coarsek} (see \eqref{ojoijiorfwrf})
the uncompleted topological coarse $K$-homology depends on a (non-unital) $C^{*}$-category $\bC$ with $G$-action which  admits all AV-sums  and whose multiplier category $M\bC$ is   idempotent complete.    Recall from \cite{antoun_voigt}, \cite{cank} that an object decomposed into an orthogonal family
of subobjects is the AV-sum of this family if
the net of projections onto the finite partial sums converges to the identity in the strict topology of the multiplier category.

 Modifying the constructions from \cref{kgopertgtergebdbfgbdfgb} we define the functor  
$$\bV^{G,\ctr}_{\bC}:G\BC\to \Add_{\C}$$
 which associates  to a $G$-bornological space  $X$ the following $\C$-linear additive category 
$\bV^{G,\ctr}_{\bC}(X)$.
Note that the multiplier category $M\bC$ is a $\C$-linear additive category. 
For the following we interpret the term "compact object" used in the algebraic case in  \cref{kgopertgtergebdbfgbdfgb} as "object whose identity belongs to $\bC$".
Furthermore infinite sums are interpreted as AV-sums.  
 The objects and morphisms
 of $\bV^{G,\ctr}_{\bC}(X)$ are the objects $(M,\rho,\mu)$ or morphisms $A:(M,\rho,\mu)\to (M',\rho',\mu')$
  of $\bV_{M\bC}^{G}(X)$ as in  \cref{kgopertgtergebdbfgbdfgb} with the following additional properties:
 \begin{enumerate}
 \item The cocycle $\rho$ consist of unitary operators.
 \item The measure $\mu$ takes values in selfadjoint projections.
 \end{enumerate}

 \begin{prop}\label{jiopgwetgergerfw} The functor $\bV_{\bC}^{G,\ctr}:G\BC\to \Add_{\C}$ has the following properties:
  \begin{enumerate}
  \item \label{gjhghjjh1}It sends coarse equivalences to equivalences.
  \item \label{gjhghjjh2}It is $u$-continuous.
  \item \label{gjhghjjh3}It sends flasque bornological coarse spaces to flasque additive categories.
   \item \label{gjhghjjh4} It sends pairs $(X,\cY)$ in $G\BC^{2}$ to Karoubi filtrations $\bV^{G,\ctr}_{\bC}(\cY)\subseteq \bV^{G,\ctr}_{\bC}(X)$.
   \item \label{gjhghjjh5} If $(Z,\cY)$ is an invariant excisive pair on $X$, then
   the induced map $$\frac{\bV^{G,\ctr}_{\bC}(Z)}{\bV^{G,\ctr}_{\bC}(Z\cap\cY)}
\to \frac{\bV^{G,\ctr}_{\bC}(X)}{\bV^{G,\ctr}_{\bC}(\cY)}$$ is an equivalence.  \end{enumerate}
 \end{prop}
\begin{proof}
The Properties   \ref{gjhghjjh1},  \ref{gjhghjjh2},  \ref{gjhghjjh3},  \ref{gjhghjjh5}
are shown  by the similar arguments as for the complete version  $\bV^{G}_{\bC}$ in \cite[Sec. 6]{coarsek}.
The Property  \ref{gjhghjjh4} is shown  by the same argument as in the algebraic case for $\bV^{G}_{\bA}$  in \cite[Lem. 8.14]{equicoarse} using that equivariant locally finite controlled objects   can be restricted to invariant subsets $Y$ 
  in the obvious manner using the assumption that $M\bC$ is idempotent complete.
\end{proof}

For a functor $H:\Add_{\C}\to \cC$  we  consider the composition
\begin{equation}\label{fwqedewdqweddqe}H\cX^{G,\ctr}_{\bC}:=H\circ \bV^{G,\ctr}_{\bC}:G\BC\to \cC
\end{equation}Recall  \cref{kopwegrfw} of a homological functor.
\begin{theorem}\label{hlkepthetrgege} If $H$ is homological, then 
$H\cX^{G,\ctr}_{\bC} $
is an equivariant coarse homology theory which is in addition strong and continuous.
\end{theorem}
 \begin{proof}  
  The derivation of the axioms    of a coarse homology theory and of the additional properties  for  $H\cX^{G,\ctr}_{\bC} $ from \cref{jiopgwetgergerfw} and \cref{kopwegrfw}
  is essentially straightforward, see 
   \cite[Sec. 8.3]{equicoarse} or  \cite[Sec. 7]{coarsek} for similar arguments.  
 \end{proof}

We now construct transfers.
Let $(f:X\to Y,\cZ)$ be a $G$-equivariant  branched coarse covering relative to the big family $\cZ$. 
We recall  \cref{uihefigvweeffd1} and introduce the following abbreviations.
\begin{ddd}\label{jgoiertbgdbgf}\mbox{}
\begin{enumerate}
\item By   $s\fdsc$ we denote the condition that  the source $(X,f^{-1}(\cZ) )$ eventually has finite dimension at coarse scales.
\item By   $\fdsc$ we denote the condition that  the target $(Y,\cZ)$  eventually has finite dimension at coarse scales.  \end{enumerate}
\end{ddd}

Note that $\fdsc$ for $(f:X\to Y,\cZ)$  implies by \cref{woigwgwegferfrefw} that
also the source  $(X,f^{-1}(\cZ))$  eventually has  finite dimension at coarse scales.

\begin{theorem}\label{kogpwreregwe91}
We have a canonical natural transformation
\begin{equation}\label{vsdfvsdfvcsdfc}\tr:t^{*}\ell \bV^{G,\ctr}_{\bC}\to s^{*}\ell \bV^{G,\ctr}_{\bC} :{}^{G}\BCov^{s\fdsc}  \to \Add_{\C,2,1}\ .
\end{equation}
\end{theorem}
\begin{proof} 
We reuse the proof of \cref{kogpwreregwe9}.
On the level of objects there is no difference.
The transfer of morphisms is defined in \cref{kogpwreregwe9} using matrices. 
In the present situation we must make sure that the 
 transferred matrices determine bounded operators.

Let $f:X\to Y$ be a branched coarse covering relative to the big family $\cZ$ on $Y$.
Let $ A:(M,\rho,\mu)\to (M',\rho',\mu')$ be a morphism $\bV^{G,\ctr}_{\bC}(Y)$ which is controlled by an entourage $V$. 
We let $(N,\sigma,\nu)$ and $(N',\sigma',\nu')$ denote the transferred objects in $\bV^{G,\ctr}_{\bC}(X)$.
Furthermore we let $(u(y):M(y)\to M)_{y\in Y}$, $(u'(y):M'(y)\to M')_{y\in Y}$ and
$(v(x):M(f(x))\to N)_{x\in X}$, $(v'(x):M'(f(x))\to N')_{x\in X}
$ denote the family of isometric inclusions chosen in the construction of the transfers.

We assume that $Z$ is a member of $\cZ$ such that the parallel transport at scale $V$ is defined on $Z^{c}$. We can assume that $A$ is supported on $Z^{c}$.   The transfer of $A$ shall be determined by the matrix \eqref{werfwerweg}
 $$B_{x',x}:=\left\{\begin{array}{cc} A_{f(x'),f(x)}&(x',x)\in P \land (f(x'),f(x))\in V_{Z^{c}}\\0 &else  \end{array} \right.\ ,$$
 where $B_{x',x}:M(f(x))\to M'(f(x'))$ and $A_{y',y}:=u'(y')^{*}Au(y)$ for all $y,y'$ in $Y$. 
 For every $x$  in $f^{-1}(Z^{c})$ the map
\begin{equation}\label{werfrewfwerf}B_{x}:=\sum_{x'\in f^{-1}(Z^{c})} v'(x') B_{x',x}:M(x)\to N'
\end{equation}
is well-defined since the sum has finitely many non-zero terms. 
Since $A$ has propagation controlled by $V$ we see that if $x'$ contributes to this sum, then
$f(x')\in V[f(x)]\cap \supp(M',\rho',\mu')$.    Since 
$V[f(x)]$ is bounded  and the support is locally finite this
intersection a finite set. The contributing $x'$  are the parallel transports of $x$ at scale $V$ 
to the fibres over the points in this finite set.


We must show that \begin{equation}\label{sfggerwggfs}\sum_{x\in f^{-1}(Z^{c})} B_{x}v(x)^{*}
\end{equation}  converges to a bounded operator
 $B:N\to N'$.  
 
 We apply \cref{weiogowegrefwfwerfwf} which gives
a possibly larger member $Z$ of $\cZ$, a   partition  $\hat \cW=(\hat W_{i})_{i\in I}$   and a covering 
$\cW=(W_{i})_{i\in I}$ of $f^{-1}(Z^{c})$ such  $n:=\mult(\cW)<\infty$, and $U[\hat W_{i}]\subseteq W_{i}$ and
 $f$ is injective on $W_{i}$ for every $i$ in $I$,  where  $U:=f^{-1}(V_{Z^{c}})\cap P$.  

 \begin{lem}\label{rekgopewergwerfwerf}
 The sum in \eqref{sfggerwggfs} converges and  defines $B:N\to N'$  in $M\bC$ with $$\|B\|\le \sqrt{n}\|A\|\ .$$
 \end{lem}
 \begin{proof}

We fix $i$ in $I$. We claim that
$$B_{i}:=\sum_{x\in \hat W_{i}} B_{x}v(x)^{*}:N\to N'$$ is well-defined and satisfies $\|B_{i}\|\le \|A\|$.
 To  this end we define the morphisms
$$a_{i}:=\sum_{x'\in W_{i}} u'(f(x')) v(x')^{*}  : N' \to M'\ ,  \quad 
 b_{i}:=\sum_{x\in \hat W_{i}}   u(f(x)) v(x)^{*} : N\to M\ .$$ 
Since the subsets $W_{i}$ and $\hat W_{i}$ of $X$ are bounded
and $\supp(N,\sigma,\nu)$ and $\supp(N',\sigma',\nu')$ are locally finite subsets of $X$
 theses sums have finitely many non-zero terms. 
Since  
 $f$ is injective on $W_{i}$ and $\hat W_{i}$ the morphisms $a_{i}$ and $b_{i}$ are partial isometries.  
 More precisely, $a_{i}$ is an isometry between the subobjects   $\nu'(W_{i})N$ and $\mu'(f(W_{i}))M'$, while
 $b_{i}$ is an isometry between $\nu(\hat W_{i})N$ and $\mu(f(\hat W_{i}))M$.
 We observe that
$B_{i}= a_{i}^{*}A b_{i}$ since 
  for every $x$ in $f^{-1}(Z^{c})$ we have 
$a_{i}^{*}A b_{i} v(x)=  B_{x}$.
We can then
  conclude  that $\|B_{i}\|\le \|A\|$.
 This finishes the proof of the claim.
 
 Since $(\hat W_{i})_{i\in I}$ is a partition of $f^{-1}(Z^{c})$
 we have  $$B_{i'}\nu(\hat W_{i})=\left\{\begin{array}{cc} B_{i} &i=i'\\0 &else  \end{array} \right.\ .$$
Furthermore, by   \cite[Lem. 4.15]{cank}  the object $N$ is the AV-sum of the pairwise orthogonal 
family of images $(\nu(\hat W_{i})N)_{i\in I}$.
Note that $B_{i}$ is a finite  partial sum of \eqref{sfggerwggfs}.
In order to show the assertion of the lemma  it suffices to show     (see e.g.   \cite[Sec. 3]{cank}) that
$$\sup_{F\subseteq I} \|\sum_{i\in F}  B_{i} B_{i}^{*} \|\le n\|A\|^{2} \ ,$$
where $F$ runs over the finite subsets of $I$.

%
%
By construction  the propagation of the matrix $(B_{x',x})$ is bounded by $U$. Thus for every $i$ in $I$ and $x'$ in $X$,
if $ v'(x')^{*}B_{i}\not=0$, then $x'\in U[\hat W_{i}] \subseteq W_{i}$.
 This implies  $B_{i}=\nu'(W_{i}) B_{i}$. We get   the operator inequality
$$B_{i} B_{i}^{*} =\nu'(W_{i}) B_{i} B_{i}^{*}\nu'(W_{i})  \le \nu'(W_{i})  \|B_{i} B_{i}^{*}\|\le  \nu'(W_{i}) \|A\|^{2} \ .$$
This gives 
$$ \|\sum_{ i\in F}  B_{i} B_{i}^{*} \|\le
\|A\|^{2} \|\sum_{i\in F} \nu'(W_{i}) \| \le n\|A\|^{2}$$ since the family $\cW$ 
 has multiplicity bounded by $n$. 
 \end{proof}
\cref{rekgopewergwerfwerf} finishes the proof of \cref{kogpwreregwe91}.
\end{proof}

Recall that $\bC$ is a $C^{*}$-category with $G$-action
 which admits all AV-sums and such that $M\bC$ is idempotent complete.
 Let $H:\Add_{\C}\to \cC$ be a homological functor for $\C$-linear additive categories \cref{kopwegrfw}.
 Recall \cref{hlkepthetrgege}.
\begin{kor}\label{rgrthegkeporg}
 The  equivariant coarse homology functor
$H \cX^{G,\ctr}_{\bC} $ admits the structure of transfers  \begin{equation}\label{vsdfvsdfvsdfvdfvs}\tr:t^{*}H \cX^{G,\ctr}_{\bC} \to s^{*}H \cX^{G,\ctr}_{\bC} :{}^{G}\BCov^{s\fdsc}\to \cC\ .
\end{equation}  
\end{kor}

We now assume that $\bC$ has the trivial $G$-action. We consider the functors
$$\bV^{G,\ctr}_{\bC}:G\BC\to \Add_{\C}\ , \quad \bV^{\ctr}_{\bC}:\BC\to \Add_{\C}\ .$$
\begin{theorem}\label{igopetgertgertgt}The transfer from \eqref{vsdfvsdfvcsdfc} restricts to a 
  natural transformation
\begin{equation}\label{wergegwerf1}\tr:t^{*}\ell \bV^{\ctr}_{\bC}\stackrel{}{\to} s^{*}\ell \bV^{G,\ctr}_{\bC}:G\BCov^{s\fdsc}\to \Add_{\C,2,1}\ .
\end{equation}
which becomes an equivalence after further restriction $G\BCov^{\fdsc}$.
\end{theorem}
\begin{proof}
We reuse that argument for \cref{pkgtrgtegw}. 
The construction of the natural transformation in \eqref{wergegwerf1} 
by restriction  of the one from  \cref{kogpwreregwe91} is the same as in  the proof of \cref{pkgtrgtegw}. 
In order to show that it is an equivalence, we again    construct   for every branched coarse $G$-covering 
 $(f:X\to Y,\cZ)$  an inverse 
\begin{equation}\label{werfewrfwerwerffwrefw}r:\bV^{G,\ctr}_{\bC}(X,f^{-1}(\cZ))\to  \bV^{\ctr}_{\bC}(Y,\cZ)\ .
\end{equation}
The construction of $r$ on the level of objects is the same as in \cref{pkgtrgtegw}.
The value $A:(M,\mu)\to (M',\mu')$ of $r$ on a morphism $B:(N,\sigma,\nu)\to (N',\sigma',\nu')$ was defined in terms of a matrix $(A_{y',y})$ in \eqref{werfwergwg}. 
 In the present case we must again must make sure that this matrix determines a bounded operator.

Let $U$ denote the propagation of $B$ and set $V:=f(U)$.
For every $y $ in $Y$ we define $$A_{y}:=\sum_{y'\in Y} u'(y')A_{y',y} :N(s(y))\to M\ .$$
This sum has finitely many non-zero summands since the matrix $(A_{y',y})$ has propagation $V:=f(U)$ 
and there are only finitely may $y'$ in $V[y]$ with $0\not=A_{y',y}$.
We must show that
\begin{equation}\label{gwergwerf} \sum_{y\in Y} A_{y}u(y)^{*}
\end{equation}
converges and defines a bounded morphism $A:M\to M'$

The construction \cref{weiogowegrefwfwerfwf1} provides  a member $Z$ of $\cZ$, a covering $\cW=(W_{i})_{i\in I}$ with $\tilde V:=\bd(\cW)\in \cC_{Y}$ and $n:=\mult(\cW)<\infty$, and
a partition $\hat \cW=(\hat W_{i})_{i\in I}$ of $Z^{c}$ such that $V(W_{i})\subseteq \hat W_{i}$ for all $i$ in $I$   and the parallel transport at scale $\tilde V$ is defined on $Z^{c}$.

\begin{lem}\label{kohpertegtrg}
The sum  in \eqref{gwergwerf} converges and defines $A:M\to M'$ in $M\bC$ with $$\|A\|\le \sqrt{n}\|B\|\ .$$ 
\end{lem}
\begin{proof}
 We fix $i$ in $I$ and 
  define 
$$A_{i}:=  \sum_{y\in \hat W_{i}} A_{y}u(y)^{*}\ .$$
 Since $\hat W_{i}$ is $\tilde V$-bounded the sum has finitely many non-zero terms. 
 We claim that $\|A_{i}\|\le \|B\|$. 
 We fix a base point $y_{i}$ in $\hat W_{i}$  and choose a section $t:W_{i}\to f^{-1}(W_{i})$ uniquely determined by the condition
  that $t(y')$ in $f^{-1}(y')$  is the parallel transport of $s(y_{i})$. This is possible since 
 $\hat W_{i}$ is $\tilde V$-bounded and the parallel transport at scale $\tilde V$ is defined on $Z^{c}$.
 We define the function $h:W_{i}\to G$ such that $h(y') t(y')=s(y')$.
  We further define
$$a_{i}:=\sum_{y\in \hat W_{i}}   \sigma^{-1}_{h(y)}   v(s(y))u(y)^{*}  :M \to N\ , \quad b_{i}:=\sum_{y'\in W_{i}}   u'(y') v'(s(y'))^{*}\sigma'_{h(y')} :  N'\to M'\ .$$ 
Since $W_{i}$ and $\hat W_{i}$ are bounded these sums have finitely many non-zero terms.
Since $s$ is injective these maps are  partial isometries. More precisely,
$a_{i}$ is an isometry between $\mu(\hat W_{i})$ and $\nu(t(\hat W_{i}))N$, and
$b_{i}$ is an isometry from $\nu'(t(W_{i}))N' $ to $\mu'(W_{i})M'$. 
We  observe that
$A_{i}=b_{i}Ba_{i}$ since for every $y$ in $Z^{c}$  one can check that
$A_{y}=b_{i}Ba_{i} u(y)$. 
We further conclude that $\|A_{i}\|\le \|B\|$.

 Since $(\hat W_{i})_{i\in I}$ is a partition of $Z^{c}$
 we have  $$A_{i'}\mu(\hat W_{i})=\left\{\begin{array}{cc} A_{i} &i=i'\\0 &else  \end{array} \right.\ .$$
Furthermore, by   \cite[Lem. 4.15]{cank}  the object $M$ is the AV-sum of the pairwise orthogonal 
family of images $(\mu(\hat W_{i})M)_{i\in I}$.
Note that $A_{i}$ is a finite  partial sum of \eqref{gwergwerf}.
In order to show the assertion of the lemma it suffices to show that
$$\sup_{F\subseteq I} \|\sum_{i\in F}  A_{i} A_{i}^{*} \|\le n\|B\|^{2} \ ,$$
where $F$ runs over the finite subsets of $I$.

Since the   matrix $(A_{y',y})$ is  $V$-controlled and $V[\hat W_{i}]\subseteq W_{i}$   we have the equality
$A_{i}=\mu'( W_{i})  A_{i}$ for every $i$ in $I$.
 We furthermore have the operator inequality
 $$A_{i}A_{i}^{*}=\mu'(  W_{i})    A_{i}  A_{i}^{*}\mu'( W_{i})  \le \mu'( W_{i})  \|A_{i}\|^{2}\ .$$
Using $n=\mult(\hat \cW)$ we get  for every finite subset $F$ of $I$
$$\|\sum_{i\in F}A_{i}A_{i}^{*}\|^{2}\le \|B\|^{2} \|\sum_{i\in F} \mu'( W_{i})  \| \le n\|B\|^{2}$$ which implies the desired estimate.
\end{proof} \cref{kohpertegtrg} finishes the proof of the theorem.
 \end{proof}

Recall that $\bC$ is a $C^{*}$-category 
 which admits all AV-sums and such that $M\bC$ is idempotent complete.
 Let $H:\Add_{\C}\to \cC$ be a homological functor for $\C$-linear additive categories \cref{kopwegrfw} 
 and consider the functors 
 $$H\cX^{G,\ctr}_{\bA}:G\BC\to \cC\ , \quad H\cX_{\bC}^{\ctr}:\BC\to \cC\ .$$
Recall 
  \cref{hlkepthetrgege}.


\begin{kor}\label{kophertgegertg}
The  functors $H \cX^{G,\ctr}_{\bC}$ and $  H \cX^{\ctr}_{\bC}$ are related by a natural transfer transformation
\begin{equation} 
\tr:t^{*}H \cX^{\ctr}_{\bC}\to 
s^{*} \cX^{G,\ctr}_{\bC}:G\BCov^{s\fdsc}\to \cC
\end{equation}
which becomes an equivalence after further restriction to $G\BCov^{\fdsc}$.
\end{kor}
%

\begin{kor}
The transfers or transfer equivalences are compatible with natural transformations between coarse homology theories induced by natural transformation between homological functors.
\end{kor}

\begin{ex}  Let $\cL$ be any $\C$-algebra. The following square of functors on ${}^{G}\BCov^{s\fdsc}$ commutes
$$\xymatrix{t^{*}K^{\alg}\cZ_{\cL}\cX^{G,\ctr}_{\bC}\ar[r]\ar[d]^{\tr} &t^{*}K^{\alg}H\cZ_{\cL}\cX^{G,\ctr}_{\bC} \ar[d]^{\tr} \\ s^{*}K^{\alg}\cZ_{\cL}\cX^{G,\ctr}_{\bC} \ar[r] &s^{*}K^{\alg}H\cZ_{\cL}\cX^{G,\ctr}_{\bC} } \ ,$$ 
where the horizontal maps are induced by the natural transformation $K^{\alg}\to K^{\alg}H$.
\hB
 \end{ex}

\section{The transfer for topological coarse $K$-homology}\label{lkhperthtgretg}

In this section we show that the  equivariant topological coarse $K$-homology $K\cX^{G}_{\bC}$ 
constructed in \cite{coarsek} has transfers or transfer equivalences for branched coarse coverings under finite asymptotic dimension assumptions.

Let $\bC$ be a   $C^{*}$-category  with $G$-action which  admits all AV-sums  and whose multiplier category $M\bC$ is   idempotent complete.  We then have the functor
$$\bV^{G}_{\bC}:G\BC\to \Ccat$$
which associates to every $G$-bornological coarse space the $C^{*}$-category
$\bV^{G}_{\bC}(X)$ of equivariant locally locally finite objects in $\bC$ introduced in \cite[Sec. 5]{coarsek} \footnote{In the notation of the reference  \cite{coarsek} the $C^{*}$-category     $\bV^{G}_{\bC}(X)$ from the present paper  denoted by $\overline{\bC}^{G,\ctr}_{\lf}$.}.

Note that the involution on $M\bC$ induces an involution on $\bV^{G,\ctr}_{\bC}$  described in  \cref{kogpegregwerf}  so that $\bV^{G,\ctr}_{\bC}$ becomes a $\C$-linear $*$-category. It further inherts a $C^{*}$-norm from $M\bC$.
By an inspection of the definitions in   \cref{kogpegregwerf}  and  \cite[Sec. 5]{coarsek}  we have an equivalence  
 \begin{equation}\label{vwercecvfsdv}\bV^{G}_{\bC}:=\overline{\bV^{G,\ctr}_{\bC}}\ ,
\end{equation} 
i.e. $  \bV^{G}_{\bC}(X)$ is the completion of $\bV^{G,\ctr}_{\bC}(X)$ in the  $C^{*}$-norm induced from $M\bC$.

Let $(f:X\to Y,\cZ)$ be a   branched coarse covering relative to the big family $\cZ$.
  \begin{ddd}\label{jgoiertbgdbgf1}\mbox{}
\begin{enumerate}
\item By   $s\fadim$ we denote the condition that the source   $(X,f^{-1}(\cZ) )$ eventually has finite asymptotic dimension. 
\item By   $\fadim$ we denote the condition that  the target
$(Y,\cZ)$  eventually has finite asymptotic dimension.  
\end{enumerate}
\end{ddd}

Note that eventually finite asymptotic dimension of he target $ (Y,\cZ)$  implies by \cref{woigwgwegferfrefw} that
also the source  $(X,f^{-1}(\cZ))$  eventually has  finite asymptotic dimension.

We let $\ell:\Ccat \to \Ccat_{2,1}$ denote the Dwyer-Kan localization at the unitary equivalences.
We realize  $\Ccat_{2,1}$ by the strict $2$-category of unital $C^{*}$-categories, unital functors and unitary equivalences.

\begin{theorem}\label{kogpwreregwe913}
We have a canonical natural transformation
\begin{equation}\label{vfsdvcsdvsfdv}\tr:t^{*}\ell \bV^{G}_{\bC}\to s^{*}\ell \bV^{G}_{\bC} :{}^{G}\BCov^{s\fadim}  \to \Ccat_{2,1}\ .
\end{equation}
\end{theorem}
\begin{proof}
We reuse the proof of \cref{kogpwreregwe91}. 
Instead of a quotient by a Karoubi filtration we now have
$$  \bV^{G}_{\bC}(X,\cY)=\frac{  \bV^{G}_{\bC}(X)}{  \bV^{G}_{\bC}(\cY\subseteq X)}$$
where $\bV^{G}_{\bC}(\cY\subseteq X)$ is the ideal  in $\bV^{G}_{\bC}(X)$ obtained by completing the Karoubi filtration $\bV^{G,\ctr}_{\bC}(\cY)$. Note that $ \bV^{G}_{\bC}(\cY\subseteq X)\cap   \bV^{G,\ctr}_{\bC}(X)= \bV^{G,\ctr}_{\bC}(\cY)$.
 We can therefore define the transfer
 on the dense wide subcategory  $  \bV^{G,\ctr}_{\bC}(X,\cY)\subseteq   \bV^{G}_{\bC}(X,\cY)$
 and then extend it by continuity.
 
Assume that $(X,f^{-1}(\cZ))$ eventually has asymptotic dimension $n-1$.
We take this integer to bound the multiplicities of the coverings $\cW$ used in the proof of  \cref{rekgopewergwerfwerf}.
  By \cref{rekgopewergwerfwerf} the norm of the transfer of a controlled morphism $A$ is then bounded by $\sqrt{n}\|A\|$.
Thus the norm of the uncompleted transfer map $\tr_{f}:\bV^{G,\ctr}_{\bC}(Y,\cZ)\to\bV^{G,\ctr}_{\bC}(X,f^{-1}(\cZ))   $ is bounded by $\sqrt{n}$.  It therefore extends by continuity to the desired transfer $\tr_{f}:\bV^{G }_{\bC}(Y,\cZ)\to\bV^{G }_{\bC}(X,f^{-1}(\cZ))$ between the completed categories.

We finally observe that the morphism $t_{\Phi}$ in \eqref{sdbsdfbkopsgdb} is unitary by construction.
So the fillers of the squares are implemented by unitary natural isomorphisms.
\end{proof}

 We consider the composition   \begin{equation}\label{ojoijiorfwrf}K\cX^{G}_{\bC}:=K^{\topp}\circ \bV^{G}_{\bC}:G\BC\to \Mod(KU)\ .
\end{equation} The following is taken from \cite{coarsek}.
\begin{theorem}\label{lrgerhrtgerg}
The functor $K\cX^{G}_{\bC}:G\BC\to \Sp$ is an equivariant coarse homology theory. In addition  it is continuous, strong and strongly additive.
\end{theorem}

  Since the topological $K$-theory functor $K^{\topp}:\Ccat\to \Mod(KU)$
sends  unitary equivalences to equivalences it has a factorization
$$ \xymatrix{\Ccat\ar[rr]^{K^{\topp}}\ar[dr]^{\ell}&&\Mod(KU)\\&\Ccat_{2,1}\ar@{-->}[ur]^{K^{\topp}_{2,1}}&}\ .$$
We conclude   an equivalence of functors
\begin{equation}\label{gwrewfwerfwf1}K\cX^{G}_{\bC}\simeq K^{\topp}_{2,1}\ell \bV_{\bC}^{G}:G\BC\to \Mod(KU)\ .
\end{equation} 
 Therefore \cref{kogpwreregwe913} immediately implies:
\begin{theorem}\label{okjgpwerergw9} The functor $K\cX^{G}_{\bC}$ admits the structure of transfers 
\begin{equation} 
tr:t^{*}K\cX^{G}_{\bC}\to s^{*}K\cX^{G}_{\bC}:{}^{G}\BCov^{s\fadim}\to \Sp\ .
\end{equation}
\end{theorem}

We now assume that $\bC$ has the trivial $G$-action. We consider the functors
$$\bV^{G}_{\bC}:G\BC\to \Ccat\ , \quad \bV_{\bC}:\BC\to  \Ccat\ .$$
\begin{theorem}\label{klgpwegregw9}
The transfer from \eqref{vfsdvcsdvsfdv} restricts to a transfer \begin{equation}\label{wergegwerf}\tr:t^{*}\ell \bV_{\bC}\stackrel{ }{\to} s^{*}\ell \bV^{G}_{\bC}:G\BCov^{s\fadim}\to \Ccat_{2,1}\ .
\end{equation}
which becomes an equivalence after further restriction to $G\BCov^{\fadim}$.
\end{theorem}
\begin{proof}
We reuse the proof of \cref{igopetgertgertgt}.
Let $(f:X\to Y,\cZ)$ be in $G\BCov^{\fadim}$.
Assume that $(Y,\cZ)$ eventually has asymptotic dimension $n-1$. Then we can choose the coverings 
in the construction of the inverse $r$  in \eqref{werfewrfwerwerffwrefw}  on the dense subcategory
$\bV^{G,\ctr}_{\bC}(X,f^{-1}(\cZ))$ of $\bV^{G}_{\bC}(X,f^{-1}(\cZ))$  all with multiplicity $n$.
By  \cref{kohpertegtrg} the functor $r$ in  \eqref{werfewrfwerwerffwrefw} is bounded by $\sqrt{n}$ and therefore extends by continuity to
   $r:\bV^{G}_{\bC}(X,f^{-1}(\cZ))\to \bV_{\bC}(Y,\cZ)$. 

The transformations $\alpha$ and $\beta$ constructed at the end of the proof of \cref{pkgtrgtegw} are implemented by unitaries and therefore again witness that $r$ is an inverse of the transfer in \eqref{wergegwerf}. 
\end{proof}

\begin{rem}
In the special case of a asymptotically faithful covering of a space of graphs an analogue of 
\cref{klgpwegregw9} was shown in \cite[Lem. 3.12]{Willett_2012I}. 
Instead of the finite asymptotic dimension assumption on the source in this reference one assumes the uniform operator norm
localization property  \cite[Def. 3.10]{Willett_2012I}
\hB
\end{rem}

\begin{rem}
In the proofs of  \cref{kogpwreregwe913} and \cref{klgpwegregw9} we have shown that the transfer functors
and its inverse were bounded by $\sqrt{n}$, where $n-1$ was the relevant asymptotic dimension.
As functors between $C^{*}$-categories are automatically contractions we a posteriori can conclude that the norms of these functors are bounded by $1$. \hB
\end{rem}

\begin{kor}\label{koheprthkoptrkgergrtgegrtgte}
The   functors $K\cX^{G}_{\bC}$ and $ K\cX_{\bC}$ are related by a natural transfer  transformation
\begin{equation}\label{gewrferfwerfw}
\tr:t^{*} K\cX_{\bC}\to s^{*}K\cX^{G}_{\bC}:G\BCov^{s\fadim}\to \Sp
\end{equation} which becomes an
equivalence after further restriction to $G\BCov^{\fadim}$.
\end{kor}

\begin{rem} It follows from the example discussed in \cref{lkthperthergetrgtre}  that the 
natural transfer  transformation \eqref{gewrferfwerfw} 
is not an equivalence in general if we drop the finite asymptotic dimension assumption on the domain. \hB 
\end{rem}

Let $Y$ be a bornological coarse space with a big family $\cZ$ and $(M,\mu)$, $(M',\mu')$ be  $Y$-controlled
objects in $\bV_{\bC}(Y)$.
Following G. Yu we adopt the following notion.
\begin{ddd}
A morphism $A:(M,\mu)\to (M',\mu')$ in $\bV_{\bC}(Y)$ is called a ghost if
for every coarse entourage $U$ of $Y$ we have
$$\lim_{Z\in \cZ} \sup_{x,x'\in  Z^{c}}  \|\mu'(U[x'])A\mu(U[x])\|=0\ .$$
\end{ddd}

\begin{lem}\label{lphzrthh}
If $Y$ has finite asymptotic dimension, then 
a ghost represents zero in  the quotient $\bV_{\bC}(Y,\cZ)$.
\end{lem}
\begin{proof}
It suffices to show that
\begin{equation}\label{fwerfwerferwfferf}\limsup_{Z\in \cY} \| \mu'(Z^{c}) A\mu(Z^{c})\|=0\ .
\end{equation} 
We will use similar estimates as in the proof of    \cref{rekgopewergwerfwerf}.
Let $n-1$ bound the asymptotic dimension of $Y$.
We fix $\epsilon$ in $(0,\infty)$ and choose
a controlled $A':(M,\mu)\to (M',\mu')$ such that $\|A-A'\|\le \epsilon$.
Let $V$ be a bound of the propagation of $A'$.
We can choose a coarsely bounded partition $\hat \cW=(\hat W_{i})_{i\in I}$ of $Y$ 
and a coarsely bounded covering $\cW=(W_{i})_{i\in I}$ of multiplicity bounded by $n$ such that
$V[\hat W_{i}]\subseteq W_{i}$ for all $i$ in $I$. Then the sum
$A'=\sum_{i\in I} A'_{i}$ converges strongly, where $A'_{i}:=\mu'(W_{i})A'\mu(\hat W_{i})$.

Let $I_{Z}:=\{i\in I\mid W_{i}\cap  Z^{c}\not=\emptyset\}$.
For a finite subset $F$ of $I_{Z}$ we have 
$$\sum_{i,j\in F}A'_{i}A^{\prime,*}_{j}=  \sum_{i\in F} A'_{i}A_{i}^{\prime,*}\ .$$
Therefore $$  \| \sum_{i,j\in F}A'_{i}A^{\prime,*}_{j}\| \le \sup_{i\in I_{Z}}\|A'_{i}\|^{2} \|\sum_{i\in F} \mu'(W_{i})\|\le n  \sup_{i\in I_{Z}}\|A'_{i}\|^{2}$$
and hence
$$\|\sum_{i\in I_{Z}} A_{i} \|^{2} =\sup_{F\subseteq I_{Z}} \|  | \sum_{i,j\in F}A'_{i}A^{\prime,*}_{j}\| \le n \sup_{i\in I_{Z}}\|A'_{i}\|^{2}\ .$$
Since $A$ was a ghost we have
$\lim_{Z\in \cZ} \sup_{i\in I_{Z}}\|A'_{i}\|^{2}\le \epsilon^{2}$.
Hence
$\limsup_{Z\in \cZ} \|\sum_{i\in I_{Z}} A'_{i}\|\le  \sqrt{n}\epsilon$.
This implies
$\limsup_{Z\in \cZ} \|\mu'(Z^{c})A \mu(Z^{c})\| \le \epsilon+  \sqrt{n}\epsilon$.
Since $\epsilon$ was arbitrary we conclude \eqref{fwerfwerferwfferf}.
\end{proof}

\begin{kor}\label{lphetgerg}
If $(f:X\to Y,\cZ)$ is a branched coarse covering such that $X$ has finite asymptotic dimension, then
the transfer $\tr: \ell\bV_{\C}(Y,\cZ)\to \ell\bV_{\bC}(X,f^{-1}\cZ)$ annihilates ghost operators.
\end{kor}
\begin{proof}
The transfer preserves ghosts by construction. But ghosts represent zero in
$\bV_{\bC}(X,f^{-1}\cZ)$ by  \cref{lphzrthh}.

\end{proof}

\section{Traces on $C^{*}$-categories of controlled objects}\label{okvpwekvpewvdf}

 In this section we first recall the concept of a trace on a $C^{*}$-category and show that an everywhere defined trace 
 induces a  trace function on $K$-theory. The main result of this section is the construction of traces on the category of  equivariant $\bC$-controlled objects $\bV^{G}_{\bC}$ from a trace on the coefficient category $\bC $.

 We consider   a possibly non-unital $C^{*}$-category   $\bC$    in $\nCcat $.
  
%
%
%

 \begin{ddd}\label{okgpergrgwerf}\mbox{}\begin{enumerate} \item   A $\C$-valued trace on $\bC$ consists of a subset $\dom(\tau)$ of objects of $\bC$ which is closed under isomorphism in $M\bC$ and a collection of linear  maps
$$\tau:=(\tau_{C})_{C\in \dom(\tau)}\ , \qquad \tau_{C}:\End_{\bC}(C)\to \C$$ such that for every two morphisms
$f:C\to C'$  and $g:C'\to C$ with $C,C'\in \dom(\tau)$ we have $\tau_{C'}(f \circ g)=\tau_{C}(g\circ f)$. \item We say that   $\tau$  on $\bC$ is continuous if the maps $\tau_{C}$ are continuous for all   objects $C$ in $\bC$.
\item We say that $\tau$ is positive if $\tau_{C}(f)\ge 0$ whenever $f\ge 0$.
\end{enumerate}
\end{ddd}

We will usually omit the subscripts $C$ and denote the members $\tau_{C}$ of $\tau$ simply also by $\tau$.
 \begin{ex}
   Recall the inclusion functor $\incl:\nCalg\to \nCcat$.
 If $\tau$ is a trace on a $C^{*}$-algebra $A$, then  it is a trace on the one-object $C^{*}$-category $\incl(A)$ in $\nCcat$.  Assume that $A$ is unital. Then the trace  canonically extends to a  trace on 
 the category $\Hilb_{c}(A)$ of Hilbert $A$-modules and compact morphisms with domain $\dom(\tau)=
 \Hilb_{c}^{\fg,\proj}(A)$.
 
 If $\tau$ is continuous or positive, then so is its extension.
  \hB
 \end{ex}

 \begin{ex}\label{plkthrherthrt9}
The identity  is a trace  $\tr:\C\to \C $. It extends to the usual matrix trace $\tr$ on 
 the   $C^{*}$-category $\Hilb_{c}(\C)$  of  Hilbert spaces and compact operators with domain $\dim(\tr)=\Hilb^{\fin}(\C)$ the finite-dimensional Hilbert spaces. Note that $\tr$ is continuous and positive.
  \hB
 \end{ex}

  \begin{ex}\label{lpherthertgertg}

 If $G$ is a group and $[g]$ is a finite conjugacy class in $G$, then we   have a continuous trace   $$\tau_{[g]}:C_{r}^{*}(G)\to \C \ ,\qquad f\mapsto \sum_{g'\in [g]} f(g')  $$  on the reduced group $C^{*}$-algebra,
 where we consider $C^{*}_{r}(G)$ as a linear subspace  of the complex-valued functions on $G$.
We therefore get a   continuous extension of $\tau_{[g]}$ to the $C^{*}$-category category $\Hilb_{c}(C_{r}^{*}(G))$ with domain $\dom(\tau_{[g]})= \Hilb_{c}^{\fg, \proj}(C_{r}^{*}(G))$. The trace $\tau_{[e]}$ is also positive.
 \hB
 \end{ex}

Let $\bC$ be a $C^{*}$-category with a trace $\tau$.
Recall that $K^{\topp}$ denotes the topological $K$-theory functor for $C^{*}$-categories, \cref{kopgwergwerg9}.
\begin{prop}\label{kogpwergweferwf}
If 
  $\bC$ is unital and $\dom(\tau)=\bC$, then $\tau$  induces a 
homomorphism \begin{equation}\label{vwerfcsdfvs}\tau^{K^{\topp}}:K^{\topp}_{0}(\bC)\to \C \ .
\end{equation}
\end{prop}
\begin{proof}
Let $\bC \to \bC_{\oplus}$ be an additive completion.
By Morita invariance of $K^{\topp}$ (see \cref{tkohprthrethrth}) it induces an equivalence
$K^{\topp}(\bC)\stackrel{\simeq}{\to} K^{\topp} (\bC_{\oplus})$.
Furthermore, the trace $\tau$ canonically extends from $\bC$ to $\bC_{\oplus}$.

We can therefore in addition assume that $\bC$ is additive.
In this case we can use the description of $K^{\topp}_{0}(\bC)$ given in  \cite[Lem. 15.3]{cank}.
Every class    $p$  in $K^{\topp}_{0}(\bC)$ can be represented by  
a pair of projections $(P,\tilde P)$ in $\End_{\bC}(C)$ for some object $C$ of $\bC$.
 Furthermore,  $p=0$ if and only if 
    exists a partial isometry 
$U$ in $\End_{\bC}(C)$ such that   $UU^{*}=P$ and $U^{*}U=\tilde P$.
A second pair $(P',\tilde P')$ in $\End_{\bC}(C')$ represents the same
class $p$ if  and only if the class represented by $(P\oplus \tilde P',\tilde P\oplus P')$ in $\End_{\bC}(C\oplus C')$ represents $0$.
Finally, the  sum $p+p'$ in $K^{\topp}_{0}(\bC)$ is represented by
$(P\oplus  P',\tilde P\oplus \tilde P')$ in $\End_{\bC}(C\oplus C')$.

We define the homomorphism $$\tau^{K^{\topp}}:K^{\topp}_{0}(\bC)\to \C \ ,\qquad \tau^{K^{\topp}}(p):=\tau(P)-\tau(\tilde P) \ .$$
Using the facts above it is straightforward to check that
$\tau^{K^{\topp}}$ is  well-defined  and a homomorphism.
\end{proof}
%
%
%
%
%
%
%
%
%

\begin{ex}
 The
homomorphism $$\tr^{K^{\topp}}:K^{C^{*}}(\C)\simeq K^{\topp}_{0}(\Hilb^{\fin}(\C))\to \C$$ is the dimension homomorphism and takes values in the integers. \hB
\end{ex}

\begin{ex}Let $A$ be a unital $C^{*}$-algebra with a   trace $\tau$
and   consider its extension to $\Hilb_{c}^{\fg,\proj}(A)$.
The inclusion $\incl(A)\to \Hilb_{c}^{\fg,\proj}(A)$ of unital  $C^{*}$-categories  is a Morita equivalence. Let $K^{C^{*}}$ denote the topological $K$-theory functor for $C^{*}$-algebras. Then  
\begin{equation}\label{gwergkop0werf}K^{C^{*}}(A)\stackrel{\eqref{gwerfwerfw}}{\simeq} K^{\topp}(\incl(A))\stackrel{\simeq}{\to} K^{\topp}( \Hilb_{c}^{\fg,\proj}(A))\ .
\end{equation}
  The induced homomorphism
$$K^{C^{*}}_{0}(A)\stackrel{\eqref{gwergkop0werf}}{\cong} K^{\topp}_{0}(\Hilb_{c}^{\fg,\proj}(A))\stackrel{\tau^{K^{\topp}}}{\to}  \C$$
coincides with the   homomorphism on $C^{*}$-algebra $K$-theory
 classically associated to $\tau$.
  \hB
\end{ex}

Let $\bC$ be a $C^{*}$-category with a $G$-action and a  trace $\tau$.

\begin{ddd}\label{hkoprhertgrtgretge}
We say that $\tau$ is $G$-invariant if  $\dom(\tau)$ is $G$-invariant and $\tau_{gC}(gA)=\tau_{C}(A)$ for every morphism  $A$ in $\bC$ and $g$ in $G$.
\end{ddd}

Let $\bC$ be a $C^{*}$-category with a $G$-action and an invariant trace $\tau$.
We assume that $\bC$  admits all $AV$-sums and
 is idempotent complete.  We further assume that the unital objects belong to the domain of $\tau$, i.e., that $\bC^{u}\subseteq \dom(\tau)$.
 
 For a   $G$-bornological coarse space $X$ we   consider the $C^{*}$-category $\bV_{\bC}^{G}(X)$ of equivariant $X$-controlled objects in $\bC$ described by  \eqref{vwercecvfsdv}.



We consider and object $(H,\rho,\mu)$ in $\bV^{G}_{\bC}(X)$.
For a point $x$ in $X$ we choose an   isometry $u(x):C\to H$ in $M\bC$ representing the image of $\mu(\{x\})$. 
For $A$ in $\End_{\bV^{G}_{\bC}(X)}((H,\rho,\mu))$   we  then set $A_{x,x}:=u(x)^{*}Au(x)$ in $\End_{\bC}(C)$.  
\begin{lem}
The map \begin{equation}\label{vwvevwev}\tau_{x}:\End_{\bV^{G}_{\bC}(X)}((H,\rho,\mu))\to \C\ , \quad A\mapsto  \tau(A_{x,x})\ . 
\end{equation}   is well-defined independent of the choices  and satisfies  \begin{equation}\label{bfdsvsdfveqwrvf}\tau^{G}_{gx} =\tau^{G}_{x} 
\end{equation}
  for every $g$ in $G$.
\end{lem}
\begin{proof}
 By local finiteness of $(H,\rho,\mu)$ the object $C$  belongs to   $\bC^{u}$.
Therefore  the complex number $\tau(A_{x,x}) $ is defined. 

If $u'(x):C'\to H$ is another representative of the image of $\mu(\{x\})$ and $A_{x,x}'= u'(x)^{*}Au'(x)$, then $w:=u(x)^{*}u'(x):C'\to C$ is a unitary isomorphism and
$A_{x,x} =w A_{x,x}' w^{*}$.  Using the trace property of $\tau$ this implies that $\tau(A_{x,x})=\tau(A_{x,x}')$. 
Consequently, $\tau_{x}(A)$ is well-defined.
 
  We claim that  
  \begin{equation}\label{bfdsvsdfveqwvf}\tau^{G}_{gx}(A)=\tau^{G}_{x}(A)
\end{equation}
  for every $g$ in $G$.
By $G$-invariance of $A$ and $\mu$ we have
$\rho_{g}A \rho_{g}^{-1} =gA $
and $ \rho_{g}^{-1} \ g\mu(\{x\})  \ \rho_{g}=  \mu(\{gx\}) $.
If $u(x):C\to H$ represents the image of $\mu(\{x\})$,
then $\rho_{g}^{-1}\ gu(x)\  : gC\to H$ represents the image of $\mu(\{gx\})$.
With this choice \begin{equation}\label{erfwerfwefwefwe}A_{gx,gx}= gu(x)^{*}\  \rho_{g} A   \rho_{g}^{-1}\ gu(x)  = g(A_{x,x})     \ .
\end{equation}
 Using the   $G$-invariance of $\tau$ we conclude \eqref{bfdsvsdfveqwvf}.
 \end{proof}
 
 A $G$-invariant subset $Y$ of $X$ is called $G$-bounded if  there exists a bounded subset $B$ of $Y$ such that $GB=Y$. 

\begin{ddd} By $\pi_{0}^{\crs}(X)$ we denote the $G$-set of coarse components of $X$.  \end{ddd}
We consider a $G$-orbit $i$ in $\pi^{\crs}_{0}(X)/G$ and let $X_{i}$ be the corresponding $G$-invariant subspace of $X$.

Consider an object $(H,\rho,\mu)$ of $\bV_{\bC}^{G}(X)$.     We assume that $\supp(H,\rho,\mu)\cap X_{i}$ is $G$-bounded. 
 We can choose a bounded set-theoretic fundamental domain
$F$  for the $G$-action on $\supp(H,\rho,\mu)\cap X_{i}$. For every $x$ in $X$ we let $G_{x}$ denote the stabilizer $G_{x}$ of $x$ in $G$.
For   $A$ in $\End_{\bV^{G}_{\bC}(X)}((H,\rho,\mu))$ we define the complex number
\begin{equation}\label{gtrewfwerfw}\tau^{G}_{i,(H,\rho,\mu)}(A):=\sum_{x\in F}  \frac{1}{|G_{x}|} \tau_{x}(A)\ ,
\end{equation} 
where we interpret $1/\infty$ as $0$.

\begin{lem}\label{jkopwergergwefwre} If $G$ acts   with finite stabilizers and  $\tau$ is continuous, then 
$ \tau^{G}_{X,(H,\rho,\mu)}$ is well-defined, independent of the choice of $F$. The family 
of maps $(\tau^{G}_{i,(H,\rho,\mu)})_{(H,\rho,\mu)}$  is a continuous trace $\tau^{G}_{i} $ on $\bV^{G}_{\bC}(X)$ with
$\dom(\tau^{G}_{i})$ consisting of all objects with the property that $\supp(H,\rho,\mu)\cap X_{i}$ is  $G$-bounded.
\end{lem}
\begin{proof}
By assumption the set $ \supp(H,\rho,\mu)$
is a $G$-invariant locally finite subset. In particular,
$\supp(H,\rho,\mu)\cap F$ is a finite set.
This  shows that the sum  in \eqref{gtrewfwerfw} is finite and that the formula defines a continuous  linear map
$$\tau^{G}_{i,(H,\rho,\mu)}:\End_{\bV^{G}_{\bC}(X)}((H,\rho,\mu))\to \C\ .$$

Since $|G_{x}|=|G_{gx}|$ for every $x$ in $X$ and $g$ in $G$
it follows from \eqref{bfdsvsdfveqwrvf} that $\tau_{i,(H,\rho,\mu)}^{G}(
A)$ does not depend on the choice of the fundamental domain $F$.

It remains to check the trace property.  From now on in order to simplify the notation  we drop the index indicating the controlled Hilbert space.  We consider $A:(H,\rho,\mu)\to (H',\rho',\mu')$ and $B:(H',\rho',\mu')\to (H,\rho,\mu)$. We fix a bounded   fundamental domain $F$  for $(\supp(H,\rho,\mu)\cup (H',\rho',\mu'))\cap X_{i}  $. We first assume that $A$ and $B$ are 
 $U$-controlled for some entourage $U$ of $X$.
Then for every $x$ in $X_{i}$ the sets $\{y\in X\mid A_{y,x}\not=0\}$ and $\{y\in X\mid B_{x,y}\not=0\}$ are finite subsets of $X_{i}$. The first  is contained in $
\supp(H',\rho',\mu')\cap U[\{x\}]$ which is finite since $\supp(H',\rho',\mu')$ is locally finite and $U[\{x\}]$ is bounded.
For the second we argue similarly.
We calculate   \begin{eqnarray*}
\tau^{G}_{i} (BA)&=&\sum_{x\in F,y\in X}  \frac{1}{|G_{x}|}\tau(B_{x,y}A_{y,x})\\&\stackrel{!}{=}&
  \sum_{x\in F,y\in F,g\in G} \frac{1}{|G_{x}||G_{y}|}  \tau(B_{x,gy}A_{gy,x})\\&\stackrel{!!}{=}&
   \sum_{x\in F,y\in F,g\in G} \frac{1}{ |G_{g^{-1}x}||G_{y}|} \tau(B_{g^{-1}x,y}A_{y,g^{-1}x})\\&\stackrel{!!!}{=}&
     \sum_{x\in F,y\in F,g\in G} \frac{1}{ |G_{g^{-1}x}||G_{y}|} \tau( A_{y,g^{-1}x}B_{g^{-1}x,y})\\&=&
 \sum_{x\in X,y\in F} \frac{1}{ |G_{y}|} \tau(A_{y, x}B_{ x,y} )\\&=&
 \tau_{i}^{G}(AB)
\end{eqnarray*}
All sums have finitely many non-zero terms.  
For the  step marked by $!$ it is important that  the groups $G_{y}$ are finite and that only points  $y$ in $ X_{i}$ contribute.
For $!!$ we use the $G$-equivariance of the trace and relations like \eqref{erfwerfwefwefwe}, and for $!!!$ the trace property. General $A,B$ can be approximated in norm by controlled morphisms. We conclude the trace property by continuous extension using the continuity of $\tau$.
\end{proof}

Recall that $i$ is in $\pi_{0}^{\crs}(X)/G$ and $X_{i}$ is the corresponding $G$-invariant subset of $X$.
 If $X_{i}$ is $G$-bounded, then $\dom(\tau^{G}_{i})=\bV^{G}_{\bC}(X)$.
 Note that $\bV^{G}_{\bC}(X)$ is unital and  $\pi_{0}K\cX^{G}_{\bC}(X)\cong K_{0}^{\topp}(\bV^{G}_{\bC}(X))$.
\begin{ddd}\label{hoertighprtgretgerg} If $X_{i}$ is $G$-bounded and $G$ acts with finite stabilizers on $X_{i}$, then we define
  $$\tau^{K,G}_{i}:\pi_{0}K\cX^{G}_{\bC}(X)\to \C $$
as the homomorphism associated  by \cref{kogpwergweferwf}  to the trace $\tau^{G}_{i}$.
\end{ddd}

If $G$ is trivial, then we omit it as a superscript.  Recall that we also omit the subscript $\bC$ in case that $\bC=\Hilb_{c}(\C)$ with the trivial $G$-action.  
In this case we consider the canonical trace from \cref{plkthrherthrt9}.
Finally, if $\pi_{0}^{\crs}(X)/G$ is a singleton, then we also omit the subscript $i$.

\begin{ex}
If $G$ is trivial and $X$ is bounded and coarsely connected, then
$\bV(X)\simeq \Hilb^{\fin}(\C)$ and $\tau^{K}:\pi_{0}K\cX(X)\to \C$ takes values in $\Z$. \hB
\end{ex}

\begin{ex}\label{kohperthertgetge}
We assume that $G$ is non-trivial and acts freely on $X$, and that $X$ $G$-bounded and coarsely connected. For any $x$ in $X$ the map $G\to X$, $g\mapsto gx$ induces a coarse equivalence
  $G_{can,min}\to X$. Under the isomorphism
$$\pi_{0}K\cX^{G}(X)\cong \pi_{0}K\cX^{G}(G_{can,min})\cong K^{C^{*}}_{0}(C_{r}^{*}(G))$$
the trace homomorphism  $\tau^{K,G}:\pi_{0}K\cX^{G}(X)\to \C$ corresponds to
the homomorphism $\tau_{[e]}^{K^{C^{*}}}:\pi_{0}K^{C^{*}}(C^{*}_{r}(G))\to \C$ induced by the trace $\tau_{[e]}:C^{*}_{r}(G)\to \C$ in 
\eqref{lpherthertgertg}.
\hB
\end{ex}
%

\begin{ex}
If $M$ is a compact Riemannian manifold with a Dirac operator $\Dirac$ of degree $0$.
Then $M$ is a bounded bornological coarse space with the metric structures.
We have the coarse index class
$\ind\cX(\Dirac)\in \pi_{0}K\cX(X)$. In this case 
 $\tau^{K}(\ind\cX(\Dirac))$ is the Fredholm index of $\Dirac$.  \hB
\end{ex}

\begin{ex}
If $M$ is a $G$-compact Riemannian $G$-manifold with an invariant generalized Dirac operator $\Dirac$ of degree $0$, then 
we have the equivariant coarse index index class $\ind\cX^{G}(\Dirac)\in \pi_{0}K\cX^{G}(X)$. 
Then  $\tau^{K,G}(\ind\cX^{G}(\Dirac))$ is the classical $L^{2}$-index of $\Dirac$ introduced in \cite{zbMATH03505915}.     \hB
\end{ex}

 \section{The algebraic $L^{2}$-index theorem}\label{gokpwegregwre9}
 
 Let $k$ be a commutative base ring and   $\Cat_{k}$ denote the category  of $k$-linear  categories.
  We consider  a $k$-linear    category  $\bA$. The following is the algebraic analogue of \cref{okgpergrgwerf}. 
    \begin{ddd} \label{kopherhetreg}A  trace on $\bA$ consists of  a subset $\dom(\tau)$ of objects of $\bA$ which is closed under isomorphisms  and a collection of  $k$-linear  maps
$$\tau:=(\tau_{C})_{C\in \dom(\tau)}\ , \qquad \tau_{C}:\End_{\bA}(C)\to k$$ such that for every two morphisms
$f:C\to C'$  and $g:C'\to C$ with $C,C'\in \dom(\tau)$ we have $\tau_{C'}(f \circ g)=\tau_{C}(g\circ f)$.
\end{ddd}

 We let $\Cat_{k}^{\tr}$ denote the category of pairs
 $(\bA,\tau)$ of a $k$-linear category $\bA$ and a   trace
 $\tau$. A morphism $(\bA',\tau') \to (\bA,\tau)$ is a  $k$-linear functor
 $\phi:\bA'\to \bA$ such that $\phi(\dom(\tau'))\subseteq \dom(\tau)$ and  $\tau_{\phi(C')}(\phi(A'))=\tau'_{C'}(A')$ for all objects $C'$ in $\dom(\tau')$ and endomorphisms $A'$  in $\End_{\bA'}(C')$.  
 
We have a forgetful functor $\cF_{\trc}:\Cat_{k}^{\tr}\to \Cat_{k}$ which forgets the traces. 
We let furthermore $\Cat_{k}^{\tr,\perf}$ denote the full category
of pairs $(\bA,\tau)$ with $\dom(\tau)=\bA$. 
 We consider a   functor $H:\Cat_{k}\to \Sp$.
 
\begin{ddd}\label{erthgerthertertgertt}
A tracing of $H$ is a natural transformation 
\begin{equation}\label{fdvsdfvsdffsc}\trc^{H}:\pi_{0}H\circ \cF_{\trc}\to \const_{k}:\Cat_{k}^{\tr,\perf}\to\Ab\ .
\end{equation}
\end{ddd}
We let $\tau^{H}:\pi_{0}H(\bA)\to k$ denote the evaluation of the tracing $\trc^{H}$ at $(\bA,\tau)$.
  

\begin{rem}
With the obvious modifications of definitions we can consider 
$$(\bC,\tau)\mapsto (\tau^{K^{\topp}}:\pi_{0}K^{\topp}(\bC)\to \C)$$  from \eqref{vwerfcsdfvs}
as a the natural transformation of functors from $\Ccat^{\tr,\perf}$   to $\Ab$ 
and therefore as a tracing $\trc^{K^{\topp}}$of the functor $K^{\topp}$. \hB
\end{rem}

\begin{ex} Recall that $\cZ:\Cat_{k}\to \Cat_{\Z}$ is the forgetful functor.
The algebraic $K$-theory functor $K^{\Cat_{\Z}}\cZ$ for $k$-linear categories is traced. 
The argument is analogous to the one from \cref{kogpwergweferwf}.
 We just replace  
the  word unitary by invertible and $U^{*}$ by $U^{-1}$. 
The naturality is obvious from the construction.

If $\Q\subseteq k$, then alternatively one can construct the tracing via cyclic homology. More precisely,
the tracing of homotopy $K$-theory from \cref{gokrjpwegerwwre} induces the tracing of algebraic $K$-theory $K^{\Cat_{\Z}}\cZ$ by restriction along the  comparison map \eqref{vdfsvewrvfsdvsdfvsd}.
 \hB
\end{ex}

%
%

%
%

\begin{ex}\label{kgopweewrfwerf}

 If $R$ is a unital $k$-algebra with a trace $\kappa$,  then
 the endofunctor $R\otimes_{k}-:\Cat_{k}\to \Cat_{k}$
 lifts to an endofunctor  $R\otimes_{k}-:\Cat^{\tr}_{k}\to \Cat^{\tr}_{k}$.
 To this end we note that a   trace $\tau$ on $\bA$ induces
 a  trace $\kappa\otimes \tau$ on $R\otimes_{k} \bA$ uniquely characterized by 
 $(\kappa\otimes \tau)(r\otimes A)=\kappa(r)\tau(A)$.

 Let now   $R$ be a possibly non-unital $k$-algebra with a 
    trace $\kappa:R\to k$.
 If $H$ is a traced homological functor (we only need split-exact) for $k$-linear  categories, then 
 the twist $H_{R}:\Cat_{k}\to \Sp$ is again traced. In order to construct this tracing we first 
   extend the trace to a trace $\tau^{+}$ on the unitalization
 $R^{+}\cong R\oplus k$ such that $\kappa^{+}_{|k}=0$ and $\kappa^{+}_{|R}=\kappa$.
 Then $H_{R^{+}}$ is traced. 
 We then define $\tau^{H_{R}}:\pi_{0}H_{R}(\bA)\to k$ as the restriction of
 $\tau^{H_{R^{+}}}:\pi_{0}H_{R^{+}}(\bA)\to k$
 along the inclusion    $\pi_{0}H_{R}(\bA)\to \pi_{0}H_{R^{+}}(\bA)$. See also \cref{koopherhegg}.\ref{hetggretg}.
   \hB
\end{ex}

If $\bA$ is a $k$-linear   category with a $G$-action, then 
$G$-invariance of a trace on $\bA$ is defined as in \cref{hkoprhertgrtgretge}.

We let $G\Add_{k}^{\tr,\sharp} $ denote the category of
idempotent  and sum complete compactly generated additive categories with $G$-action and with an invariant  trace
such that $\dom(\tau)$ contains all compact objects,
and equivariant compact-object and  trace preserving $k$-linear functors.
 Let $X$ be a $G$-bornological coarse space.
 
 \begin{lem} \label{jkopwergergwrefwre}  If $G$ acts  with finite stabilizers on $X$, then for every $i$ in $\pi_{0}^{\crs}(X)/G$
  the functor $ \bA\to \bV^{G}_{\bA}(X)$ from \cref{tlrkhpetrhertgetre} lifts to a functor
 $$ (\bA,\tau)\to (\bV^{G}_{\bA}(X),\tau^{G}_{i}):G\Add_{k}^{\tr,\sharp}  \to \Add^{\tr}_{k}\ ,$$
 where   $\dom(\tau^{G}_{i})$ consists of all objects $(M,\rho,\mu)$ of $\bV_{\bA}^{G}(X)$ such that
 $\supp(M,\rho,\mu)\cap X_{i}$ is $G$-bounded.
 \end{lem}
\begin{proof}
We repeat the construction from \cref{jkopwergergwefwre}. 
We can drop the steps discussing completions. 
The naturality of the construction  in $(\bA,\tau)$ is obvious.
\end{proof}
%
%
%
If $G$ acts with  finite stabilizers on $X$, then we have  a vector-valued trace \begin{equation}\label{fqpopoqwedqwedqwedqwede}  \tau^{G}_{X}:=(\tau^{G}_{i})_{i\in \pi^{\crs}_{0}(X)/G} :=\bV_{\bA}^{G}(X)\to \prod_{\pi^{\crs}_{0}(X)/G} k\end{equation}
defined on objects with coarse component-wise   bounded support.

Let $f:X\to Y$ be a branched coarse $G$-covering  relative to the big family $\cZ $.
We assume that the coarse components of $Y$ are bounded and that  the coarse $G$-components of $X$ are $G$-bounded. Note that $G$ acts freely on $X$ by assumption.

  We assume that   every member of $\cZ$  is a union of coarse components.
  We then introduce the following abbreviations:  
$$  T(X):= \colim_{Z\in \cZ}  \frac{\prod_{\pi_{0}^{\crs}(X)/G}k }{\prod_{\pi_{0}^{\crs}(f^{-1}(Z))/G} k}\ , \quad T(Y):= \colim_{Z\in \cZ}  \frac{\prod_{\pi_{0}^{\crs}(Y)}k }{\prod_{\pi_{0}^{\crs}(Z)} k}\ .$$
We have a pull-back   map  $$f^{*}:T(Y)\to T(X)\ .$$

\begin{lem}
We get induced every-where defined vector-valued traces  
$$\bar \tau^{G}_{X}:=([(\tau^{G}_{j})_{j\in  \pi_{0}^{\crs}(X)/G  }]:\bV_{\bA}^{G}(X,f^{-1}(\cZ)) \to T(X)\ , \quad  \bar \tau_{Y}:=[(\tau_{i})_{i\in \pi_{0}^{\crs}(Y)}]:\bV_{\bA}(Y,\cZ) \to T(Y)\ .$$
\end{lem}
\begin{proof}
We consider the second map. The argument for the first is similar.
Note that morphisms  $[A]$ in $ \bV_{\bA}(Y,\cZ) $ are represented by morphisms
$A:(M,\mu)\to (M',\mu')$ in $ \bV_{\bA}(Y)$, and $[A]=0$ if and only if $A$ factorizes
over an object coming from $\bV_{\bA}(\cZ)$. Let $Z$ be a member of $\cZ$ containing the support of this object. Then  indeed $\tau_{i}(A)=0$ for all $i$  in $\pi_{0}^{\crs}(Y) $ with $Y_{i}\subseteq Z^{c}$.
%
%
\end{proof}

\begin{rem}
Let $(f:X\to Y,\cZ)$ be a branched coarse covering relative to the big family $\cZ$. We assume that the members of  $\cZ$  are unions of coarse components.
Let $C$ be a coarse component of $X$. Then $f(C)$ is contained in a coarse component $D$ of $Y$, but in general is not equal to $D$.
 
Assume that there exists an entourage $V$ of $Y$ such that $D$ is already coarsely connected in $Y_{V}$ and
that the parallel transport at scale $V$ is defined on $D$. 
Then any two points in $D$ are connected by a $V$-path. Using the path-lifting
we can conclude that $f(C)\to D$ is surjective. 

Consequently, if $Y$ has a generating entourage $V$, then we can choose $Z$ in $\cZ$ such that for every coarse component 
$C$ in $f^{-1}(Z^{c})$ the image $f(C)$ is a coarse component of $Y$.
The latter property is an assumption in \cref{gouweriofjmwkemrfwklerf} below. \hB
\end{rem}

Recall that we  assume that the coarse components of $Y$ are bounded and that  the coarse $G$-components of $X$ are $G$-bounded.
By \cref{pkgtrgtegw} the transfer induces an equivalence
$$\tr_{f}:\bV_{\bA}(Y,\cZ)\stackrel{\simeq}{\to} \bV_{\bA}^{G}(X,f^{-1}(\cY))\ .$$ 
\begin{prop} \label{gouweriofjmwkemrfwklerf}
We assume  that the members of $\cZ$ are unions of  coarse components and that there exists  $Z$ in $\cZ$ such that $f$ sends
every coarse component of $f^{-1}(Z^{c})$ surjectively to a coarse component of $Z^{c}$.
Then the following square
  \begin{equation}\label{fqwedqewde} \xymatrix{  \bV_{\bA}(Y,\cZ)\ar[r]^-{\bar \tau_{Y} }\ar[d]_{\simeq }^{\tr_{f}} &T(Y) \ar[d]^{f^{*}} \\   \bV^{G}_{\bA}(X,f^{-1}(\cZ))\ar[r]^-{\bar \tau_{X}^{G}} &  T(X)} 
\end{equation} 
commutes.
\end{prop}
\begin{proof}
Let $[A]$ be in $\End_{\bV_{\bA}(Y,\cZ)}((M,\mu))$ be represented by $
A$ in $ \End_{\bV_{\bA}(Y)}((M,\mu))$ with 
propagation bounded by $V$. It gives rise to a matrix $A_{y',y}$.
Then the transfer $\tr_{f}([A])$ is represented by $B$ in $ \End_{\bV^{G}_{\bA}(X)}((N,\rho,\nu))$ 
with matrix $B_{x',x}=A_{f(x'),f(x)}$ given by \eqref{werfwerweg}. It is
defined on $f^{-1}(Z^{c})$ for a sufficiently enlarged member $Z$ of $\cZ$ depending on $V$. 

By construction, 
$\bar \tau_{Y}([A])$ is represented by the family
$(\tau_{i}(A))_{i\in \pi_{0}^{\crs}(Y)}$, where
$\tau_{i}(A)=\sum_{y\in Y_{i}} \tau(A_{y,y})$.
For $j$ in $\pi^{\crs}_{0}(X)/G$ let $F_{j}$ be the fundamental domain for the $G$-action chosen in the construction of $\tau^{G}_{j}$.
  Then $\bar \tau^{G}_{X}(B)$ is represented by the family
  $(\tau^{G}_{j} (B))_{j\in \pi_{0}^{\crs}(X)/G}$    and
  $\tau^{G}_{j} (B)=\sum_{x\in F_{j}} \tau(B_{x,x})$.

Assume now that $\pi_{0}^{\crs}(f)(j)=i$ and that   $Y_{i} \subseteq Z^{c}$. Then
the map  $f:F\cap X_{j}\to Y_{i}$  is a bijection. It is injective since $F$ is a fundamental domain and surjective by the assumption on $f$.
 We conclude that $ 
\tau^{G}_{j}(B)=\tau_{i}(A)$
since $\tau(B_{x,x})=\tau(A_{f(x),f(x)})$ for all $x$ in $F\cap X_{j}$.
This implies the assertion. 
\end{proof}

We assume that $H:\Add_{k}\to \Sp$ is a traced homological functor.
 Then we get  functions
$$\tau_{X}^{H,G}:\pi_{0}H\cX_{\bA}^{G}(X)\to \prod_{\pi_{0}^{\crs}(X)/G} k\ , \qquad 
\tau_{Y}^{H}:\pi_{0} H\cX_{\bA}(Y)\to \prod_{\pi_{0}^{\crs}(Y)} k\ .
$$ 
In detail, let $i$ be in $ \pi_{0}^{\crs}(X)/G$. Our assumptions on $(f:X\to Y,\cZ)$ imply that $\tau_{i}^{G}$ is an everywhere defined trace on $\bV_{\bA}^{G}(X)$.
Then the $i$th component of 
$\tau^{H,G}_{X}$ is given by $\tau^{H,G}_{i}$.
We furthermore get functions
 $$\bar \tau_{X}^{H,G}:\pi_{0}H\cX_{\bA}^{G}(X,f^{-1}(\cZ))\to T(X)\ , \qquad 
\bar \tau_{Y}^{H}:\pi_{0} H\cX_{\bA}(Y,\cZ)\to  T(Y)\ .
$$ 

\begin{kor}\label{kogpwergrefwferfw} The following square commutes:
 \begin{equation}  \xymatrix{\pi_{0}  H \cX_{\bA}(Y,\cZ)\ar[r]^-{\bar \tau_{Y}^{H}}\ar[d]^{\tr_{f}} &T(Y) \ar[d]^{f^{*}} \\ \pi_{0}H\cX_{\bA}^{G}(X,f^{-1}(\cZ))\ar[r]^-{\bar \tau_{X}^{H,G}} &  T(X)}\ .
\end{equation}
\end{kor}

Assume now  that $X\to Y$ is  a uniform $G$-covering. We assume that the coarse components of $Y$ are bounded and that  the coarse $G$-components of $X$ are $G$-bounded.
Then 
$\cO^{\infty}(X)\to \cO^{\infty}(Y)$ is  by \cref{okgpthertetg}.\ref{pklhtrthrteht} a branched coarse $G$-covering
with respect to the big family $\cO^{-}(Y)$.
 The following is our algebraic $K$-homology version of Atiyah's $L^{2}$-index theorem.
 Recall the  \cref{khphertgetrg} of the cone transfer $\tr_{\cO^{\infty}(f)}$.
\begin{theorem}[Algebraic $L^{2}$-index theorem]\label{koprhererge}
 The  diagram commutes:
\begin{equation}\label{grwegergwefwerf}
\xymatrix{ \pi_{1}H\cX_{\bA}^{G}(\cO^{\infty}(X))\ar[r]^-{\partial^{\cone}}  &\pi_{0}H\cX_{\bA}^{G}(X) \ar[r]^{\tau_{X}^{H,G}}& \prod_{\pi_{0}^{\crs}(X)/G}k \\  \pi_{1}H\cX_{\bA}(\cO^{\infty}(Y))\ar[u]^{\simeq}_{\tr_{\cO^{\infty}(f)}}\ar[r]^-{\partial^{\cone}} & \pi_{0}H\cX_{\bA}(Y) \ar[r]^{\tau_{Y}^{H}}& \prod_{\pi_{0}^{\crs }(Y)}k\ar[u]^{f^{*}}} \ .\end{equation}
\end{theorem}

\begin{proof}

We form  the $G$-uniform bornological space $\hat X$   as in \eqref{weirug9egrferw} and the squeezing space $\hat X_{h}$, the $G$-bornological coarse space described in \cref{kohperthtregtg}.
  We consider the cone $\cO^{\infty} (\hat X)$ with the coarsely excisive pair of invariant subsets
$$\tilde \cO^{-}(\hat X):=\bigsqcup_{n\in \nat} (-\infty,n]\times X\ , \qquad \tilde \cO^{+}(\hat X):=\bigsqcup_{n\in \nat} [n,\infty)\times X\ .$$
We then have a coarse equivalence
$$\hat X_{h}\to  \tilde \cO^{-}(\hat X)\cap \tilde \cO^{+}(\hat X)\ , \quad x_{n}\mapsto (n,x_{n})\ ,$$
where $x_{n}$ denotes the point $x$ in the $n$th component $X_{n}$ of $\hat X$.

Let $$\Yo:G\BC\to G\Sp\cX$$ denote the universal equivariant coarse homology theory \cite[Def. 4.9]{equicoarse}.
The Mayer-Vietoris boundary of the decomposition $(\tilde \cO^{-}(\hat X),\tilde \cO^{+}(\hat X))$ of $\cO^{\infty} (\hat X)$   gives a morphism
$$\hat \partial:\Yo \cO^{\infty} (\hat X)\to \Sigma \Yo \hat X_{h}$$
in $G\Sp\cX$.

Note that $\cO^{\infty}(\hat X)$ and $\hat X_{h}$ have a coarsely disjoint  decomposition (indexed by $\nat$) into copies of $\cO^{\infty}(X)$ or $X$, respectively. 
Using excision in coarse homology theory for every $n$ we have  projections $\cO^{\infty}(p_{n})$ and $p_{n}$ to the $n$th component, respectively. 
The following lemma follows from the naturality of the Mayer-Vietoris boundary  maps and the usual identification of the cone-boundary with  the Mayer-Vietoris boundary.
\begin{lem}\label{kogpwergrefwef} For every $n$ in $\nat$ the following square commutes:
$$\xymatrix{\Yo \cO^{\infty} (\hat X)\ar[r]^-{\hat \partial}\ar[d]^{\cO^{\infty}(p_{n})} &\Sigma \Yo\hat X_{h} \ar[d]^{p_{n}} \\ \Yo\cO^{\infty}(X)\ar[r]^-{\partial^{\cone}} & \Sigma \Yo X} \ .$$
\end{lem}
 
By \cref{kjthgprthhdh} the map $\hat f:\hat X_{h}\to \hat Y_{h}$ is a 
  branched coarse $G$-covering  with respect to the big  family $\hat \cZ$ from \eqref{regwerwerfwefwef}.
  We set $\hat \cY:=f^{-1}(\hat \cZ)$.

%
%

We consider an equivariant coarse homology theory $E^{G}:G\BC\to \cC$ and a
 coarse homology theory $E:G\BC\to \cC$. Recall \cref{wegjiowgwergrwewf}.
\begin{lem}\label{gjiwoegrefwf}
If $E^{G}$ and $E$ are related by a transfer, then the following diagram commutes:
$$\xymatrix{E \cO^{\infty}(\hat Y)\ar[r]^{\hat \partial}\ar[d]^{\tr_{\cO^{\infty}}(\hat  f)} &\Sigma E(\hat Y_{h})  \ar[r]&\Sigma E(\hat Y_{h},\hat \cZ) \ar[d]^{\tr_{\hat f_{h}}} \\ E^{G}\cO^{\infty}(\hat X)\ar[r]^{\hat \partial} &\Sigma E^{G}(\hat X_{h} ) \ar[r]&\Sigma E^{G}(\hat X_{h},\hat \cY) } \ .$$
\end{lem}
\begin{proof}
Using the definition of the cone transfer  \cref{khphertgetrg} on the left vertical part we expand the diagram as follows:
$$\hspace{-2cm}\xymatrix{E \cO^{\infty}(\hat Y) \ar[r]^{\hat \partial}\ar[d]^{\simeq}  &\Sigma E(\hat Y_{h})  \ar[rr]&&\Sigma E(\hat Y_{h},\hat \cZ) \ar[ddd]^{\tr_{\hat f_{h}}} \\  E (\cO^{\infty}(\hat Y),\cO^{-}(\hat Y))\ar[r]\ar[d]^{\tr_{\cO^{\infty}}(\hat f)}&E(\cO^{\infty}(\hat Y),\cO^{\infty}(\hat \cZ)\cup  \cO^{-}(\hat Y))\ar[d]_{}^{\tr_{\cO^{\infty}}(\hat f)} &E(\cO^{\infty}(\hat Y),\cO^{\infty}(\hat \cZ) )\ar[l]^-{\simeq}\ar[d]^{\tr_{\cO^{\infty}}(\hat  f)}\ar[ur]^{\hat \partial}&\\ E^{G} (\cO^{\infty}(\hat X),\cO^{-}(\hat X))\ar[r]& E^{G}(\cO^{\infty}(\hat X),\cO^{\infty}(\hat \cY)\cup \cO^{-}(\hat X)) &E^{G}(\cO^{\infty}(\hat X),\cO^{\infty}(\hat \cY) )\ar[l]^-{\simeq}\ar[dr]^{\hat \partial}& \\ \ar[u]_{\simeq  } E^{G}\cO^{\infty}(\hat X)\ar[r]^{\hat \partial} &\Sigma E^{G}(\hat X_{h} ) \ar[rr]&&\Sigma E^{G}(\hat X_{h},\hat \cY) } \ .$$
The two squares in the middle commute by the naturality of the transfer
for morphisms of branched coarse coverings which in this case just increase the big families.
The right middle square commutes since 
$$\xymatrix{\hat X_{h}\ar[rr]^-{x_{n}\mapsto (n,x)}\ar[d] &&\cO^{\infty}(\hat X) \ar[d] \\\hat Y_{h} \ar[rr]^-{y_{n}\mapsto (n,y)} && \cO^{\infty}(\hat Y)} $$
gives rise to a morphism of branched coarse coverings
with respect to the families $\hat \cZ$ and  $\cO^{\infty}(\hat \cZ)$, respectively,
and the naturality of the transfer  as a transformation between excisive functors.
The upper and lower hexagon commute since $E$ and $E^{G}$ are excisive.
%
%
%
%
%
%
\end{proof}

 In the following diagram we abbreviate $\pi_{i}H\cX^{G}_{\bA} $ by $H\cX_{i}^{G} $
 and similarly for $\pi_{i}H\cX_{\bA} $ by $H\cX_{i} $.

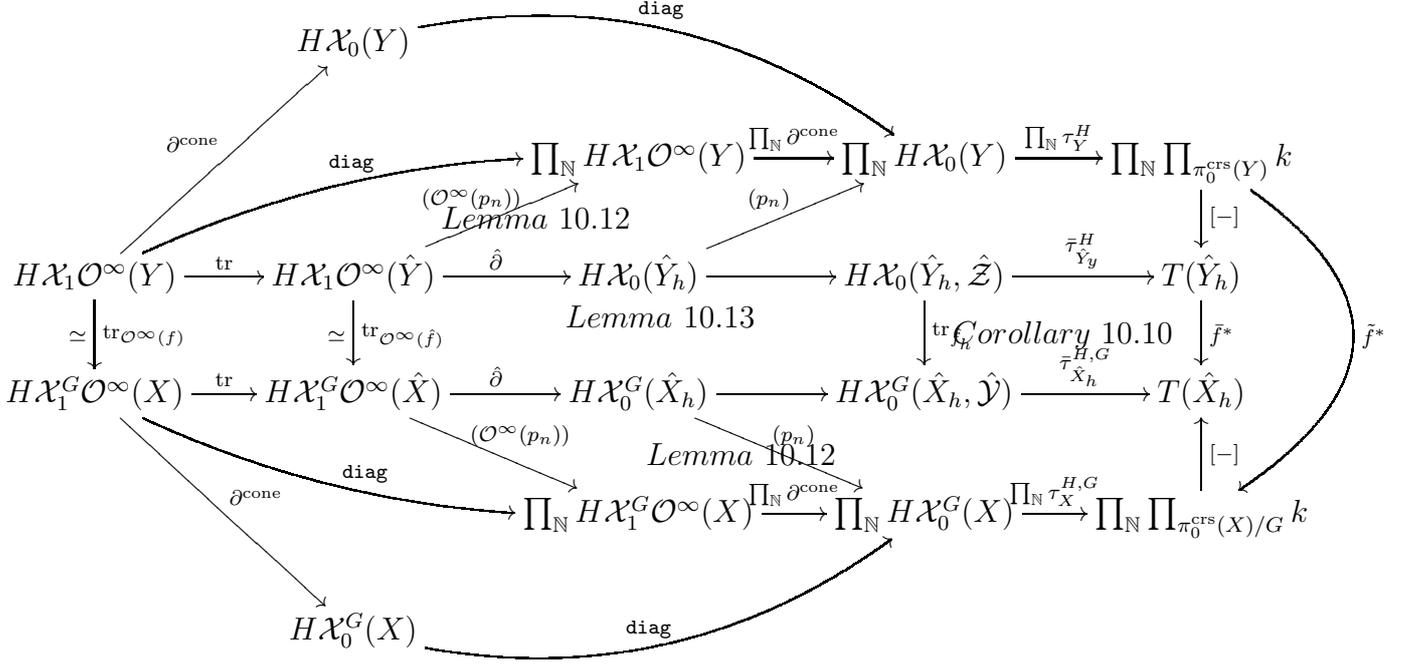
$$\hspace{-1.8cm}\xymatrix{&H\cX_{0}(Y) \ar@/^1cm/[drr]^{\diag}&&&\\&&\prod_{\nat}H\cX_{1}\cO^{\infty}(Y) \ar[r]^-{\prod_{\nat }\partial^{\cone}}&\prod_{\nat} H\cX_{0}(Y)\ar[r]^-{\prod_{\nat}\tau^{H}_{Y}}& \prod_{\nat} \prod_{\pi_{0}^{\crs}(Y)} k\ar[d]^{[-]}\ar@/^2cm/[ddd]^{\tilde f^{*}} \\H\cX_{1}\cO^{\infty}(Y)\ar@/^0.5cm/[rru]^{\diag}\ar[uur]^{\partial^{\cone}}\ar[r]^{\tr}\ar[d]^{\tr_{\cO^{\infty}(f)}}_{\simeq}&H\cX_{1} \cO^{\infty}(\hat Y)\ar@{}[rrd]^{\cref{gjiwoegrefwf}}\ar[ur]^{(\cO^{\infty}(p_{n}))}\ar[r]^{\hat \partial}\ar[d]_{\simeq }^{\tr_{\cO^{\infty}(\hat f)}} &H\cX_{0}(\hat Y_{h}) \ar@{}[u]^{\cref{kogpwergrefwef}} \ar[ur]^{(p_{n})} \ar[r]&H\cX_{0}(\hat Y_{h},\hat \cZ) \ar@{}[dr] | { \cref{kogpwergrefwferfw} }\ar[d]^{\tr_{\hat f_{h}}}\ar[r]^-{ \bar \tau^{H}_{\hat Y_{y}}}&T(\hat Y_{h})\ar[d]^{\bar f^{*}}\\ 
\ar@/_0.5cm/[drr]^{\diag}H\cX_{1}^{G}\cO^{\infty}(X)\ar[r]^{\tr}\ar[ddr]^{\partial^{\cone}}&\ar[dr]^{(\cO^{\infty}(p_{n}))}H\cX_{1}^{G}\cO^{\infty}(\hat X)\ar[r]^{\hat \partial} &H\cX_{0}^{G}(\hat X_{h} )\ar@{}[d]^{\cref{kogpwergrefwef}}\ar[dr]^{( p_{n})} \ar[r]&H\cX_{0}^{G}(\hat X_{h},\hat \cY)\ar[r]^-{\bar \tau^{H,G}_{\hat X_{h}}}&T(\hat X_{h})\\&& \prod_{\nat} H\cX^{G}_{1}\cO^{\infty}(X)\ar[r]^-{\prod_{\nat} \partial^{\cone}}& \prod_{\nat} H\cX^{G}_{0}(X)  \ar[r]^-{\prod_{\nat} \tau^{H,G}_{X}}& \prod_{\nat} \prod_{\pi_{0}^{\crs}(X)/G}k \ar[u]_{[-]}\\& H\cX_{0}^{G}(X)\ar@/_1cm/[urr]^{\diag}&&&} $$
All squares   commute for obvious or the    indicated reasons.
The commutativity of the left middle square is best seen by an inspection of the formulas for the two compositions on the level of categories of controlled objects.

 This diagram implies the commutativity of \eqref{grwegergwefwerf} as follows.
If we start with a class $H\cX_{1}\cO^{\infty}(Y)$, then its images in the right upper and lower corners are   constant families
with respect to the index in $\nat$.
Let $t$ in  $\prod_{\pi_{0}^{\crs}(Y)} k$ and $t'$ in $ \prod_{\pi_{0}^{\crs}(X)/G} k$ denote these constants. We let    $[(t)_{n}]$ be the constant family with value $t$.
The commutativity of the diagram says that $$[(f^{*}t)_{n}]=[\tilde f^{*}((t)_{n})]=\bar f^{*}[(t)_{n}]=[(t')_{n}]$$ in $T(\hat X_{h})$.
This implies that $t'=f^{*}t$.
 \end{proof}

We now state a version of \cref{koprhererge}  for the uncompleted topological coarse $K$-homology \eqref{fwqedewdqweddqe}.
Let $(\bC,\tau)$ be an idempotent complete $C^{*}$-category   admitting all $AV$-sums with a trace $\tau$ defined at least on $\bC^{u}$.

\begin{kor}\label{kohpertgertgertg} Assume that $Y$ has finite asymptotic dimension, finite uniform topological dimension, and a countably generated uniform structure. Then the following  diagram commutes:
\begin{equation}\label{grwegergwefrrrrwerf}
\xymatrix{ \pi_{1}H\cX_{\bC}^{G,\ctr}(\cO^{\infty}(X))\ar[r]^-{\partial^{\cone}}  &\pi_{0}H\cX_{\bC}^{G,\ctr}(X) \ar[r]^{\tau_{X}^{H,G}}& \prod_{\pi_{0}^{\crs}(X)/G}\C \\  \pi_{1}H\cX_{\bC}^{\ctr}(\cO^{\infty}(Y))\ar[u]^{\simeq}_{\tr_{\cO^{\infty}(f)}}\ar[r]^-{\partial^{\cone}} & \pi_{0}H\cX_{\bC}^{\ctr}(Y) \ar[r]^{\tau_{Y}^{H}}& \prod_{\pi_{0}^{\crs }(Y)}\C\ar[u]^{f^{*}}} \ .\end{equation}
\end{kor}
\begin{proof}
The assumption implies by \cref{kohpehtgeg}, \cref{woigwgwegferfrefw} and \cref{kophertgegertg}  that the transfers $\tr_{\cO^{\infty}(f)}$, $\tr_{\cO^{\infty}(\hat f)}$ and $\tr_{\hat f_{h}}$ are defined and equivalences for
the uncompleted topological coarse $K$-homology theory.
Also the analogues of \cref{gouweriofjmwkemrfwklerf} and \cref{kogpwergrefwferfw} hold true for the uncompleted topological coarse $K$-homology theory. The corollary now follows  by repeating the proof of \cref{koprhererge}.
\end{proof}

\section{The topological $L^{2}$-index theorem}\label{ohkjptwehgrtherge}

The main result of
this section depends on the existence of  a trace-preserving algebraic approximation of topological $K$-theory.
We let $\Ccat^{\tr}$ denote the category of pairs $(\bC,\tau)$ of a $C^{*}$-category $\bC$ and a continuous trace $\tau$.
We let furthermore $\Ccat^{\tr,\perf}$ be the full subcategory of objects $(\bC,\tau)$ where $\tau$ is defined everywhere.
Let $\cT^{\trc}_{\C}:\Ccat^{\tr,\perf}\to \Cat_{\C}^{\tr,\perf}$ denote the forgetful functor from  $C^{*}$-categories to $\C$-linear categories with traces.
Let $\bC$ be in  $\Ccat^{\tr,\perf}$.  

\begin{ddd} \label{objpbdbdgbd} A trace-preserving algebraic $\bC$-approximation of topological $K$-theory
 is  a homological  and traced functor $H:\Cat_{\C}\to \Sp$ with a natural transformation
$$c:H\cT_{\C}\to K^{\topp}
:\Ccat\to \Sp$$ such that:
\begin{enumerate} 
\item The triangle  
\begin{equation}\label{sdfvsdfvsdfvwerv}\xymatrix{\pi_{0}H\cT^{\trc}_{\C}\ar[rr]^{c}\ar[dr]_{\trc^{H}}&&\pi_{0} K^{\topp}\ar[dl]^{\trc^{K^{\topp}}}\\&\const_{\C}&} 
\end{equation} of functors from $\Ccat^{\tr,\perf}$ to $\Ab$ commutes.
\item We have an equivalence  \begin{equation}\label{sbfdvdfvfdvsd}c_{\bC}:H\cT_{\C}( \bC)\stackrel{\simeq}{\to} K^{\topp}(\bC)\ .
\end{equation} 
\end{enumerate}
\end{ddd}
 
A more precise  notation for a  trace-preserving algebraic $\bC$-approximation would be $((H,\trc^{H}),c)$ but we drop the tracing  in order to shorten.
 
%
%
%
%
%
%

\begin{ex}
By \cref{oiw0rgwrefwerfwrefw}  the pair $((K^{\Cat_{\Z}}H\cZ)_{\cL^{1}}\cT_{\C},c^{\cL^{1}})$ with tracing \eqref{erfewrfwegrgw1} is a trace-preserving $\Hilb^{\fin}(\C)$-approximation of 
$K^{\topp}$.
%
\hB
\end{ex}

We consider an idempotent complete $C^{*}$-category $\bC$  admitting all AV-sums  with a trace  
$\tau$ such that  $\bC^{u}\subseteq  \dom(\tau)$, where 
$\bC^{u}$ denotes the subset of unital objects.

Let $f:X\to Y$ be a uniform $G$-covering.
Then in analogy to \eqref{fqpopoqwedqwedqwedqwede} and using \cref{jkopwergergwefwre} for the components  we get vector-valued traces 
\begin{equation}\label{vweroivhioefw2}\tau_{Y}:\bV_{\bC}(Y)\to \prod_{\pi_{0}^{\crs}(Y)}\C\ , \quad  \tau^{G}_{X}:\bV^{G}_{\bC}(X)\to \prod_{\pi_{0}^{\crs}(X)/G}\C\end{equation}
defined on objects with component-wise bounded support or  $G$-component-wise $G$-bounded support, respectively.

If the coarse components of $Y$ are bounded, and the coarse $G$-components of $X$ are $G$-bounded, then the traces $\tau_{Y}$ and $\tau_{X}^{G}$ are   defined everywhere and we get
induced maps 
\begin{equation}\label{vweroivhioefw}\tau^{K}_{Y}:\pi_{0}K\cX_{\bC}(Y)\to \prod_{\pi_{0}^{\crs}(Y)}\C\ , \quad  \tau^{K,G}_{X}:\pi_{0}K\cX^{G}_{\bC}(X)\to \prod_{\pi_{0}^{\crs}(X)/G}\C
\end{equation}
(we abbreviate $\tau^{K^{\topp}}_{Y}$ by $\tau^{K}_{Y}$ etc. for better readability).
Recall the universal local homology theory $\Yo\cB:\UBC\to \Sp\cB$ from \cite[Sec. 4]{ass}.
By $ \Sp\cB\langle \Yo\cB(*)\rangle$ we denote the localizing subcategory of $\Sp\cB$ generated by $\Yo\cB(*)$.
\begin{theorem} [Topological $L^{2}$-index theorem]\label{igoewrifoperfrefw} Let $ f:X\to Y$ be a uniform $G$-covering.  
We assume: \begin{enumerate}
\item\label{pkothperthertge} $\Yo\cB(Y)\in  \Sp\cB\langle \Yo\cB(*)\rangle$.
\item  \label{htekogrtgrteeptrgertgreger} $Y$ has finite asymptotic dimension, finite uniform topological dimension, and a countably generated uniform structure.\item \label{htekoeptrgertgreger} All coarse components of $Y$ are bounded and all coarse $G$-components of $X$ are $G$-bounded.
\item \label{htekoeptrgertgrege}There exist a trace-preserving   algebraic $\bC^{u}$-approximation $(H,c)$ of $K^{\topp}$.
\end{enumerate}
Then  the  diagram
\begin{equation}\label{grwegergwefwetttrf}
\xymatrix{ \pi_{1} K\cX_{\bC}^{G}(\cO^{\infty}(X))\ar[r]^-{\partial^{\cone}}  &\pi_{0}K\cX_{\bC}^{G}(X) \ar[r]^{\tau_{X}^{K,G}}& \prod_{\pi_{0}^{\crs}(X)/G}\C \\  \pi_{1}K\cX_{\bC}(\cO^{\infty}(Y))\ar[u]_{\tr_{\cO^{\infty}(f)}}\ar[r]^-{\partial^{\cone}} & \pi_{0}K\cX_{\bC}(Y) \ar[r]^{\tau_{Y}^{K}}& \prod_{\pi_{0}^{\crs}(Y)}\C\ar[u]^{f^{*}}} \end{equation}
 commutes.
\end{theorem}
 
\begin{proof}
The assumptions \ref{igoewrifoperfrefw}.\ref{htekoeptrgertgreger} on the components ensures that the traces $\tau^{K,G}_{X}$ and
$\tau^{K}_{Y}$ in \eqref{vweroivhioefw} are defined. The assumption  \ref{igoewrifoperfrefw}.\ref{htekogrtgrteeptrgertgreger} ensures that the cone transfer   $\tr_{\cO^{\infty}(f)}$ is defined by \cref{okjgpwerergw9} and \cref{kophokhotrpherthetrhee} and that \cref{kohpertgertgertg} is applicable.
The transformation
$c:H\cT_{\C}\to K^{\topp}$ (from the data assumed in \ref{igoewrifoperfrefw}.\ref{htekoeptrgertgrege}) 
together with the inclusions  $$\bV^{\ctr}_{\bC}(-)\to\cT_{\C} (\bV_{\bC}(-)) \ ,\quad
 \bV^{G,\ctr}_{\bC}(-)\to\cT_{\C} (\bV^{G}_{\bC}(-))$$  of the uncompleted into the completed Roe categories (see \eqref{vwercecvfsdv})
  induce natural transformations
  $$c:H\cX_{\bC}^{\ctr} \to K\cX_{\bC}\ , \quad c^{G}:H\cX_{\bC}^{G,\ctr} \to K\cX^{G}_{\bC}$$
  of coarse homology theories.
We extend the diagram  \eqref{grwegergwefwetttrf} as follows
\begin{equation}\label{grwegergwefwfffffetttrf}
\xymatrix{ \pi_{1} H\cX_{\bC}^{G,\ctr}(\cO^{\infty}(X))\ar[dr]^{c_{\cO^{\infty}(X)}^{G}}\ar[rr]^{\partial^{\cone}} &&\ar[d]^{c^{G}_{X}}\pi_{0}H\cX^{G,\ctr}(X)\ar[dr]^{\tau_{X}^{H,G}}&\\&  \pi_{1}K\cX_{\bC}^{G}(\cO^{\infty}(X))\ar[r]^-{\partial^{\cone}}  &\pi_{0}K\cX_{\bC}^{G}(X) \ar[r]^{\tau_{X}^{K,G}}& \prod_{\pi_{0}^{\crs}(X)/G}\C \\  & \pi_{1}K\cX_{\bC}(\cO^{\infty}(Y))\ar[u]_{\tr_{\cO^{\infty}(f)}}\ar[r]^-{\partial^{\cone}} & \pi_{0} K\cX_{\bC}(Y) \ar[r]^{\tau_{Y}^{K}}& \prod_{\pi_{0}^{\crs}(Y)}\C\ar[u]^{f^{*}}\\ \pi_{1}H\cX^{\ctr}_{\bC}(\cO^{\infty}(Y))\ar[ur]^{\simeq}_{c_{\cO^{\infty}(Y)}}\ar[uuu]_{\tr_{\cO^{\infty}(f)}}\ar[rr]_{\partial^{\cone}} &&\ar[u]_{c_{Y}}\pi_{0}H\cX^{\ctr}(Y)\ar[ur]_{\tau_{Y}^{H}}&} \end{equation}
The squares involving the cone boundary commute by the naturality of $c$ or $c^{G}$.
 The triangles on the right commute by 
\eqref{sdfvsdfvsdfvwerv}, the fact that the traces on $\bV^{\ctr}_{\bC}(Y)$ or $\bV^{G,\ctr}_{\bC}(X)$
are the restrictions of the traces on $ \bV_{\bC}(Y)$ or $\bV^{G}_{\bC}(X)$, respectively, and the naturality of the transformation in \eqref{fdvsdfvsdffsc}.
The left square commutes by the compatibility of the
definitions of the transfers on the level of categories of controlled objects and the naturality of $c$ or $c^{G}$.
Finally, the outer part of the diagram commutes by \cref{kohpertgertgertg}.

Note that $H\cX^{\ctr}_{\bC}$ and $K\cX_{\bC}$ are strong coarse homology theories, see  \cref{hlkepthetrgege} and \cref{lrgerhrtgerg}. By \cite[Lem. 9.6]{ass}  their compositions with the cone functor are local homology theories.
The equivalence \eqref{sbfdvdfvfdvsd} (with $\bC^{u}$ in place of $\bC$) and the fact that
$\bV^{\ctr}_{\bC}(*)\simeq \bC^{u}$ implies that the transformation of  local homology theories 
$H\cX^{\ctr}_{\bC}(\cO^{\infty}(-))\to K\cX_{\bC}(\cO^{\infty}(-))$  is an equivalence on the point, and therefore on all objects with $\Yo\cB(Y)\in \Sp\cB\langle \Yo\cB(*)\rangle$.   Hence by our assumption \ref{igoewrifoperfrefw}.\ref{pkothperthertge} the lower left diagonal map $c_{\cO^{\infty}(Y)}$
is an equivalence as indicated.

The commutativity of the cells shown so far now imply the commutativity of the cell \eqref{grwegergwefwetttrf}.
\end{proof}

\begin{ex} 
The Assumption \ref{igoewrifoperfrefw}.\ref{pkothperthertge} is the most general finiteness assumption on $Y$ under which we can show that the comparison map $c$ is an equivalence. It is 
 e.g. satisfied if $Y$ is a compact manifold or  homotopy equivalent to a finite simplicial complex.  
If $Y$ is a compact Riemannian manifold or finite simplicial complex and $f:X\to Y$ is a $G$-covering
with induced metric structures, then   Assumption \ref{igoewrifoperfrefw}.\ref{htekogrtgrteeptrgertgreger}
is obvious and Assumption
 \ref{igoewrifoperfrefw}.
 \ref{htekoeptrgertgreger} is satisfied by 
 \cref{knknonkpfgfghnfnfhnf9}. 
 \hB
\end{ex}

\begin{ex}\label{ojperherthetrg}
Assume that $Y$ is a compact Riemannian manifold with a Dirac type operator $\Dirac_{Y}$ of degree $0$ with symbol
$\sigma(\Dirac_{Y})$ in $\pi_{1}K\cX(\cO^{\infty}(Y))$.
Assume further that $X\to Y$ is a Riemannian $G$-covering and that $\Dirac_{X}$ is a $G$-equivariant lift of $\Dirac_{Y}$.
Then one can check  that $$\tr_{\cO^{\infty}(f)}(\sigma(\Dirac_{Y}))=\sigma(\Dirac_{X})$$ in $\pi_{1}K\cX^{G}(\cO^{\infty}(X))$. For the argument one must go into the details of \cite{indexclass}. In particular one must construct the analogue of the transfer along  for  Roe algebras and show  that it is compatible with the identification   \cite[(6.2)]{indexclass}.

 Note that $$\partial^{\cone}(\sigma(\Dirac_{X}))=\ind\cX(\Dirac_{X})$$ in $K\cX_{0}^{G}(X)$, and similarly $$\partial^{\cone}(\sigma(\Dirac_{Y}))=\ind\cX(\Dirac_{Y})$$ in $K\cX_{0}(Y)$.
We can now recover the classical statement of  Atiyah's $L^{2}$-index theorem  
$$\tau^{K}_{Y}(\ind\cX(\Dirac_{Y}))=\tau^{K,G}_{X}(\ind\cX(\Dirac_{X}))$$ as an immediate consequence
of \cref{igoewrifoperfrefw}.
\hB

%
%
%
%

\end{ex}

\section{ Higson's counterexample to the coarse Baum-Connes conjecture}\label{lkthperthergetrgtre}

The goal of this section is to present Higson's counter example \cite{cbcc} to the surjectivity of the coarse assembly map within the formalism of this  note.
We consider a discrete group $G$ with the following properties:
\begin{ass}\label{rkojgpwergwrefwerf} \mbox{}
\begin{enumerate}
\item $G$ is finitely generated.
\item \label{fjioqwefwqdqewd} $G_{can}$    has finite asymptotic dimension.
\item\label{fjioqwefwqdqewd1}  $G$ has property $T$.
\item \label{fjioqwefwqdqewd2} $G$ is residually finite.
\end{enumerate}
\end{ass}
In \cref{rkojgpwergwrefwerf}.\ref{fjioqwefwqdqewd}   we more precisely require that the coarse space $G_{can}$, i.e., $G$ with the canonical $G$-coarse structure generated by the entourages $\{(g,gh)\mid g\in G\}$ for all $h$ in $G$, has finite asymptotic dimension as defined in  \cref{uihefigvweeffd1}.\ref{okgopwererfwerf}.
Property $T$ in \cref{rkojgpwergwrefwerf}.\ref{fjioqwefwqdqewd1} is  equivalent to the existence of a projection
$q_{Kzdn}$ in the maximal group $C^{*}$-algebra $C^{*}_{\max}(G)$ with the following property:  If $\lambda:G\to U(H)$ is any unitary representation of $G$
on a Hilbert space, then $\lambda(q_{Kzdn})$ is the orthogonal projection onto the subspace of $G$-invariant vectors in $H$.  Here $\lambda(q_{Kzdn})$ is defined  using  the universal property of $C^{*}_{\max}(G)$ saying that any
unitary representation of $G$ on a Hilbert spaces uniquely extends to a homomorphismus of $C^{*}$-algebras from $C^{*}_{\max}(G)$ to the bounded operators on the same Hilbert space.
Residual finiteness in  \cref{rkojgpwergwrefwerf}.\ref{fjioqwefwqdqewd2} is witnessed by a decreasing family $(G_{n})_{n\in \nat}$  of normal and  finite index  subgroups
 with $\bigcap_{n\in \nat} G_{n}=\{e\}$.

\begin{ex}
  Torsion-free cocompact lattices  $G$   in $Sp(1,n)$  satisfy \cref{rkojgpwergwrefwerf}. 
They exist by \cite{zbMATH03598578}.

The coarse space $G_{can}$ is coarsely equivalent to the symmetric space (the quarternionic hyperbolic space of dimension  $4n$) of $Sp(1,n)/Sp(1)Sp(n)$ which has finite asymptotic dimension $\le 4n$ since is boundary sphere has dimension $4n-1$ 
\cite{Buyalo_2007}.

The group  $G$ is a  fundamental group of the compact  manifold  $G\backslash Sp(1,n)/Sp(1)Sp(n)$ and therefore finitely generated.

It is well-known that  the Lie group $Sp(1,n)$ has property $T$.
 The property $T$  for $G$ is inhertited  as a subgroup of $Sp(1,n)$. 
 
 Finally, discrete subgroups  of  linear groups like $Sp(1,n)$ are residually finite. 
\hB
 \end{ex}

We consider the branched coarse $G$-covering $f:X\to Y$ with respect to the family $\cZ$  constructed from the data $G$ and $(G_{n})_{n\in \nat}$ in \cref{kopgergwerg9}. 

Following \cite[Def. 9.4]{ass} the coarse assembly map for $Y$ is the map
\begin{equation}\label{gwergerwgergwerffrf}\mu: \colim_{V\in \cC_{Y}} \Sigma^{-1}K\cX\cO^{\infty}(P_{V}(Y))\stackrel{\partial^{\cone}}{\to}\colim_{V\in \cC_{Y}} K\cX (P_{V}(Y))\stackrel{\simeq}{\leftarrow} K\cX(Y)\ .
\end{equation}
In \cite{cbcc} Higson constructs a non-trivial class $p$ in $\pi_{0}K\cX(Y)$ and then shows that
 it is   not in the image of the coarse assembly map. The goal of the rest of this section is to give a detailed proof of this result.
 
 We start with a description of the class $p$.
Recall that $X\cong G_{can,min}\otimes S_{min,min}$ and $Y\cong G\backslash X$.
 The underlying set of  $X$ has an additional right $G$-action induced by the right multiplication of $G$ on itself  which induces a right $G$-action on the underlying set of $Y$. For $g$   in $G$ the subset $U_{g}:=\{(x,xg)|x\in X\}$ is a coarse entourage of $X$. In fact, the coarse structure of $X$ is generated by the collection of  these entourages for all $g$ in $G$.
 It follows that
 $V_{g}:=\{(y,yg)|y\in Y\}$ is a coarse entourage of $Y$.

 The Hilbert space $L^{2}(Y)$ gives rise to an object $(L^{2}(Y),\mu)$ of $\bV(Y)$ in a natural way.
 The $G$-action on $Y$ induces   a unitary representation $\lambda$ of $G$ on $L^{2}(Y)$.
 By construction the  operator $\lambda(g)$ is controlled by $V_{g}$. We therefore get a homomorphism
 $G\to U \End_{\bV(Y)}((L^{2}(Y),\mu))$. By the universal property of the maximal group $C^{*}$-algebra it extends to a homomorphism of $C^{*}$-algebras
 $\lambda:C_{\max}^{*}(G)\to \End_{\bV(Y)}((L^{2}(Y),\mu))$.
 Then we set
 $P:=\lambda(q_{Kzdn}) $ and let $p$ be the  class of the  projection $P$  in $\pi_{0}K\cX(Y)$.
 Here we use $\pi_{0}K\cX(Y)\cong \pi_{0}K^{\topp}(\bV(Y))$ and the description of this group as in \cref{kogpwergweferwf}.

 The coarse components of $Y$ are given by the subspaces $Y_{n}:=G\backslash G_{can,min}\otimes G/G_{n}$ and indexed by $\nat$.
 These are finite sets and hence bounded. Therefore  the trace \eqref{vweroivhioefw}
$$\tau_{Y}^{K}:\pi_{0}K\cX(Y)\to \prod_{n\in \nat}  \C$$ is well-defined. 
 We let $P_{n}$ be the component of the projection $P$ on $Y_{n}$. By the universal property of the Kaszhdan projection it projects onto the $G$-invariant subspace of $L^{2}(Y_{n})$.
 We now note that $G$ acts transitively on $Y_{n}$. It follows that $  L^{2}(Y_{n})^{G}\cong \C$ consists of the constant functions  and we get  $\tr_{Y_{n}}(P)=1$ for the $n$'th component of  the trace $\tau_{Y}(P)$ in \eqref{vweroivhioefw2}.   This implies that
 $\tau^{K}_{Y}(p) $  is   the constant family with value $1$ in $ \prod_{n\in \nat}  \C$.
 In particular $p\not=0$.

Since   $G_{ can}$ has finite asymptotic dimension also   $X$ has finite asymptotic dimension and
 $(f:X\to Y,\cZ)$ belongs to
$G\BCov^{s\fadim}$. In particular the transfer $\tr_{f}:K\cX(Y,\cZ)\to K\cX^{G}(X,f^{-1}(\cZ))$ is defined 
by \cref{koheprthkoptrkgergrtgegrtgte}.

\begin{rem}
We will see below that $\tr_{f}[p]=0$ so that we do not have a transfer equivalence. So $Y$ can not have finite asymptotic dimension in view of \cref{koheprthkoptrkgergrtgegrtgte}. \hB
\end{rem}

The transfer of the object $(L^{2}(Y),\mu)$ is $(L^{2}(X),\sigma,\nu)$ in $\bV^{G}(X)$, where
$\sigma$ is the representation induced by the left action on $X $. We let $\kappa$ denote the  unitary $G$-action on $L^{2}(X) $ given by the right action of $G$ on $X$. For $g$ in $G$ the unitary operator  $\kappa(g)$ is controlled by the entourage  $U_{g}$ introduced above.

 Let $(q_{i})_{i\in \nat}$ be a sequence of elements in $\C[G]$ such that $\lim_{i\to \infty} q_{i}=q_{Kzdn}$ in $C^{*}_{\max}(G)$. 
 Then $P=\lim_{i\to \infty} \lambda(q_{i})$. Since the representation $\kappa$ on $L^{2}(X)$  does not have  non-trivial invariant vectors we get
 $\lim_{i\to \infty} \kappa(q_{i})=\kappa(q_{Kzdn})=0$.

 Consider $i$ in $\nat$. Since $q_{i}$ is a finite linear combination of elements of $G$ the operator $\lambda(q_{i})$ is $V$-controlled by some coarse entourage $V_{i}$ of $Y$.  There exists $n_{i}$ in $\nat$ such that the parallel transport at scale $V_{i}$ is defined on $Y_{\le n_{i}}^{c}$. In view of the explicit description of the transfer in terms of matrices \eqref{werfwerweg} the transfer of the unitary $\lambda(g) \mu(Y_{\le n_{i}}^{c}) $ is $\kappa(g) \nu(X_{\le n_{i}}^{c})$ provided $V_{g}\subseteq V_{i}$. Therefore 
 the transfer of $ \lambda(q_{i}) \mu(Y_{\le n_{i}}^{c})$ is  represented by  $\kappa(q_{i})\nu(X_{\le n_{i}}^{c})$.
 Fix $\epsilon$ in $(0,\infty)$ and let $i$ be so large that $\|q_{i}-q_{Kzdn}\|\le \epsilon/2$ and $\|\kappa(q_{i})\|<\epsilon/2$.  Since the transfer is a morphism of $C^{*}$-categories and therefore contractive
 we get   $\|[\kappa(q_{i})\nu(X_{\le n_{i}}^{c})]-\tr_{f}([P])\|\le \epsilon/2$ and hence $\|\tr_{f}([P])\|\le \epsilon$.
 As $\epsilon$ was arbitrary we get $\tr_{f}([P])=0$.
 The brackets  indicate taking equivalence classes in the quotients  in $\bV(Y,\cZ)$ or $\bV^{G}(X,f^{-1}(\cZ))$.
 We conclude that $\tr_{f}[p]=0$ in $\pi_{0}K\cX^{G}(X,f^{-1}(\cZ))$.
 
 Alternatively we could argue that $p$ is a ghost and apply \cref{lphetgerg}.

 \begin{prop}
 The class $p$ is not in the image of  the assembly map \eqref{gwergerwgergwerffrf}.
 \end{prop}
\begin{proof}
Assume by contradiction that $p$ is in the image of the assembly map \eqref{gwergerwgergwerffrf}.
 Then there exists an entourage $V$ (which we can assume to be generating)  of $Y$ and $u$ in $\pi_{1}K\cX\cO^{\infty}(P_{V}(Y))$ such that
 $p':=\partial^{\cone}(u)$ in $\pi_{1}K\cX(P_{V}(Y))$ is the image of $p$ under $\pi_{0}K\cX(Y)\stackrel{}{\to} \pi_{0}K\cX(P_{V}(Y))$ induced by the map $Y\to P_{V}(Y)$ sending the points of $Y$ to Dirac measures.

  We  can assume that $V=f(U)$ for some invariant  generating entourage $U$ of $X$.
 By \cref{okoprtkgpobgfbdfgbdfgbdfgbdgfb}.\ref{elrtjhperthtrgrtgetr} there exists an integer $n$ such that
 $$P_{U}(X_{\le n}^{c})\to P_{V}(Y_{\le n}^{c})$$ is a uniform $G$-covering.
In order to simplify the notation we can throw away the first few members of the big family and then assume that
 $$P(f):P_{U}(X)\to P_{V}(Y )$$ is itself a uniform $G$-covering.
 Since $U$ and $V$ are generating entourages the coarse components of  $P_{U}(X)$ and $P_{V}(Y)$ correspond to the coarse components of $X$ and $Y$ and are indexed by $\nat$.

 We consider the $n$'th component $ P(f_{n}):P_{U}(X_{n})\to P_{V}(Y_{n})$. In the following we argue that this uniform $G$-covering 
satisfies the  
  assumptions for the coarse $L^{2}$-index  \cref{igoewrifoperfrefw}. 
  First of all $P_{V}(Y_{n})$ is a finite simplicial complex and therefore $\Yo\cB(Y_{n})\in \Sp\cB\langle \Yo\cB(*)\rangle$  
  by homotopy invariance and excisiveness of $\Yo\cB$. This verifies Condition \ref{igoewrifoperfrefw}.\ref{pkothperthertge}. 
Since $P_{U}(Y_{n})$ is a finite simplicial complex also
 Condition 
 \ref{igoewrifoperfrefw}.\ref{htekogrtgrteeptrgertgreger} is satisfied.
 The bornological coarse space $Y_{n}$ consists of one bounded component and
  $X_{n}$ consists of a single  $G$-bounded $G$-component yielding Condition  \ref{igoewrifoperfrefw}.\ref{htekoeptrgertgreger}).
  Finally,  as we work with the coefficient category $\bC=\Hilb_{c}(\C)$ with the standard trace
 by \cref{oiw0rgwrefwerfwrefw} there exists a trace-preserving $\bC^{u}=\Hilb^{\fin}(\C)$-approximation of $K^{\topp}$ as required in  Condition  \ref{igoewrifoperfrefw}.\ref{htekoeptrgertgrege}.

 By $u_{n}$ in $\pi_{1}K\cX(\cO^{\infty}(P_{V}(Y_{n})))$ and $p'_{n}=\partial^{\cone}u_{n}$ in $\pi_{0}K\cX(P_{V}(Y_{n}))$ we denote the  corresponding components of $u$ and $p'$.  
 By 
 \cref{igoewrifoperfrefw}  we have
 \begin{equation}\label{ertgertggertgetrg} \tau^{K,G}_{X_{n}}(\partial^{\cone}\tr_{\cO^{\infty}(P(f_{n}))}(u_{n}))= \tau^{K}_{Y_{n}}(p'_{n})=1\ .
\end{equation} 

  We now use the   diagram 
$${\tiny \hspace{-2.5cm}\xymatrix{ \ar[d]\Sigma^{-1}K\cX(\cO^{\infty}(P_{V}(Y_{n})) )\ar@/_1cm/[ddddd]_{\tr_{\cO^{\infty}(P(f_{n}))}}\ar[rrr]^{\partial^{\cone}}&&&K\cX(P_{V}(Y_{n}))\ar[d]
\\ \Sigma^{-1}K\cX(\cO^{\infty}(P_{V}(Y)) )\ar[dr]^{[-]}\ar[ddd]^{\tr_{\cO^{\infty}(P(f))}}\ar[rrr]^-{\partial^{\cone}} &&& K\cX(P_{V}(Y))\ar[dl]_{[-]} \\ & \Sigma^{-1}K\cX(\cO^{\infty}(P_{V}(Y)),\cO^{\infty}(P_{V}(\cZ))) \ar[r]^-{\partial^{\cone}}\ar[d]^{\tr_{\cO^{\infty}(P(f))}}&  K\cX(P_{V}(Y),P_{V}(\cZ))\ar[d]^{\tr_{P(f)}}&\\&\Sigma^{-1}K\cX^{G}(\cO^{\infty}(P_{U}(X)),\cO^{\infty}(P_{U}(f^{-1}(\cZ))))\ar[r]^-{\partial^{\cone}}& K\cX^{G}(P_{U}(X),P_{U}(f^{-1}(\cZ)))&\\ \Sigma^{-1}K\cX^{G}(\cO^{\infty}(P_{U}(X)))\ar[ur]^{[-]} \ar[rrr]^-{\partial^{\cone}} &&& K\cX^{G}(P_{U}(X)) \ar[ul]_{[-]} \\ \Sigma^{-1}K\cX^{G}(\cO^{\infty}(P_{U}(X_{n})))\ar[rrr]^{\partial^{\cone}}\ar[u]&&& K\cX^{G}(P_{U}(X_{n}))\ar[u]} }$$
which commutes by the naturality of the transfer. The unnamed vertical maps are all induced by inclusions of components.
The  commutativity of the middle part says that 
$$ [\partial^{\cone}\tr_{\cO^{\infty}(P(f))}(u)]= \tr_{P(f)} [p']$$
where the brackets indicate the projection to the relative $K$-groups.
We conclude that
\begin{equation}\label{werfewerfwerf} \bar \tau^{K,G}_{X}  [\partial^{\cone}\tr_{\cO^{\infty}(P(f))}(u)]=\bar \tau^{K,G}_{X}( \tr_{P(f)} [p'])\ . \end{equation}
The outer part of the diagram says that
$ \partial^{\cone}\tr_{\cO^{\infty}(P(f))}(u)$ has the components
$ \partial^{\cone}\tr_{\cO^{\infty}(P(f_{n}))}(u_{n})$ and therefore
 that by \eqref{ertgertggertgetrg} the trace 
$\bar \tau^{K,G}_{X}  [\partial^{\cone}\tr_{\cO^{\infty}(P(f))}(u)]$ is  the class of the constant family $(1_{n})$ in $T(X)\cong \prod_{\nat}/\bigoplus_{\nat}\C$.
Since $\tr_{f}[p]=0$ we also have  $ \tr_{P(f)} [p'] =0$. This would imply 
$\bar \tau^{K,G}_{X}( \tr_{P(f)} [p'])=0$ which  is a contradiction.
\end{proof}

\begin{rem}
The argument above is essentially the same as the one given in \cite[Sec. 5]{cbcc} with the following exceptions.  Firstly, we provide an alternative to
   \cite[Lem. 5.4]{cbcc} and propose the condition of finite asymptotic dimension as a condition for the 
existence of the transfer. Secondly, instead of justifying the  applicability of Atiyah's $L^{2}$-index theorem in \cite[Prop. 5.6]{cbcc}  
we propose to use our topological $L^{2}$-index theorem \cref{igoewrifoperfrefw}
which allows us to argue completely within coarse $K$-homology theory.
 \hB
 \end{rem}

\section{Comparison of algebraic and topological $K$-theory functors}\label{jiogwerfwrefrefw}


Let $$K^{\Cat_{\Z}}:\Cat_{\Z}\to \Sp\qquad \mbox{and}\qquad  K^{\topp}:\nCcat\to \Sp$$ denote the
algebraic $K$-theory functor for  $\Z$-linear categories and
the topological $K$-theory functor for $C^{*}$-categories. We   form the universal polynomially homotopy invariant functor (see \eqref{vdfsvewrvfsdvsdfvsd})
  $$h:K^{\Cat_{\Z}}\to K^{\Cat_{\Z}}H:\Cat_{\Z}\to \Sp$$ under $K^{\Cat_{\Z}}$ called the homotopy $K$-theory for  $\Z$-linear categories.

Let $\cZ:\Cat_{\C}\to \Cat_{\Z}$ denote the forgetful functor from $\C$-linear categories to $\Z$-linear categories, and
$\cT_{\C}:\Ccat\to \Cat_{\C}$ the forgetful functor from unital $C^{*}$-categories to $\C$-linear categories.
Finally let 
$\cL^{1}$ be the algebra of trace class operators on the Hilbert space $\ell^{2}$.  Then we can twist the composition $K^{\Cat_{\Z}}H\cZ$ with $\cL^{1}$ and obtain
  the functor 
  $$(K^{\Cat_{\Z}}H\cZ)_{\cL^{1}}:\Cat_{\C}\to \Sp\ .$$
The main goal of this section is  to state \cref{hlrztrhertgetrg} of the natural transformation
\begin{equation}\label{gwegerfwr} c^{\cL^{1}}:(K^{\Cat_{\Z}}H\cZ)_{\cL^{1}}\cT_{\C} \to K^{\topp}_{|\Ccat}:\Ccat\to \Sp
\ .\end{equation}  We will observe that it induces an equivalence on the   $C^{*}$-category $\Hilb^{\fin}(\C)$ of finite-dimensional Hilbert spaces, and that it is compatible with the trace homomorphisms.
  
 In the language of \cref{objpbdbdgbd} we can formulate our result as follows.
 
 \begin{theorem}
 The pair $((K^{\Cat_{\Z}}H\cZ)_{\cL^{1}},c^{\cL^{1}})$ is a trace-preserving $\Hilb^{\fin}(\C) $-approximation of $K^{\topp}$.
 \end{theorem}

An intermediate step towards the transformation \eqref{gwegerfwr} is
a transformation $$K^{\Cat_{\Z}}\cT \to K^{\topp}_{|\Ccat}:\Ccat\to \Sp $$ (see \eqref{qewdqwdwedqd})
from the algebraic to the topological  $K$-theory of unital $C^{*}$-categories, where $\cT:\Ccat\to \Cat_{\Z}$ is the obvious forgetful functor.
The existence of such a transformation is surely folklore. But the details of the construction of 
$K^{\Cat_{\Z}}$ and $K^{\topp}$ are quite different, and so the construction of this transformation 
 is not completely obvious.

\subsection{$K$-theory for additive and $\Z$-linear categories}
\label{regoierjfoiewfrw}
 
The source of all algebraic $K$-theory functors considered in the present note is the non-connective $K$-theory 
\begin{equation}\label{roijhoetgret}K^{\alg}:\Add\to \Sp
\end{equation} for additive  categories. Original references for this functor  are 
 \cite{MR2206639}, \cite{zbMATH04095731}. 
 
 \begin{rem}
 The most direct classical construction of the connective cover $K^{\alg}_{\ge 0}$ of this functor 
 associates to $\bA$ in $\Add$ the group completion $K^{\alg}_{\ge 0}(\bA)$  of the commutative groupoid $\bA^{\simeq}$
  considered as a monoid in anima (realized by the nerve of the underlying groupoid of $\bA$) with  monoid structure induced by the sum in $\bA$.
Here we identify commutative groups in anima with connective spectra.

Alternatively one can equip $\bA$ with  a Waldhausen category structure and obtain $K^{\alg}_{\ge 0}(\bA)$ by applying Waldhausen's $S_{\bullet}$-construction.

The non-connective versions $K^{\alg}$ are then constructed by applying delooping constructions, see e.g. \cite{Bunke:2017aa}
for a discussion of the uniqueness of the resulting functor $K^{\alg}$. 

In the present paper we prefer to define  $K^{\alg}$  using universal constructions.
\hB
 \end{rem}

Following 
  \cite{Bunke:2017aa} 
    we consider the functor
 $$\Ch^{b}_{\infty}:\Add\to \Cat^{\exa}_{\infty}$$
which sends an additive catgeory $\bA$ to the stable $\infty$-category obtained  by Dwyer-Kan localizing  the category
of bounded chain complexes in $\bA$  at  the homotopy equivalences.

\begin{rem}
In this remark we characterize $\Ch^{b}_{\infty}$ by universal properties.
Let $\ell:\Add\to \Add_{2,1}$ denote the Dwyer-Kan localization of $\Add$ at the equivalences. 
Since $\Ch^{b}_{\infty}$ sends equivalences to equivalences it has a factorization
$$\xymatrix{\Add\ar[dr]_{\ell}\ar[rr]^{\Ch^{b}_{\infty}}&& \Cat^{\exa}_{\infty}\\ &\Add_{2,1}\ar[ur]_{\Ch^{b}_{\infty,2,1}}&}\ .$$
We have a fully faithful embedding $\Add_{2,1}\to \Cat^{\add}_{\infty}$ into the category of additive $\infty$-categories. We furthermore have a forgetful (inclusion) functor $\cR:\Cat_{\infty}^{\exa}\to \Cat^{\add}_{\infty}$
from stable $\infty$-categories to additive $\infty$-categories. 
By \cite[Thm. 7.4.9]{unik} the functor $\Ch^{b}_{\infty,2,1}$ is equivalent to the restriction
of the left-adjoint of $\cR$ to $\Add_{2,1}$.
\hB
\end{rem}

Recall that a localizing invariant   is a functor  $\Cat_{\infty}^{\exa}\to \cC$ to a cocomplete stable $\infty$-category which preserves filtered colimits and sends Verdier sequences to fibre sequences. By  \cite{MR3070515} there exists
a  universal  localizing invariant \begin{equation}\label{fewdqewdqwedq}\cU_{\loc}:\Cat^{\exa}_{\infty}\to \cM_{\loc}\ .
\end{equation} 
   Let $\Sp^{\omega}$ denote the stable $\infty$-category of compact objects in $\Sp$.
 \begin{ddd}\label{hkoperthrthe9}  We define 
 the algebraic $K$-theory functor  for additive categories as the composition  \begin{equation}\label{werfwerfewfwefwref}K^{\alg}:\Add\xrightarrow{\Ch^{b}_{\infty}} \Cat^{\exa}_{\infty}\xrightarrow{\cU_{\loc}} \cM_{\loc}\xrightarrow{\map_{\cM_{\loc}}(\cU_{\loc}(\Sp^{\omega}),-)} \Sp\ .
\end{equation}
 \end{ddd}

\begin{rem}\label{okgpwerewfwerfw}  The following properties of the functor in \eqref{werfwerfewfwefwref} are relevant for the present paper:    \begin{enumerate}
 \item $K^{\alg}$ sends equivalences to equivalences.
 \item $K^{\alg}$ sends  Karoubi filtrations (see \cite[Def. 8.2]{equicoarse}) to fibre sequences.
 \item $K^{\alg}$ annihilates flasques  (see \cite[Def. 8.1]{equicoarse}).
 \item $K^{\alg}$ preserves filtered colimits.
  \item $K^{\alg}$ sends Morita equivalences to equivalences.   
 \item $K^{\alg}$ sends infinite products of additive categories to products.
  \end{enumerate} 
  So   $K^{\alg}$ is in particular homological in the sense of \cref{kopwegrfw}.
  Based on the definition \eqref{werfwerfewfwefwref} the first five properties follow from the corresponding properties of the universal $K$-theory functor
 $$\UK:=\cU_{\loc}\circ \Ch^{b}_{\infty}:\Add\to \cM_{\loc}$$ shown in    \cite[Sec. 2.3]{Bunke:2017aa}.
 The preservation of products is more involved and uses different models \cite{zbMATH07183274}. \hB
  \end{rem}
  
  
As an intermediate category between additive categories and rings we consider   the category $\Cat_{\Z}$   of categories enriched in abelian groups,  called $\Z$-linear categories.
A $\Z$-linear category with a single object is a unital ring.  
 We indicate the $\infty$-categories obtained by inverting equivalences with subscripts $(2,1)$.
 We have  a forgetful functor \begin{equation}\label{goiugjowergwergw}\cS:\Add\to \Cat_{\Z}
\end{equation} ($\cS$ for "forgets existence of sums"). It preserves equivalences and therefore has a unique  factorization $\cS_{2,1}$  over the localization at the equivalences. This factorization then  fits into an adjunction
 \begin{equation}\label{werfwevfwerf}(-)_{\oplus,2,1}:\Cat_{\Z,2,1}\leftrightarrows \Add_{2,1}:\cS_{2,1}
\end{equation} whose left-adjoint is
given by the free sum completion \cite[Cor. 2.60]{zbMATH07194060}. The notation indicates that the free sum completion
exists as a functor $(-)_{\oplus}:\Cat_{\Z}\to \Add$. But only its descent to the localization fits into the adjunction.

Since $K^{\alg}$ preserves equivalences it has a unique factorization 
$$\xymatrix{\Add\ar[dr]_{\ell}\ar[rr]^{K^{\alg}}&&\Sp\\&\Add_{2,1}\ar@{..>}[ur]_{K^{\alg}_{2,1}}&}\ .$$
 \begin{ddd}\label{oophkerpthetreg}
We define the algebraic $K$-theory functor 
$$K^{\Cat_{\Z}}:\Cat_{\Z}\to \Sp$$ for $\Z$-linear categories by the diagram
\begin{equation}\label{wergwerfwerfwer1}\xymatrix{ \ar@{..>}[dr]_{K^{\Cat_{\Z}}}\Cat_{\Z}\ar[r]^{\ell}&\Cat_{\Z,2,1}\ar[r]^{(-)_{\oplus,2,1}} &\Add_{2,1}\ar[dl]^{K^{\alg}_{2,1}}\\&\Sp& }\ .
\end{equation} 

\end{ddd}

 
\begin{rem} \label{okgpwerewfwerfw1} 
This extension has the following properties inherited from $K^{\alg}$   since $(-)_{\oplus}$ preserves equivalences, Morita equivalences
and filtered colimits:
\begin{enumerate}
\item $K^{\Cat_{\Z}}$ sends equivalences to equivalences.
\item $K^{\Cat_{\Z}}$ sends Morita equivalences to equivalences. 
\item $K^{\Cat_{\Z}}$ preserves filtered colimits.
\end{enumerate} \hB
\end{rem}

\subsection{$K$-theory for Rings and $C^{*}$-algebras}\label{kopwergwergwef}

In this subsection we recall  the $K$-theory functors for rings and $C^{*}$-algebras. Its main result is the construction of a natural comparison map from the algebraic $K$-theory of $C^{*}$-algebras considered as  rings to  their topological $K$-theory. Thereby  both functors are defined in terms of universal constructions, and the transformation is derived from universal properties.

We have a canonical inclusion $\Ring\to \Cat_{\Z}$ viewing  unital rings as $\Z$-linear categories
with a single object.

\begin{ddd}
We define the algebraic $K$-theory functor for unital rings by
$$K^{Ring}_{|\Ring}:\Ring\to \Cat_{\Z}\xrightarrow{K^{\Cat_{\Z}}} \Sp\ .$$ 
\end{ddd}
Let $\nRing$ denote the category of  possibly non-unital rings.
\begin{ddd} 
We extend the functor above to a functor 
\begin{equation}\label{gwreferfwerf} K^{Ring}:\nRing \to \Sp\ , \quad K^{Ring}(R):=\Fib(K^{Ring}(R^{+})\to K^{Ring}(\Z))\ ,
\end{equation} 
where $R^{+}$ is the unitalization of $R$ and $R^{+}\to \Z$ is the canonical split projection.
\end{ddd}

\begin{rem} 
Making  \eqref{gwreferfwerf} and \eqref{wergwerfwerfwer1} explicit  for a unital ring $A$  and using the equivalence
  $$ \ell(\incl(A))_{\oplus}\simeq \ell(\Mod^{\fg,\free}(A))$$ in $\Add_{2,1}$  we get
$$K^{Ring}(A)\simeq K^{\alg}(\Mod^{\fg,\free}(A))$$ which is one of the usual definitions.
Since $K^{\alg}:\Add\to \Sp$ is Morita invariant 
and  the inclusion
$ \Mod^{\fg,\free}(A)\to  \Mod^{\fg,\proj}(A)$ is a Morita equivalence we also get the  formula 
\begin{equation}\label{jovaopsdcadscasdca}K^{Ring}(A)\simeq K^{\alg}(\Mod^{\fg,\proj}(A))\ .
\end{equation}  \hB
\end{rem}

%
%
%

Extending \cite{MR1068250} from ordinary to $\infty$-categories by  \cite[Thm. 8.5]{zbMATH07948612} we have a   universal symmetric monoidal (for $\otimes_{\max}$) homotopy invariant, $K$-stable, exact and $s$-finitary functor \begin{equation}\label{refwreerf}\ee:\nCalg\to \EE
\end{equation} 
   called the $E$-theory functor. It is the analogue of the universal localizing invariant in  \eqref{fewdqewdqwedq}.
   Following \cite[Def. 9.3]{zbMATH07948612}  and in analogy to \cref{hkoperthrthe9} we adopt the following:   \begin{ddd}\label{kopgwgergw9}
 We  define the $K$-theory functor for $C^{*}$-algebras as the composition \begin{equation}\label{werfwefefwf}K^{C^{*}}:\nCalg\stackrel{\ee}{\to} \EE\xrightarrow{\map_{\EE}(\ee(\C),-)}\Sp\ .
\end{equation} \end{ddd}
As explained in  \cite[Sec. 9 \& 10]{zbMATH07948612} this is compatible with classical definitions.

\begin{rem}
The first six of the following properties of $K^{C^{*}}$ are immediate from this definition and the corresponding properties of the $E$-theory functor 
\begin{enumerate}
\item \label{gwerreg} $K^{C^{*}}$ takes values in $\Mod(KU)$.
\item $K^{C^{*}}$ is homotopy invariant.
\item $K^{C^{*}}$ is $K$-stable.
\item $K^{C^{*}}$ is exact.
\item $K^{C^{*}}$ is $s$-finitary.
\item $K^{C^{*}}:\nCalg\to \Mod(KU)$ is lax symmetric monoidal.
\item $K^{C^{*}}$ preserves filtered colimits.
 \end{enumerate}
Property \ref{gwerreg} is automatic if one defines the commutative complex $K$-theory  ring spectrum by $KU:=\map_{\EE}(\ee(\C),\ee(\C))$. It encodes Bott periodicity.
The preservation of filtered colimits   is equivalent to the statement that $\ee(\C)$ is a compact object. The verification is more complicated and uses the comparison with the classical picture which expresses e.g.  $K_{0}(A)$ in terms of projections in matrix algebras over the unitalization of $A$.
 \hB\end{rem}

We have a forgetful functor \begin{equation}\label{vfsdvsfdvsvvewe}\cT:\nCalg\to \nRing
\end{equation}
($\cT$ for "forgets the topology and $\C$-linear structure")
from possibly non-unital $C^{*}$-algebras to possibly non-unital rings
which we use  in order to restrict  functors 
defined on $\nRing$ to $C^{*}$-algebras. 
The main goal of this subsection is to construct a natural transformation
\begin{equation}\label{t432r2f34r234}c^{C^{*}}:K^{Ring}\cT\to K^{C^{*}}:\nCalg\to \Sp\end{equation}
relating the algebraic $K$-theory of $C^{*}$-algebras  considered as  rings with their topological $K$-theory.
 
\begin{rem} The existence of such a transformation based on the construction of $K^{Ring}$ in terms of Quillen's $+$-construction is well-known \cite{zbMATH04095731}, \cite{Rosenberg_1997}. Let $A$ be a unital $C^{*}$-algebra.  Then we have a  map $BGL(A)^{\delta}\to BGL(A)$, where one considers $GL(A)=\colim_{n\in \nat} GL_{n}(A)$  as a topological group and $GL(A)^{\delta}$ is the underlying discrete group. 
One then observes a factorization
$$\xymatrix{BGL(A)^{\delta}\ar[rr]\ar[dr]&&BGL(A)\\&BGL(A)^{\delta,+}\ar@{..>}[ru]&}\ .$$
This  can be used to determine a map 
$$K^{Ring}_{\ge 1}(\cT(A))\to K^{C^{*}}_{\ge 1}(A)$$
of  $1$-connected covers. One then extends this to the non-connective versions of the $K$-theory spectra and non-unital $C^{*}$-algebras. 

In the present paper we will provide a different construction  of the transformation \eqref{t432r2f34r234}
which is adapted to our definition of $K$-theory functors via universal properties.
  \hB\end{rem}


We let  $K$  denote the $C^{*}$-algebra of compact operators on $\ell^{2}$.
\begin{lem}\label{plherthetrh9}
The functor $$K^{Ring}\cT(K\otimes_{\max}-):\nCalg\to \Sp \ .$$
has the following properties:
\begin{enumerate}
\item homotopy invariant
\item $K$-stable
\item exact
\item s-finitary
\end{enumerate}
\end{lem}
\begin{proof}
By construction the functor $K^{Ring}\cT(K\otimes_{\max}-)$ is $K$-stable.
Since the maximal tensor product of $C^{*}$-algebras preserves exact sequences and $K^{Ring}\cT$ is known to be exact (every $C^{*}$-algebra is $H$-unital by a result of Wodzicki), the functor $K^{Ring}\cT(K\otimes_{\max} -)$ is exact, too. 
   It is  in particular split exact and therefore homotopy invariant by 
 \cite[Thm. 3.2.2]{zbMATH04176083}.
 
  Recall that a functor $F:
\nCalg\to \cC$ is called $s$-finitary   if for all $A$ in $\nCalg$ we have
$$\colim_{A'\subseteq A} F(A')\simeq F(A)\ ,$$
where the colimit runs over the $\aleph_{1}$-filtered poset of separable subalgebras $A'$ of $A$.
  The functor $\cT\circ (K\otimes_{\max}-):\nCalg\to \nRing$ preserves $\aleph_{1}$-filtered colimits.
 Since $K^{Ring}$ preserves all filtered colimits we finally conclude that $K^{Ring}\cT(K\otimes_{\max}-)$ is $s$-finitary.
 \end{proof}

 Since $K^{C^{*}}$  is defined as a functor corepresented by $\ee(\C)$ in $E$-theory it is easy  to construct  natural  transformations out of $K^{C^{*}}$ using universal properties.  
It follows from \cref{plherthetrh9} and the universal property of the $E$-theory functor  in \eqref{refwreerf}  that $  K^{Ring}\cT(K\otimes_{\max}-):\nCalg\to \EE$ has a factorization \begin{equation}\label{gwerfwerfwerfwref}\xymatrix{\nCalg\ar[rrrr]^{K^{Ring}\cT(K\otimes_{\max}-)}\ar[drr]_{\ee}&&&& \Sp\\&&\EE \ar@{..>}[urr]_{!}&&}\ .
\end{equation}
The choice of a one-dimensional projection $p$ in $K$ determines  a natural transformation
\begin{equation}\label{wfgerfwerfw}\id\to K\otimes_{\max}-:\nCalg\to \nCalg
\end{equation}
called the left upper corner embedding.
 We have a map
 \begin{equation}\label{werferwferwfwf}S\xrightarrow{unit} K^{Ring}(\Z)\xrightarrow{\Z\to \C}  K^{Ring}\cT(\C)    \xrightarrow{\eqref{wfgerfwerfw}}  K^{Ring}\cT(K\otimes_{\max}\C) \ .
\end{equation} Together with the dotted arrow in  \eqref{gwerfwerfwerfwref} it  determines a transformation
\begin{equation}\label{rfwrefwfwf}K^{C^{*}} \to  K^{Ring}\cT(K\otimes_{\max}-):\nCalg\to \Sp\ .
\end{equation} 
In detail it is given by
\begin{eqnarray*}
K^{C^{*}}&\simeq&\map_{\EE}(\ee(\C),\ee(-))\\&\stackrel{!}{\to}&
 \map_{\Sp}( K^{Ring}\cT(K\otimes_{\max}\C),  K^{Ring}\cT(K\otimes_{\max}-))\\
 &\stackrel{\eqref{werferwferwfwf}}{\to}&
 K^{Ring}\cT(K\otimes_{\max}-)\ .
\end{eqnarray*}

The statement of  following theorem is  known as the   Karoubi conjecture  and was shown by Suslin and Wodzicki \cite{Suslin_1990}.
\begin{theorem}
The transformation \eqref{rfwrefwfwf}
 is an equivalence.
 \end{theorem}
  
  \begin{proof}
 After all these preparations it is not complicated to finish the argument for the Karoubi coinjecture.  
Using the exactness of the functors in the domain and target   it suffices to verify that \eqref{rfwrefwfwf} induces an isomorphism
$$\pi_{0}K^{C^{*}}(A) \stackrel{\cong}{\to}   \pi_{0}K^{Ring}\cT(K\otimes_{\max}A)$$
for all $C^{*}$-algebras $A$. To see this we can replace $A$ by $K\otimes_{\max}A$ in the domain using $K$-stability of $K^{C^{*}}$ and
then argue that for $C^{*}$-algebras topological and algebraic $K$-theory in degree zero coincide.
\end{proof}
%
%
%
%
%
%
%
%
\begin{ddd}
We define the comparison map from algebraic to topological $K$-theory of $C^{*}$-algebras as the 
  composition
 \begin{equation}\label{t432r234r234}
 c^{C^{*}}:K^{Ring}\cT  \xrightarrow{\eqref{wfgerfwerfw}}
 K^{Ring}\cT(K\otimes_{\max} -) \xleftarrow{\eqref{rfwrefwfwf},\simeq} K^{C^{*}}   \ .\end{equation}
 \end{ddd}

 \subsection{$K$-theory of $C^{*}$-categories}
 
 In this subsection we recall the construction of the $K$-theory functor
 $$K^{\topp}:\nCcat\to \Sp$$
 for possibly non-unital $C^{*}$-categories.
 The usual construction derives  $K^{\topp}$  from the
topological $K$-theory functor for $C^{*}$-algebras 
\eqref{werfwefefwf}.
Following \cite{joachimcat} we use the adjunction
\begin{equation}\label{werfwerfwef}A^{f}:\nCcat\leftrightarrows  \nCalg:\incl\ .
\end{equation} 
\begin{ddd}\label{kopgwergwerg9} We define the $K$-theory functor for $C^{*}$-categories as the composition
  \begin{equation}\label{sdfvsdfvsdv}K^{\topp}:\nCcat\stackrel{A^{f}}{\to} \nCalg\stackrel{K^{C^{*}}}{\to} \Sp\ .
\end{equation}
\end{ddd}

 We refer to \cite{cank}, \cite{KKG} for a detailed discussion of properties of this functor.

 \begin{rem}\label{tkohprthrethrth} For the present paper the following are relevant:
\begin{enumerate}
\item $K^{\topp}$ extends $K^{C^{*}}$ in the following sense: For a $C^{*}$-algebra $A$ we have \begin{equation}\label{gwerfwerfw}
 K^{\topp}(\incl(A))\simeq K^{C^{*}}(A)\ .\end{equation}
\item $K^{\topp}$ sends unitary equivalences to equivalences.
\item $K^{\topp}$ sends exact sequences to fibre sequences.
\item $K^{\topp}$ annihilates  flasques.
\item  $K^{\topp}$ preserves filtered colimits.
\item $K^{\topp}$ ist stable: It sends the stabilization morphisms
$\bC\to K\otimes_{\max} \bC$ to equivalences.
\item $K^{\topp}$ is homotopy invariant: It sends the morphisms
$\bC\to \bC[0,1]:=C([0,1])\otimes_{\max} \bC$ to equivalences. 
\item $K^{\topp}$ sends Morita and weak Morita equivalences to equivalences.
\item $K^{\topp}$  sends infinite products of additive $C^{*}$-categories to products.
\item $K^{\topp}$ has  a lax symmetric monoial refinement.
\end{enumerate}  
\hB \end{rem}
The only properties in the list in \cref{tkohprthrethrth} which can be seen  directly from the definition \eqref{sdfvsdfvsdv} are that $K^{\topp}$ preserves filtered colimits and admits a lax symmetric monoidal refinement.
 In order to see the other properties one uses a different construction based on the functor 
 \begin{equation}\label{wefwvefsvvdfvs}A:c\nCcat\to \nCalg
\end{equation}
 which sends a $C^{*}$-category to the completion $A(\bC)$  of the  pre-$C^{*}$-algebra with underlying vector space 
 \begin{equation}\label{wregwfrewr}A^{0}(\bC):=\bigoplus_{C,D\in \bC}\Hom_{\bC}(C,D)\ ,
\end{equation} and with matrix multiplication and involution
 induced by the composition and the involution of $\bC$.   Here 
 $c\nCcat$ denotes the wide subcategory of $\nCcat$ of functors which are injective on objects.
 
 The unit $\bC\to A^{f}(\bC)$ is initial for functors from $\bC$ to $C^{*}$-algebras.
 The canonical functor $\bC\to A(\bC)$ therefore  induces a natual map $\alpha_{\bC}:A^{f}(\bC)\to A(\bC)$.
 It has been shown in  \cite{joachimcat} (in the unital case) and \cite{KKG} in general that
 $$K^{C^{*}}(\alpha_{\bC}):K^{C^{*}}(A^{f}(\bC))\to K^{C^{*}}(A(\bC))$$ is an equivalence for all $\bC$ in $\nCcat$.
  The restriction of 
 $K^{\topp}$ to $c\nCcat$ is thus given by the formula
 \begin{equation}\label{gwergefwefwerfweff}K^{\topp}_{|c\nCcat}\simeq K^{C^{*}}A :c\nCcat\to \Sp\ .
\end{equation} 
Using this formula one can check the  remaining properties 
 of $K^{\topp}$ listed in \cref{tkohprthrethrth}, see  \cite{cank}.

 The following constructions are inspired by \cite{Land_2017}.
   As shown in  \cite{zbMATH06107958} (see also \cite{startcats}) the category $\Ccat$  of unital $C^{*}$-categories 
  has a simplicial combinatorial model category structure whose weak equivalences are unitary equivalences, and whose
  category of cofibrations is the subcategory $c\Ccat$ of functors which are injective on objects. In particular all objects are cofibrant.
  Recall that the subscript $(-)_{2,1}$ denotes the  Dwyer-Kan localization at unitary equivalences.
The notation indicates that  $\Ccat_{2,1}$ is represented by the $(2,1)$-category of unital $C^{*}$-categories, functors and unitary natural isomorphisms.  
 By \cite[Thm. 1.3.4.20]{HA}  the inclusion induces an equivalence 
$c\Ccat_{2,1}\stackrel{\simeq}{\to} \Ccat_{2,1}$.
Since $K^{C^{*}} A_{|c\Ccat}:c\Ccat \to \Sp$  and $K^{\topp}_{|\Ccat}:\Ccat\to \Sp$ send unitary equivalences to equivalences they factorize over the localization.
We get a commutative diagram \begin{equation}\label{werfwerfwer}\xymatrix{c\Ccat_{2,1}\ar[rr]^-{(K^{C^{*}} A_{|c\Ccat})_{2,1}}\ar[d]_{\simeq} && \Sp\\   \Ccat_{2,1}&& \Ccat\ar[ll]^{\ell}\ar[u]_{K^{\topp}_{|\Ccat}}&}
\end{equation}
expressing  $K^{\topp}_{|\Ccat}$ in terms of $K^{C^{*}}A_{|c\Ccat}$.

   \subsection{Algebraic $K$-theory of $\Z$-linear categories  via $K$-theory of rings} 

In this section we derive a formula the algebraic $K$-theory of $\Z$-linear categories which is similar to \eqref{werfwerfwer}. It will be used to construct the comparison map from the algebraic $K$-theory of $C^{*}$-categories considered as $\Z$-linear categories to their topological $K$-theory.

 In analogy with \eqref{werfwerfwef} the inclusion of unital rings  into $\Z$-linear categories has a left-adjoint:
\begin{equation}\label{rrwefrefwerfwef}A^{f}_{\Z}:\Cat_{\Z}\leftrightarrows \Ring:\incl\ .
\end{equation}


\begin{rem} In view of \eqref{sdfvsdfvsdv} the first try could be to use the left adjoint in \eqref{rrwefrefwerfwef} and 
 $$\xymatrix{\Add\ar[r]^{\cS}\ar@{..>}[dr]&\Cat_{\Z}\ar[r]^{A_{\Z}^{f}}&\Ring \ar[dl]^{K^{Ring}}\\&\Sp&}\ .$$
 It is not clear that the resulting 
 functor $\Add\to \Sp$ indicated by the dotted arrow is equivalent to $K^{\alg}$. 
 \footnote{As observed by A. Engel (private communication) this reconstruction would work if one replaces $K$-theory by homotopy $K$-theory.} \hB \end{rem}

  As shown e.g. in \cite{zbMATH07194060} the categories $\Cat_{\Z}$ and $\Add$ admit combinatorial simplicial model category structures whose weak equivalences are the 
equivalences of  categories, and whose cofibrations are functors which are injective on objects.
We let $c\Add$ and $c\Cat_{\Z}$ denote the wide subcategories of cofibrations. Recall that the subscript $(-)_{2,1}$ denotes Dwyer-Kan localization at equivalences. As the notation suggests these are again represented by the $(2,1)$-categories of additive or $\Z$-linear categories, $\Z$-linear functors and natural isomorphisms.
By \cite[Thm. 1.3.4.20]{HA}     the inclusions induce equivalences
$$c\Add_{2,1}\stackrel{\simeq}{\to} \Add_{2,1} \ , \qquad c\Cat_{\Z } \stackrel{\simeq}{\to} \Cat_{\Z,2,1}\ .$$  In analogy to  \eqref{wefwvefsvvdfvs} we have a functor
\begin{equation}\label{vfvdsvsdfvdfvsv}A_{\Z}:c\Cat_{\Z}\to  \nRing
\end{equation} 
  which sends  a $\Z$-linear category $\bC$ to the matrix ring with underlying group 
  \begin{equation}\label{ewdqewdwedqfef}A_{\Z}(\bC)=\bigoplus_{C,D\in \bC}\Hom_{\bC}(C,D) \ .
\end{equation}
We consider the composition
$$K^{Ring}(A_{\Z}(-)):  c\Cat_{\Z}\xrightarrow{A_{\Z}} \nRing\xrightarrow{K^{Ring}}\Sp\ .$$
 %

\begin{lem}\label{lhprertherth9}
If $\bA\to \bB$ is a weak equivalence in $c\Cat_{\Z}$, then
$K^{Ring}(A_{\Z}(\bA))\to K^{Ring}(A_{\Z}(\bB))$ is an equivalence.
\end{lem} \begin{proof}  
We use that $K^{Ring}$ is matrix stable and preserves filtered colimits.
The category $\bA$ is a full subcategory of $\bB$ and every object in $
\bB\setminus \bA$ is isomorphic to an object of $\bA$.
Let us first assume that $\bA$ and $\bB$ have finitely many objects.
Then we can construct a square
$$\xymatrix{\bA\ar[r]\ar[d]^{\simeq} & \bB\ar[d]^{\simeq} \\\bA_{\infty} \ar[r]^{\cong} & \bB_{\infty}} $$
where
$\bA_{\infty}$ and $\bB_{\infty}$ are obtained by repeating
every object of $\bA$ or $\bB$ countably many times.
The vertical maps are equivalences in $c\Add$
and
the lower  horizontal map is an isomorphism.
We now apply the functor $A_{\Z}$ and get, using isomorphisms of the form 
$A_{\Z}(\bA_{\infty})\cong M_{\infty}\otimes A_{\Z}(\bA)$,
$$\xymatrix{A_{\Z}(\bA)\ar[r]\ar[d]^{corner} & A_{\Z}(\bB)\ar[d]^{corner}  \\ M_{\infty}\otimes A_{\Z}(\bA) \ar[r]^{\cong} & M_{\infty}\otimes A_{\Z}(\bB)} $$
The vertical maps are corner conclusions.
By matrix stability of $K^{Ring}$
we conclude the isomorphisms in the commutative square
$$\xymatrix{K^{Ring}(A_{\Z}(\bA))\ar[r]\ar[d]^{\simeq} & K^{Ring}(A_{\Z}(\bB))\ar[d]^{\simeq}  \\ K^{Ring}(M_{\infty}\otimes A_{\Z}(\bA)) \ar[r]^{\simeq} &K^{Ring}(M_{\infty}\otimes A_{\Z}(\bB))} \ .$$
Thus the assertion is true for categories with finitely many objects.

For arbitrary categories we use that
$K^{Ring}(A_{\Z}(-))$ preserves filtered colimits and that every $\Z$-linear category is the filtered colimit of its subcategories with finitely many objects.
\end{proof}

By \cref{lhprertherth9}  and the universal property of $\ell$       the functor 
$K^{Ring}(A_{\Z}(-))$ has a  unique factorization 
$$\xymatrix{\Cat_{\Z} \ar[rr]^{K^{Ring}(A_{\Z}(-))}\ar[dr]_{\ell}&& \Sp\\&\Cat_{\Z,2,1} \ar@{..>}[ur]_{K^{Ring}(A_{\Z}(-))_{2,1}}&}$$
Inspired by \cite{Land_2017}  we can now define the functors
$$K^{\alg,\prime}:\Add\to \Sp\ , \quad      K^{\Cat_{\Z},\prime}:\Cat_{\Z}\to \Sp$$ using the diagram: 
\begin{equation}\label{werfewrfeewerfwerfwerf}\xymatrix{c\Add_{2,1}\ar[r]^{\cS_{2,1}}\ar[d]^{\simeq}&c\Cat_{\Z,2,1}\ar[rr]^-{K^{Ring}(A_{\Z}(-))_{2,1}}\ar[d]^{\simeq}&&\Sp\ar@{=}[d]\\\Add_{2,1} \ar[r]^{\cS_{2,1}} &\Cat_{\Z,2,1} \ar[rr]^{K^{\Cat_{\Z,2,1},\prime}}&&\Sp \ar@{=}[d]\\ \Add \ar@{..>}@/_1cm/[rrr]_{K^{\alg,\prime}}\ar[r]^{\cS}\ar[u]^{\ell}&\Cat_{\Z}\ar@{..>}[rr]^{K^{\Cat_{\Z},\prime}}\ar[u]^{\ell}&&\Sp}\ .
\end{equation} 
    The upper vertical equivalences are   explained above, and the maps
  denoted by $\ell$ are the localizations.

The following lemma shows that the construction above is really a reconstruction of the algebraic $K$-theory for $\Z$-linear categories from $K^{Ring}$.
\begin{lem}\label{lpetgertgergert} We have an equivalence
$K^{\Cat_{\Z}}\simeq K^{\Cat_{\Z},\prime}:\Cat_{\Z}\to \Sp$.
\end{lem}
\begin{proof}
Let $\bA$ be in $\Cat_{\Z}$.
Since $A_{\Z}(\bA)$ is non-unital (except if $\bA$ has finitely many objects)
we must use  \eqref{gwreferfwerf}.  Unfolding  definitions  thus have $$K^{\Cat_{\Z},\prime}
(\bA)\simeq \Fib(  K^{\alg} (\Mod^{\fg,\proj}(A_{\Z}(\bA)^{+}))\to  K^{\alg}(\Mod^{\fg,\proj}(\Z)))\ .$$ We have a functor
$$\bA\to \Mod^{\fg,\proj}(A_{\Z}(\bA)^{+})$$
which sends the object $A$ in $\bA$ to the submodule of $A_{\Z}(\bA)$ given by
the column with index $A$. This module is generated by the matrix $1_{A}\in A_{\Z}(\bA)$.
By the Yoneda Lemma this functor is fully faithful.
 The composition $\bA\to    \Mod^{\fg,\proj}(A_{\Z}(\bA)^{+})\to  \Mod^{\fg,\proj}(\Z)$ vanishes. 
We therefore get a natural  map
$K^{\Cat_{\Z}}(\bA)\to K^{\Cat_{\Z},\prime}(\bA)$.

We claim that  the natural transformation $$K^{\Cat_{\Z}}(-)\to K^{\Cat_{\Z},\prime}(-):\Cat_{\Z}\to \Sp$$ just defined is an equivalence. We use  that both
functors preserve filtered colimits.
Since every $\Z$-linear category is the colimit of its full subcategories with finitely many objects
it suffices to show the assertion under the hypothesis that $\bA$ has finitely many objects.
Then $A_{\Z}(\bA)$ is unital and $K^{\alg,\prime}(\bA)\simeq K^{\alg}(\Mod^{\fg,\proj}(A_{\Z}(\bA))$.

Since $ \Mod^{\fg,\proj}(A_{\Z}(\bA))$ is additive  the functor $\bA\to  \Mod^{\fg,\proj}(A_{\Z}(\bA))$ described above extends to $ \bA_{\oplus}\to  \Mod^{\fg,\proj}(A_{\Z}(\bA))$ which is still fully faithful.
Any  summand of  $ A_{\Z}(\bA)^{n}$ corresponds to a projection on some object of $\bA_{\oplus}$.
It follows that 
$$\bA_{\oplus}\to  \Mod^{\fg,\proj}(A_{\Z}(\bA))$$ is  a Morita equivalence.
Since $K^{\alg}$ sends Morita equivalences to equivalences, the induced map
$$K^{\Cat_{\Z}}(\bA) \stackrel{\cref{oophkerpthetreg}}{\simeq} K^{\alg}(\bA_{\oplus})\stackrel{\simeq}{\to} K^{\alg}( \Mod^{\fg,\proj}(A_{\Z}(\bA)))\simeq K^{\Cat_{\Z},\prime}(\bA)$$ is an equivalence as desired.
\end{proof}

\subsection{The comparison map}\label{kogpwerfwerfwf}

Analogously to   \eqref{vfsdvsfdvsvvewe} we have a "forget the topology, involution and $\C$-linear structure" functor
  \begin{equation}\label{werfwrefw424r2}\cT:\Ccat\to  \Cat_{\Z}
\end{equation}
 which we use to  evaluate functors defined on $\Z$-linear categories 
 on $C^{*}$-categories.
The main goal of this  subsection  
is the construction of  the comparison map \begin{equation}\label{qewdqwdwedqd} c:K^{\Cat_{\Z}} \cT\to K^{\topp}_{|\Ccat}:\Ccat\to \Sp
\end{equation} 
    from the algebraic $K$-theory of a $C^{*}$-category to its topological $K$-theory.
 
  The completion maps  (the inclusions  $A^{0}(\bC)\to A(\bC) $ of the pre-$C^{*}$-algebras from  \eqref{wregwfrewr} into their  completions $A(\bC)$ for all $C^{*}$-categories $\bC$)  give a natural transformation
  \begin{equation}\label{gwergerfwref}A_{\Z}\cT\to \cT A:c\Ccat\to \nRing\ ,
\end{equation} 
  where $A_{\Z}$ is as in \eqref{vfvdsvsdfvdfvsv} and
  $A$ is the functor from \eqref{wefwvefsvvdfvs}.
  We get the transformation
 \begin{equation}\label{fferfrfwefferfw}K^{Ring}A_{\Z}\cT\xrightarrow{\eqref{gwergerfwref}} K^{Ring}\cT A\stackrel{\eqref{t432r2f34r234}}{\to} K^{C^{*}}A
 \end{equation}
 of functors from  
$c\Ccat$ to $\Sp$.
The functor $\cT$ in \eqref{werfwrefw424r2} sends unitary equivalences to equivalences and therefore
has a factorization $\cT_{2,1}$ over the localizations.
We have the following   diagram
\begin{equation}\label{fweqdqewdqwedqwedqwed}\xymatrix{ &&&\\
c\Ccat_{1,2}\ar[r]^{\cT_{2,1}}\ar@/^1.5cm/[rrr]^{(K^{C^{*}}A_{|c\Ccat})_{2,1}} \ar[d]^{\simeq}&c\Cat_{\Z,2,1}\ar[rr]^{K^{Ring}(A_{\Z}(-))_{2,1}}\ar[d]^{\simeq}\ar@{=>}[u]^{\eqref{fferfrfwefferfw}}&&\Sp\ar@{=}[dd]\\
\Ccat_{1,2}\ar[r]^{\cT_{2,1}}&\Cat_{\Z,2,1}&&\\\Ccat\ar[u]^{\ell}\ar[r]^{\cT}\ar@/^-1.5cm/[rrr]_{K^{\topp}_{|\Ccat}}&\ar@{==>}[d]^{c}\Cat_{\Z}\ar[u]^{\ell}\ar[rr]^{K^{\Cat_{\Z}}}&&\Sp\\&&&}\ .
\end{equation}
  The right square   commues by  \eqref{werfewrfeewerfwerfwerf} and \cref{lpetgertgergert}.
  The upper triangle is filled by  the factorization of \eqref{fferfrfwefferfw} over the localizations (a natural transformation, not necesarily an equivalence). Finally, 
  the outer region is filled by an equivalence \eqref{werfwerfwer}.
  We therefore get a transformation
  $c$ in \eqref{qewdqwdwedqd} filling the lower triangle as a composition of cells of this diagram.

  \subsection{Homotopy $K$-theory for $\Z$-linear and additive categories}\label{gwergrefwfref} 
  
  In this subsection we recall the construction of the homotopy $K$-theory  functor
   \begin{equation}\label{fqijioqjwodkmqow}K^{\Cat_{\Z}}H:\Cat_{\Z}\to \Sp
\end{equation}
 for $\Z$-linear categories.
 We then define  the homotopy $K$-theory functor \begin{equation}\label{wergoowerfwefewfwe}K^{\alg}H:= K^{\Cat_{\Z}} H\cS:\Add\to \Sp
\end{equation}
  
for additive categories, where $\cS$ is the forgetful functor from \eqref{goiugjowergwergw}.
The constructions are designed such that
 \begin{equation}\label{erfwefwerfwgw}
   K^{\Cat_{\Z}}H\circ \incl:\Ring \to \Sp\ ,\end{equation} 
 is the restriction of Weibel's   homotopy $K$-theory functor $K^{Ring}H:\nRing \to \Sp$ (see \cite{zbMATH04095731}) to unital rings.

Our construction    follows the approach of \cite{stlue}.  
 Any unital ring $R$  
 gives naturally rise to an endofunctor $$\Cat_{\Z}\to \Cat_{\Z}\ , \quad \bA\mapsto R\otimes \bA\ .$$ 
 The unit of $R$ induces a canonical inclusion
$ \bA\to R\otimes \bA$. 
 A functor $E:\Cat_{\Z}\to \cC$ is called homotopy invariant if $E( \bA)\to E(\Z[t]\otimes\bA)$ is an 
equivalence for all $\bA$ in $\Cat_{\Z}$. We let $\Fun^{h}(\Cat_{\Z},\cC)$ denote the full subcategory of $\Fun(\Cat_{\Z},\cC)$ of homotopy invariant functors.

 If  $\cC$ is admits sifted colimits, then
we have a left Bousfield localization \begin{equation}\label{tkohpwegwregewrg}(E\mapsto  EH): \Fun(\Cat_{\Z},\cC)\leftrightarrows \Fun^{h}(\Cat_{\Z},\cC):\incl\ .
\end{equation}
If $\cC$ is presentable, then the  assertion simply follows from the fact that  $\Fun^{h}(\Cat_{\Z},\cC)$ is closed under limits.
If $\cC$  just admits sifted colimits, then
we construct, following \cite{stlue}, an explicit model of $EH$ by 
\begin{equation}\label{oiewrfoiewrfwefwef}EH(\bA):=\colim_{[n]\in \Delta^{\op}} E(\Delta_{\alg}^{n}\otimes \bA)  \ . \end{equation}  
Here $([n]\to \Delta_{\alg}^{n}):\Delta^{\op}\to  \Ring$ denotes the usual functor of polynomial simplices.
 The unit of the adjunction provides a natural transformation
 \begin{equation} E\to EH:\Cat_{\Z}\to \cC\ .
\end{equation}
 
 Using the model for $EH$ we can see that if $E $ has one of the properties listed below, then $EH$ has this property too.
\begin{enumerate}
\item It sends equivalences to equivalences.
\item It sends Morita equivalences to equivalences.
\item It preserves filtered colimits.
\end{enumerate}
For the restriction to additive categories we get in addition the preservation of the following properties:
\begin{enumerate}
\item It sends Karoubi filtrations to fibre sequences. 
\item  It annihilates flasque objects.
\end{enumerate}
All this follows since $R\otimes -$ preserves the corresponding class of functors, objects or colimits, and the properties pass to the colimit.

If we apply this construction to the functor $K^{\Cat_{\Z}}$, then we
 get the functor in \eqref{fqijioqjwodkmqow} together with a natural transformation
 \begin{equation}\label{vdfsvewrvfsdvsdfvsd}h: K^{\Cat_{\Z}}\to K^{\Cat_{\Z}}H:\Cat_{\Z}\to \Sp\ .
\end{equation} 
If we replace in this construction $\Cat_{\Z}$ by $\Add$, then we obtain the construction of the homotopy $K$-theory
functor   in \cite{stlue} which turns out to be equivalent to the functor defined in \eqref{wergoowerfwefewfwe}.
Using \eqref{wergoowerfwefewfwe}, \cref{okgpwerewfwerfw} and \cref{okgpwerewfwerfw1} we get: 
\begin{kor}\label{kopgwerfewfwefw}
The functor $K^{\alg}H:\Add\to \Sp$ is homological in the sense of \cref{kopwegrfw}. \end{kor}
%
%
%
%

  \subsection{Factorization of the comparison map over homotopy $K$-theory}
 
The goal of this subsection is to construct a factorization
\begin{equation}\label{fwrefwerfwefer}\xymatrix{K^{\Cat_{\Z}}\cT\ar[rr]^{c}_{\eqref{qewdqwdwedqd}}\ar[dr]^{h\cT}_{\eqref{vdfsvewrvfsdvsdfvsd}}&&K^{\topp}_{|\Ccat}\\&K^{\Cat_{\Z}}H\cT\ar@{..>}[ur]_{cH}&}
\end{equation}
of natural transformations between functors from $\Ccat$ to $\Sp$.

%
%
%
%

By definition the transformation $h$ from  \eqref{vdfsvewrvfsdvsdfvsd}  is initial for transformations from $K^{\Cat_{\Z}}$ to polynomially homotopy invariant functors defined on all of $\Cat_{\Z}$. 
But it is not clear that $h\cT$ in the diagram \eqref{fwrefwerfwefer} has a universal property.
We will therefore use the explicit model \eqref{oiewrfoiewrfwefwef}.


 The following could be developed for $\nCcat$ but we will state the constructions for $\Ccat$ since this case is needed for our purpose.
 
 For any unital  $C^{*}$-algebra $A$ we have an endofunctor 
   \begin{equation}\label{werfwefwer}A\otimes -:\Ccat\to \Ccat\ , \quad \bC \mapsto A\otimes \bC
\end{equation}
   involving the maximal tensor product of $C^{*}$-categories \cite[Def. 7.2]{KKG}, where 
   we view $A$ as a $C^{*}$-category with a single object. The unit of $A$ induces a natural transformation
   $\id\to   A\otimes - $ of endofunctors of $\Ccat$.

 A functor $E:\Ccat\to \cC$ is called homotopy invariant if the induced map $E(\bC)\to E(C([0,1])\otimes \bC)$ is an equivalence for all $\bC$ in $\Ccat$. We let $\Fun^{h}(\Ccat,\cC)$ denote the full subcategory of $\Fun(\Ccat,\cC)$ of  homotopy invariant functors.  In analogy to \eqref{tkohpwegwregewrg} (with essentially the same proof) we get:
 If $\cC$ admits sifted colimits, then 
 there is a left Bousfield localization
$$(E\mapsto E\cH):\Fun( \Ccat,\cC)\leftrightarrows \Fun^{h}( \Ccat ,\cC):\incl\ .$$  
A model for $E\cH$ is given in analogy to \eqref{oiewrfoiewrfwefwef} by 
\begin{equation}\label{qwefwqedqwed}E\cH(\bC) :=\colim_{[n]\in \Delta^{\op}} E(C(\Delta^{n})\otimes \bC)\ .
\end{equation} 
The inclusions $\Delta_{\alg}^{n}\to C(\Delta^{n})$ of integral polynomial functions on the simplex into continuous complex-valued functions on the topological simplex induces a natural transformation
$$ \Delta^{n}_{\alg}\otimes \cT\bC\to \cT(C(\Delta^{n})\otimes \bC)$$ of functors from $\Delta^{\op}$ to $\Cat_{\Z}$, where  $\cT$ is as in 
\eqref{werfwrefw424r2}.  
Applying a functor 
  $E:\Cat_{\Z}\to \cC$ to this family 
 and taking the colimit over $\Delta^{op}$   we get the natural transformation of functors which send
$E:\Cat_{\Z}\to \cC$ to the transformation \begin{equation}\label{fwerfwerfwer}t_{E}:EH\cT\to E\cT\cH: \Cat_{\Z}\to \Sp
\end{equation} Since $K^{\topp}_{|\Ccat}$ is homotopy invariant the unit map
$K^{\topp}_{|\Ccat}\to K^{\topp}_{|\Ccat}\cH$ is an equivalence.
We have the following commutative diagram
$$\xymatrix{&K^{\topp}_{|\Ccat}\ar[r]^{\simeq}&K^{\topp}_{|\Ccat}\cH \\
K^{\Cat_{\Z}}\cT\ar[r]\ar[d]_{h\cT,\eqref{vdfsvewrvfsdvsdfvsd}}\ar[ur]^{c, \eqref{qewdqwdwedqd}} &K^{\Cat_{\Z}}\cT\cH\ar[ur]_{c\cH}&\\K^{\Cat_{\Z}}H\cT\ar[ur]_{ t_{K^{\Cat_{\Z}}},\eqref{fwerfwerfwer}}\ar@{..>}[uur]^(0.55){cH}&&}$$
of natural transformations between functors from $\Ccat$ to $\Sp$.
The dotted arrow
provides the desired factorization in \eqref{fwrefwerfwefer}.

\subsection{The twist by trace class operators}

The goal of this subsection is finally  to construct the  natural transformation
\begin{equation}\label{fqweiqwdiew9qdi09qwedq}
c^{\cL^{1}}:(K^{\Cat_{\Z}}H\cZ)_{\cL^{1}} \cT_{\C}\to K^{\topp}_{|\Ccat}:\Ccat\to \Sp\ ,
\end{equation}
announced in \eqref{gwegerfwr}.
We further show that its evaluations at $\C$ considered as a $C^{*}$-category or the $C^{*}$-category $\Hilb^{\fin}(\C)$ of finite-dimensional Hilbert spaces are an equivalences.


Recall that a unital $\C$-algebra $R$   defines an endofunctor $R\otimes_{\C}-:\Cat_{\C}\to \Cat_{\C}$.
Given a possibly non-unital $\C$-algebra $R$   in $\nAlg_{\C}$ and a functor $E:\Cat_{\C}\to\cC$ with a stable target
 we  define the twist  \begin{equation}\label{rewfewrfvdsfv}E_{R}:\Cat_{\C}\to \cC \ ,\quad \bA\mapsto \Fib(E(R^{+}\otimes_{\C} \bA)\to E(\bA))\ ,
\end{equation}
  where the map is induced by the canonical projection $R^{+}\to \C$ from the unitalization $R^{+}$ in $\Alg_{\C}$ to the ground ring $\C$.

Given a possibly non-unital $C^{*}$-algebra $A$   in $\nCalg$ and a functor $E:\Ccat\to\cC$ with a stable target, then using \eqref{werfwefwer}
 we  similarly define the twist 
  \begin{equation}\label{rewfewrfvdsfv1} E_{A}:\Ccat\to \cC \ ,\quad \bC\mapsto \Fib(E(A^{+}\otimes \bC)\to E(\bC))\ .\end{equation}

These constructions are functorial in $R$ or $A$ in $\nAlg_{\C}$ or $\nCalg$, respectively. Furthermore,
if the original functor $E$ is Morita invariant, then the twisted functor
$E_{R}$ or $E_{A}$, respectively, is Morita invariant, too. Here we use that the functors
$R\otimes_{\C}- $ or $A\otimes-$ for unital $R$ or $A$ preserve Morita equivalences.

 In the following we let  $\cT_{\C}:\Ccat\to \Cat_{\C}$ or $\cT_{\C}:\nCalg\to \nAlg_{\C}$ denote the forgetful functors.
The completion map from the algebraic to the maximal tensor product induces a natural transformation
of bi-functors
$$\cT_{\C}(-)\otimes^{\alg}_{\C} \cT_{\C}(-)\to \cT_{\C}( -\otimes^{C^{*}} -):\Calg\times \Ccat\to \Cat_{\C}\ ,$$
where for clarity we indicated the context of the tensor products by superscripts.
Combined with the constructions from \eqref{rewfewrfvdsfv1}
and \eqref{rewfewrfvdsfv} for any functor $E:\Cat_{\C} \to \cC$ we get a natural transformation of bi-functors
\begin{equation}\label{hrtegtrgeg}E_{\cT_{\C}(-)}\cT_{\C}(-)\to (E\cT_{\C})_{-}(-):\nCalg
\times \Ccat\to \cC\ .
\end{equation} 
 Let $\cZ:\Cat_{\C}\to \Cat_{\Z}$ denote the forgetful functor.
Recall that $\cL^{1}\subseteq K$ denotes the ideal of trace class operators on $\ell^{2}$.
We consider $\cL^{1}$ as an object of $\nAlg_{\C}$. 
The inclusion $\cL^{1}\to \cT_{\C}(K)$ induces a natural transformation
\begin{equation}\label{vfsfsvfdv1}(K^{\Cat_{\Z}}H\cZ)_{\cL^{1}}\cT_{\C}\to (K^{\Cat_{\Z}}H\cZ)_{\cT_{\C}(K)}\cT_{\C}:\Ccat\to \Sp\ .
\end{equation}  We appologize for the accumulation of forgetful functors, but they are important
in order to clarify the which tensor product (amongst $\otimes_{\Z}$, $\otimes_{\C}$ or $\otimes_{\max}$) is involved in the definition of  the terms.
 Specializing  \eqref{hrtegtrgeg} to $E=K^{\Cat_{\Z}}H\cZ$    and using the equivalence $\cT\simeq \cZ\cT_{\C}$ we get a transformation
 \begin{equation}\label{vfsfsvfdv2}
 (K^{\Cat_{\Z}}H\cZ)_{\cT_{\C}(K)}\cT_{\C}\to (K^{\Cat_{\Z}}H\cT)_{K}:\Ccat\to \Sp\ .
\end{equation} 
 The diagram \eqref{fwrefwerfwefer}   gives a transformation 
\begin{equation}\label{vfsfsvfdv3}
(cH)_{K}: (K^{\Cat_{\Z}}H\cT)_{K}\to  (K^{\topp}_{|\Ccat})_{K}
\end{equation} 
By $K$-stability of $K^{\topp}$  we have an equivalence
 \begin{equation}\label{vfsfsvfdv4}K^{\topp}_{|\Ccat}\simeq (K^{\topp}_{|\Ccat})_{\C}\xrightarrow{\simeq,\eqref{wfgerfwerfw}} (K^{\topp}_{|\Ccat})_{K}\ .
\end{equation} 
\begin{ddd}\label{hlrztrhertgetrg}
We define   the comparison map 
\eqref{fqweiqwdiew9qdi09qwedq} as the composition 
$$c^{\cL^{1}}:(K^{\Cat_{\Z}}H\cZ)_{ \cL^{1}}\cT_{\C}\xrightarrow{\eqref{vfsfsvfdv4}^{-1}\circ \eqref{vfsfsvfdv3}\circ \eqref{vfsfsvfdv2}  \circ \eqref{vfsfsvfdv1}} K^{\topp}_{|\Ccat}: \Ccat\to \Sp\ .
$$
\end{ddd}

We consider the  one-object $C^{*}$-category $\incl(\C)$  in $\Ccat$, see \eqref{werfwerfwef} for notation.
\begin{prop}
The component
$$c^{\cL^{1}}_{\incl(\C)}: (K^{\Cat_{\Z}}H\cZ_{\cL^{1}}) ( \cT_{\C}(\incl(\C)))\to 
K^{\topp}(\incl(\C))$$ 
is an equivalence.
\end{prop}
\begin{proof}
We use the expression of the algebraic $K$-theory of $\Z$-linear categories or $C^{*}$-algebras in terms of the 
$K$-theory of rings or $C^{*}$-algebras using the functors  $A_{\Z}$ or $A$, respectively. 
Using  \cref{lpetgertgergert},  \eqref{erfwefwerfwgw},  the comparison of \eqref{rewfewrfvdsfv} with \eqref{gwreferfwerf} and $A_{\Z}(\incl(R))\cong R$ for any ring $R$ in $\Ring$
we get an equivalence
$$(K^{\Cat_{\Z}}H\cZ_{\cL^{1}})(\cL^{1}\otimes_{\C}\cT_{ \C}(\incl(\C))\simeq K^{Ring}H \cZ(\cL^{1})\ .$$
Similarly, using $A(\C)\cong \C$  we get 
$$K^{\topp}(\incl(\C))\stackrel{\simeq}{\to} K^{C^{*}}(\C) \ .$$ 
Unfolding the definitions, under these identifications
the map $c^{\cL^{1}}_{\C}$ becomes the composition 
$$K^{Ring}H \cZ(\cL^{1})\stackrel{!}{\to} K^{Ring}H \cT(K)  \stackrel{\eqref{rfwrefwfwf}^{-1}_{\C}}{\simeq}
K^{C^{*}}(\C)\ .$$
The marked map is an equivalence e.g. by 
\cite[Thm. 6.5.3]{Corti_as_2008}.
 \end{proof}

Using that the domain and target of $c^{\cL^{1}}$ are Morita invariant functors and that
$\incl(\C)\to \Hilb^{\fin}(\C)$ is a Morita equivalence we conclude:
\begin{kor}
The evaluation
$$c^{\cL^{1}}_{\Hilb^{\fin}(\C)}: (K^{\Cat_{\Z}}H\cZ)_{\cL^{1}} \cT_{\C}(\Hilb^{\fin}(\C))\to 
K^{\topp}(\Hilb^{\fin}(\C))$$
is an equivalence.
\end{kor}

\subsection{Traces}

Recall that $\Cat^{\tr}_{\C}$ denotes the category of pairs $(\bA,\tau)$  of a $\C$-linear  category $\bC$  and 
a $\C$-valued trace (defined on $\dom(\tau)$) and trace-preserving functors. 
Furthermore,  $\Cat^{\tr,\perf}_{\C}$ is the full subcategory of  $\Cat^{\tr}_{\C}$ of pairs where $\tau$ is defined on all objects. 
The main goal of this subsection is to construct a tracing (\cref{erthgerthertertgertt}) \begin{equation}\label{erfewrfwegrgw}\trc_{\cL^{1}}^{K^{\Cat_{\Z}H\cZ}}:\pi_{0}K^{\Cat_{\Z}}H\cZ_{\cL^{1}}  \to \const_{\C} 
\end{equation}
see \eqref{erfewrfwegrgw1}, a natural transformation of functors from
$\Cat^{\tr,\perf}_{\C}$ to $\Ab$.

We will use cyclic homology theory for $\C$-algebras. We have  a fibre sequence of functors \begin{equation}\label{vsdfvsfdvsdfvsd}HN\stackrel{I}{\to} HP\stackrel{S}{\to} \Sigma^{2} HC \end{equation} from $\nAlg_{\C}$  to $\Sp$
(see \cite[Sec. 5.1]{zbMATH01093754} for definitions)
representing negative  cyclic homology theory, periodic cyclic homology theory and cyclic homology.

\begin{rem}The periodic cyclic homology is represented by the total complex (involving products) of the bi-complex
$CC^{\per}$ living in degrees  $\Z\times \nat$.
Then $HN$ is represented by the part $CC^{-}$ in degrees $(-\infty,1]\times \nat$, and 
$HC$ by the part $CC$ in degree $[0,\infty)\times \nat$.
The map $I$ is given by the inclusion $CC^{-}\to CC^{\per}$ and
the map $S$ is induced by
the chain map $CC^{\per}\to CC$ 
given by the projection 
onto its part in degrees $[2,\infty)\times \nat$
 and a twofold left-shift (putting the column in degree $2$ into degree $0$). 
We will use   the further natural transformation \begin{equation} p:HP\to HC
\end{equation} induced by the projection  $CC^{\per}\to CC$
  \hB\end{rem}
  As the functors in \eqref{vsdfvsfdvsdfvsd} are defined in terms of chain complexes of complex vector spaces 
 their natural target    is the 
stable $\infty$-category $\bD(\C)$, the derived category of $\C$. We implicitly use the functor
$\map_{\bD(\C)}(\C,-):\bD(\C) \to  \Sp$
in order to interpret them   as spectrum valued.

We start with rings. Recall that $\cZ:\nAlg_{\C}\to \nRing$ is the forgetful functor.
The primary character is the Goodwillie-Jones character
$$\ch N: K^{Ring}\cZ  \to HN:\nAlg_{\C}\to \Sp$$
\cite{McCarthy_1994}, \cite{lrd}.

The
natural transformation
 $K^{Ring}\cZ\to K^{Ring}H\cZ:\nAlg_{\C}\to \Sp$  is the initial   natural  transformations from $ K^{Ring}\cZ$ to a polynomially homotopy invariant $\Sp$-valued functor on $\nAlg_{\C}$.
Using that $HP$ is polynomially homotopy invariant     \cite{Goodwillie_1986} \cite{Goodwillie_1985} and the universal property of the upper horizontal map  just discussed
 we get a factorization (denoted by $\ch P$)
$$\xymatrix{K^{Ring}\cZ\ar[r]\ar[d]^{\ch N} &K^{Ring}H\cZ \ar@{..>}[d]^{\ch P} \\ HN \ar[r]^{I} & HP} $$
of natural transformation between functors from $\nAlg_{\C}$ to $\Sp$
 called the Chern character (see e.g. the diagram on \cite[p.9 ]{Corti_as_2008}).

 We now extend the Chern character $\ch P$
 from $\C$-algebras to $\C$-linear  categories.
 The functor $A_{\Z}$ from \eqref{vfvdsvsdfvdfvsv} restricts to a functor 
 $$A_{\C}:c\Cat_{\C}\to \nAlg_{\C}$$ so that 
    \begin{equation}\label{sdfvsfdvw}\cZ A_{\C}\simeq A_{\Z}\cZ:c\Cat_{\C}\to \nAlg_{\Z}\ ,
\end{equation} 
where as before $\cZ$ forgets the $\C$-vector space structures.

Using the natural isomorphisms $A_{\Z}(\Delta_{\alg}^{n}\otimes \bA)\cong \Delta_{\alg}^{n}\otimes A_{\Z}(\bA)$
  we get
an equivalence \begin{equation}\label{gwerfwerfrefw}
E A_{\Z} H\simeq E H A_{\Z}\ .
\end{equation}
We define the natural transformation 
\begin{equation}\label{gwergweopfjeiorjogwerg}\ch: K^{\Cat_{\Z}}H\cZ_{|c\Cat_{\C}}\to HC A_{\C}:c\Cat_{\C}\to \Sp
\end{equation}
as the composition
\begin{eqnarray*}
 K^{\Cat_{\Z}}H\cZ&\stackrel{\cref{lpetgertgergert}}{\simeq}&
 K^{Ring} A_{\Z} H\cZ\\
 &\stackrel{\eqref{gwerfwerfrefw}}{\simeq}&K^{Ring} H A_{\Z} \cZ\\&\stackrel{\eqref{sdfvsfdvw}}{\simeq}&
 K^{Ring} H  \cZ A_{\C} \\&\stackrel{\ch P A_{\C}}{\to}&
 HP A_{\C}\\&\stackrel{p A_{\C}}{\to}& HC A_{\C}
 \end{eqnarray*}
 We must restrict to the wide subcategory $c\Cat_{\C}$ of $\Cat_{\C}$ because of the restricted functoriality of $A_{\C}
$. 

\begin{rem}
As it is not clear that $HC A_{\C}$ sends equivalences in $\Cat_{\C}$ to 
equivalences we can not use the method from \eqref{werfewrfeewerfwerfwerf} in order to extend
$HCA_{\C}$ from $c\Cat_{\C}$ to $\Cat_{\C}$. \hB
\end{rem}

If $A$ is in $\nAlg_{\C}$, then we have an isomorphism
\begin{equation}\label{dfvsdfvwrvfv}\pi_{0}HC(A)\cong \frac{A}{[A,A]}\ .
\end{equation}

\begin{rem} In this remark we justify the formula \eqref{dfvsdfvwrvfv}.
By convention the functor $HC$ is first defined on unital algebras and then extended to non-unital algebras $A$
by
$$HC(A):=\Fib(HC(A^{+})\to HC(\C))\ .$$ 
We use the the complex $C^{\lambda}(A)$ (the quotient of the zero column in $CC(A)$ by the image of the horizontal differential) in order to represent $HC(A)$. 
The split surjection  $A^{+}\to \C$ induces a split surjection of chain complexes
$C^{\lambda}(A^{+})\to C^{\lambda}(\C)$ representing the map $HC(A^{+})\to HC(\C)$ in $\bD(\cC)$.
It follows that the fibre $HC(A)$
is represented by the actual kernel.
In the following diagram we see the degree $\le
1$-part of the  cyclic homology complexes vertically with the zero's homology added in the lowest line. 
The right two vertical complexes are $C^{\lambda}(A^{+})$ and $C^{\lambda}(\C)$, while the left vertical complex is the fibre.
All tensor products are over $\C$:
 \begin{equation}\label{werfwerfwerfwerf}\xymatrix{
0\ar[r]& (A^{+}\otimes A+A\otimes  A^{+})_{\lambda}  \ar[d]\ar[r]&  (A^{+}\otimes  A^{+})_{\lambda}\ar[r]\ar[d]&(\C\otimes \C)_{\lambda}\ar[r]\ar[d]&0\\
0\ar[r]&\ar[d]\ A\ar[r]& A^{+}\ar[r]\ar[d]&\C\ar[r]\ar[d]&0\\&\frac{A}{[ A , A]}\ar[r]&\frac{ A^{+}}{[ A , A]}\ar[r]&\frac{ \C}{[\C,\C]}&}
\end{equation} 
where $(-)_{\lambda}$ indicates the quotient by the symmetric tensors, and 
where we use that
$[A^{+},A^{+}]=[A,A]$.
Of course, we have $(\C\otimes \C)_{\lambda}\cong 0$ and $[\C,\C]\cong 0$.
\hB
\end{rem}

A trace $\hat \tau:A \to \C$ annihilates the commutator $[A,A]$ and therefore factorizes over a  homomorphism
$$\hat \tau_{(A,\tau)}^{HC}:\pi_{0}HC(A) \to \C\ .$$

We consider $(\bA,\tau)$ in $\Cat^{\tr,\perf}_{\C}$. The trace $\tau$ on the $\C$-linear category $\bA$  induces a trace $\hat \tau$ on the $\C$-algebra $A_{\C}(\bA)$ in the canonical way such that
 for $f:C\to D$ in $\bA$ we have $\tau(f)=\tau(i(f))$, where  $i(f)$ in $A_{\C}(\bA)$ is the matrix with a single non-zero  
 entry  $f$ at place $(D,C)$ see \eqref{ewdqewdwedqfef}.
  We can now define
$$\tau_{(\bA,\tau)}^{HC}:=\hat \tau^{HC}_{(A_{\C}(\bA),\hat \tau^{HC}_{\hat \tau})}:\pi_{0}HC A_{\C} (\bA)\to \C\ .$$ 
In the following we use the obvious adaptation of the \cref{fdvsdfvsdffsc} of a tracing to functors only defined on $c\Cat_{\C}$.
 \begin{prop}
The map  $\trc^{HC}:(\bA,\tau)\mapsto(\tau_{(\bA,\tau)}^{HC}:\pi_{0}HC A_{\C} (\bA)\to \C)$
is a $\C$-valued tracing of the functor $HC A_{\C}:c\Cat_{\C}\to \C$.
\end{prop}
\begin{proof} The  claimed naturality statement is  observed by 
unfolding definitions. 
\end{proof}

\begin{lem}\label{gokrjpwegerwwre}
We obtain  a tracing \begin{equation}\label{wergpojkwoperfwerfwerf}
\trc^{K^{\Cat_{\Z}}H\cZ}:\pi_{0}K^{\Cat_{\Z}}H\cZ  \stackrel{\ch }{\to}\pi_{0}HCA_{\C}\stackrel{ \trc^{HC}}{\to} \const_{\C}\ .
 \end{equation} 
 of $K^{\Cat_{\Z}}H\cZ$. \end{lem}
 \begin{proof}
It is clear that the formula defines a tracing of $ K^{\Cat_{\Z}}$
on the wide subcategory $c\Cat_{\C}^{\tr,\perf}$ of $\Cat^{\tr,\perf}_{\C}$ of functors which are injective on objects.
 Naturality on the bigger category is a property.
 Let $(\bA,\tau)\to (\bA',\tau')$ be a morphism  in $\Cat^{\tr,\perf}_{\C}$. Then we can
 find a  factorization
$$\phi:(\bA,\tau)\stackrel{\alpha}{\to} (\bA'',\tau'')\stackrel{\beta}{\to} (\bA',\tau')$$
where $\alpha$ is injective on objects and $\beta$ is an equivalence witnessed by an inverse functor $\gamma$ which is also injective on objects.
The lower two triangles of the  following diagram commute by  the already known naturality of $ \tr^{K^{\Cat_{\Z}}H\cZ} $ on $c\Cat_{\C}^{\tr,\perf}$: $$\xymatrix{\ar@/^1cm/[rr]^{\phi}\pi_{0}K^{\Cat_{\Z}}H\cZ(\bA)\ar[r]^{\alpha}\ar[dr]_{\trc_{\bA}^{K^{\Cat_{\Z}H\cZ}}}&\ar[d]^{\trc^{K^{\Cat_{\Z}H\cZ}}_{\bA''}}\pi_{0}K^{\Cat_{\Z}}H\cZ(\bA'')&\ar[l]_{\cong}^{\gamma}\pi_{0}K^{\Cat_{\Z}}H\cZ(\bA')\ar[dl]^{\trc^{K^{\Cat_{\Z}H\cZ}}_{\bA'}}\\&k&}$$
Commutativity of the upper triangle shows the desired naturality for $\phi$.
  \end{proof}

 Recall  the twisting construction   \cref{kgopweewrfwerf}.
%
%
%
%
%
 A trace-preserving homomorphism $f:(R,\kappa)\to (R',\kappa')$ 
 induces a natural transformation of split-exact (we need this property in order to incorporate non-unital twists) traced functors
 $H_{R}\to H_{R'}$ such that
 $$\xymatrix{\pi_{0}H_{R}\ar[rr]\ar[dr]_{\trc_{R}}&&\pi_{0}H_{R'}\ar[dl]^{\trc_{R'}}\\&\C&}$$
 commutes. 
%
%

 The algebra $\cL^{1}$ of trace class operators is a $\C$-algebra with a trace $\kappa$.
 Note that $K^{\Cat_{\Z}}H\cZ:\Cat_{\C}\to \Sp$ is split-exact. \begin{ddd}\label{erfewrfwegrgw11} We define the   tracing \eqref{erfewrfwegrgw} 
 as the twist of \eqref{wergpojkwoperfwerfwerf} by $(\cL^{1},\kappa)$: \begin{equation}\label{erfewrfwegrgw1}\trc_{\cL^{1}}^{K^{\Cat_{\Z}H\cZ}}:\pi_{0} K^{\Cat_{\Z}}H\cZ_{\cL^{1}} \to \const_{\C}\ ,
\end{equation}
  \end{ddd}
 Recall that $\trc_{\cL^{1}}^{K^{\Cat_{\Z}H\cZ}}$ is natural transformation between functors from $\Cat_{\C}^{\tr,\perf}$ to $\Ab$.

  \subsection{The trace comparison theorem}
  
  We have a forgetful functor
  $$\cT_{\C}^{\tr}:\Ccat^{\tr,\perf}\to \Cat_{\C}^{\tr,\perf}$$
  which sends a pair $(\bC,\tau)$ of a $C^{*}$-category with an everywhere defined trace to the pair
  $(\cT_{\C}(\bC),\tau)$ of $\bC$ considered as a $\C$-linear category with the same trace.
  
  Recall that we have a tracing
  $$\trc^{K^{\topp}}:\pi_{0}K^{\topp} \to \const_{\C}\ ,$$
  a natural transformation of functors from $\Ccat^{\tr,\perf}$ to $\Ab$.
  
  The main goal of this subsection is to show the trace comparison theorem:
  \begin{theorem}\label{kohpetrhrtgertg}
  The following triangle of natural transformations of functors from $\Ccat^{\tr,\perf}$ to $\Ab$ commutes:
  $$\xymatrix{\pi_{0} (K^{\Cat_{\Z}}H\cZ)_{\cL^{1}}\cT_{\C} \ar[rr]^-{c^{\cL^{1}},\eqref{fqweiqwdiew9qdi09qwedq}}\ar[dr]_{\trc_{\cL^{1}}^{K^{\Cat_{\Z}H\cZ}}\cT_{\C},\eqref{erfewrfwegrgw1}}&&\pi_{0}K^{\topp} \ar[dl]^{\trc^{K^{\topp}}}\\&\C&}$$
  \end{theorem}

Recall \cref{objpbdbdgbd}.
       \begin{kor}\label{oiw0rgwrefwerfwrefw}
       The pair $(K^{\Cat_{\Z}}H\cZ)_{\cL^{1}}\cT_{\C},c^{\cL^{1}})$ is a 
    trace-preserving algebraic $\Hilb^{\fin}(\C)$-approximation of topological $K$-theory.
       \end{kor}

  \begin{proof}
  By a similar argument as in the proof of \cref{gokrjpwegerwwre} it suffices to verify the naturality on the wide subcategory $c\Ccat^{\tr,\perf}$ of $\Ccat^{\tr,\perf}$.
  Recall from \cref{kogpwerfwerfwf}  that the restriction of the comparison map $c:K^{\Cat_{\Z}}\cT \to K^{\topp}$
  to $c\Ccat$ is equivalent to a natural transformation
  $K^{Ring} A_{\Z} \cT \to K^{C^{*}}A$. In detail it is the composition 
  \begin{align*}K^{Ring} A_{\Z} \cT\stackrel{h}{\to}K^{Ring} HA_{\Z} \cT& \stackrel{\eqref{sdfvsfdvw}}{\simeq} K^{Ring} H\cZ A_{\C} \cT_{\C}\to  (K^{Ring} H\cZ)_{\cT_{\C}(K)} A_{\C} \cT_{\C}\\&\stackrel{\eqref{hrtegtrgeg}}{\to}  (K^{Ring} H \cT (K\otimes A(-))\stackrel{\eqref{rfwrefwfwf}}{\simeq} K^{C^{*}}(K\otimes A(-))\simeq K^{C^{*}}A\ ,\end{align*}
  where the two unnamed maps are induced by the left-upper corner inclusion $\C\to  K$, and for the last equivalence we use that it   induces an equivalence in $C^{*}$-algebra $K$-theory. 
Using the canonical isomorphism $A_{\C}(\cL^{1}\otimes -)\cong \cL^{1}\otimes A_{\C}(-)$ the restriction of 
  $c^{\cL^{1}}$ from \eqref{fqweiqwdiew9qdi09qwedq} to $c\Ccat$ can then be identified with the composition  %
 \begin{align*}
  (K^{Ring} H\cZ)_{\cL^{1}} A_{\C}\cT_{\C} & \stackrel{!}{\to}  (K^{Ring}H\cZ)_{\cT_{\C}(K)}A_{\C}\cT_{\C} \stackrel{!!}{\to}   (K^{Ring}H\cT) (K\otimes A -))\\ &\simeq K^{C^{*}}(K\otimes A( -))\simeq K^{C^{*}}A
  \end{align*}
where  $!$ is induced by the inclusion $\cL^{1}\to \cT_{\C}(K)$,
$!!$ is induced by  the completion map $\cT_{\C}(K)\otimes  \cT_{\C}A(-)\to \cT_{\C}(K\otimes A(-))$, and we also used that
$\id\otimes p:K\to K\otimes_{\max} K$
   induces  an equivalence in $C^{*}$-algebra $K$-theory.
    
  Write $R:=A_{\C}\cT_{\C} (\bA)$. Then the map $$(h\cZ)_{\cL^{1}} A_{\C}\cT_{\C} :(K^{Ring} \cZ )_{\cL^{1}}A_{\C}\cT_{\C} \to   (K^{Ring} H\cZ)_{\cL^{1}} A_{\C}\cT_{\C} $$ 
  is the map (written without forgetful functors)
  $$ K^{Ring}(\cL^{1}\otimes R) \to K^{Ring}H(\cL^{1}\otimes R)\ .$$
   By \cite[Thm. 6.5.3]{Corti_as_2008} it induces an isomorphism in $\pi_{0}$.
 It thus suffices to show for every pair $(A,\tau)$ of a $C^{*}$-algebra $A$ with a trace  that the following triangle commutes:
    \begin{equation}\label{sdfvsdfvs}\xymatrix{\pi_{0} K^{Ring}\cZ(\cL^{1}\otimes \cT_{\C}(A))  \ar[r] \ar[dr]_{\tau_{\cL^{1}}} &\pi_{0}K^{C^{*}}(K\otimes A)&\ar[l]_-{\cong}\pi_{0}K^{C^{*}} (A)\ar[dl]^{\tau^{K^{C^{*}}}}\\&\C&}
\end{equation}
 We extend this triangle as follows:
  $$\xymatrix{\pi_{0} K^{Ring}\cZ(\cL^{1}\otimes \cT_{\C}(A)) \ar[dr] \ar[r] \ar[ddr]_{\tau_{\cL^{1}}} & \pi_{0}K^{C^{*}}(K\otimes A)&\ar[dl]_{\cong}\ar[l]_-{\cong}\pi_{0}K^{C^{*}} (A)\ar[ddl]^{\tau^{K^{C^{*}}}}\\&\pi_{0}K^{\Ban}(\cL^{1}\otimes_{\pi} A)\ar[u]^{\cong} \ar[d]^{\tau'}&\\&\C&}$$
where $\otimes_{\pi}$ is the projective tensor product of Banach spaces and $K^{\Ban}$ is the topological $K$-theory 
  for Banach algebras. By continuity the trace $\tr\otimes \tau$ on $\cL^{1}\otimes \cT_{\C}(A)$ extends to
  a trace $\tau'
$ on the projective tensor product  $\cL^{1}\otimes_{\pi}A$
which induces the vertical morphism with the same name.
The commutativity of \eqref{sdfvsdfvs} now follows by a diagram chase. \footnote{I thank A. Thom and W. Cortinas for suggesting to use the Banach $K$-theory of $\cL^{1}\otimes_{\pi} A$ as an intermediate step on which the trace is still defined.}
 \end{proof}

\bibliographystyle{alpha}
\bibliography{forschung2021}

\end{document}